\documentclass[english]{amsart}

\usepackage{amssymb,amscd,amsthm, verbatim,amsmath,fancyhdr, mathrsfs}
\usepackage{graphicx}
\usepackage{turnstile}
\usepackage{url}
\usepackage{xcolor}
\usepackage{hyperref}
\usepackage{cleveref}
\usepackage[toc,page]{appendix}
\usepackage{tikz-cd}
\usepackage[shortlabels]{enumitem}

\setlist[enumerate,1]{label=(\roman*)}
\setlist[enumerate,2]{label=(\alph*)}

\newtheorem{theorem}{Theorem}[section]
\newtheorem{proposition}[theorem]{Proposition}
\newtheorem{lemma}[theorem]{Lemma}
\newtheorem{remark}[theorem]{Remark}
\newtheorem{corollary}[theorem]{Corollary}

\newtheorem{definition}[theorem]{Definition}
\newtheorem{conjecture}[theorem]{Conjecture}

\newtheorem{thmx}{Theorem}

\newtheorem{manualtheoreminner}{Theorem}
\newenvironment{manualtheorem}[1]{%
  \renewcommand\themanualtheoreminner{#1}%
  \manualtheoreminner
}{\endmanualtheoreminner}

\newcommand{\A}{\mathbb{A}}
\newcommand{\C}{\mathbb{C}}
\newcommand{\F}{\mathbb{F}}
\newcommand{\G}{\mathbb{G}}

\renewcommand{\P}{\mathbb{P}}
\newcommand{\Q}{\mathbb{Q}}
\newcommand{\R}{\mathbb{R}}
\newcommand{\T}{\mathbb{T}}
\newcommand{\Z}{\mathbb{Z}}

\newcommand{\Fp}{\F_p}
\newcommand{\Fpbar}{\overline{\mathbb{F}}_p}

\newcommand{\Qbar}{\overline{\Q}}
\newcommand{\Qpbar}{\overline{\Q}_p}

\newcommand{\bff}{\mathbf{f}}

\newcommand{\AF}{\A_F}
\newcommand{\Af}{\A_{\mathbf{f}}}
\newcommand{\AFf}{\A_{F,\mathbf{f}}}
\newcommand{\Afp}{\Af^{(p)}}
\newcommand{\AFfp}{\AFf^{(p)}}

\newcommand{\Fbar}{\overline{F}}

\newcommand{\CA}{\mathcal{A}}

\newcommand{\CC}{\mathcal{C}}

\newcommand{\CE}{\mathcal{E}}
\newcommand{\CF}{\mathcal{F}}

\newcommand{\CH}{\mathcal{H}}

\newcommand{\CJ}{\mathcal{J}}

\newcommand{\CL}{\mathcal{L}}

\newcommand{\CN}{\mathcal{N}}
\newcommand{\CO}{\mathcal{O}}

\newcommand{\CV}{\mathcal{V}}

\newcommand{\COF}{\CO_F}
\newcommand{\COFppx}{\CO_{F, (p),+}^\times}
\newcommand{\COFhat}{\widehat{\CO}_F}

\newcommand{\sfD}{\mathsf{D}}
\newcommand{\sfM}{\mathsf{M}}
\newcommand{\sfSigma}{\mathsf{\Sigma}}

\newcommand{\frakc}{\mathfrak{c}}
\newcommand{\frakd}{\mathfrak{d}}
\newcommand{\frakn}{\mathfrak{n}}
\newcommand{\frakm}{\mathfrak{m}}
\newcommand{\frakp}{\mathfrak{p}}
\newcommand{\frakH}{\mathfrak{H}}
\newcommand{\frakM}{\mathfrak{M}}

\newcommand{\ttS}{\mathtt{S}}
\newcommand{\ttT}{\mathtt{T}}

\newcommand{\Ximin}{\Xi_{\min}}
\newcommand{\Ximinplus}{\Ximin^{+}}

\newcommand{\tilY}{\widetilde{Y}}
\newcommand{\tilYUe}{\tilY_U^\epsilon}
\newcommand{\tilYUetor}{\tilY_U^{\epsilon,\tor}}
\newcommand{\tilYUesig}{\tilY_{U, \sfSigma}^\epsilon}
\newcommand{\tilYUemin}{\tilY_U^{\epsilon,\min}}
\newcommand{\YUtor}{{Y_U^{\tor}}}
\newcommand{\YUmin}{{Y_U^{\min}}}
\newcommand{\tilYUtor}{{\tilY_U^{\tor}}}
\newcommand{\tilYUmin}{{\tilY_U^{\min}}}
\newcommand{\dotw}{\widetilde{\omega}}
\newcommand{\dotdelta}{\widetilde{\delta}}

\newcommand{\Fr}{\mathrm{Fr}}
\newcommand{\Ha}{\mathrm{Ha}}
\newcommand{\Frob}{\mathrm{Frob}}
\newcommand{\Nm}{\mathrm{Nm}}
\newcommand{\JH}{\mathrm{JH}}
\newcommand{\Sym}{\mathrm{Sym}}
\renewcommand{\det}{\mathrm{det}}
\newcommand{\Symo}{\Sym_{[0]}}
\newcommand{\Syml}{\Sym_{[1]}}
\newcommand{\Szero}{\mathrm{S}_{[0]}}
\newcommand{\Sone}{\mathrm{S}_{[1]}}
\newcommand{\e}{\mathrm{e}}
\newcommand{\kmin}{k_{\min}}
\newcommand{\dR}{\mathrm{dR}}
\newcommand{\BDJ}{\mathrm{BDJ}}
\newcommand{\AH}{\mathrm{AH}}
\newcommand{\Tr}{\mathrm{Tr}}

\newcommand{\tor}{\mathrm{tor}}

\renewcommand{\min}{\mathrm{min}}

\newcommand{\HT}{\mathrm{HT}}

\newcommand{\JL}{\mathrm{JL}}
\newcommand{\Ind}{\mathrm{Ind}}
\newcommand{\Art}{\mathrm{Art}}
\newcommand{\univ}{\text{univ}}
\newcommand{\ab}{\text{ab}}

\newcommand{\tr}{\mathrm{tr}}
\newcommand{\PSL}{\mathrm{PSL}}
\newcommand{\PGL}{\mathrm{PGL}}

\DeclareMathOperator{\Gal}{Gal}
\DeclareMathOperator{\GL}{GL}
\DeclareMathOperator{\Proj}{Proj}
\DeclareMathOperator{\End}{End}

\DeclareMathOperator{\Res}{Res}

\title{On the geometric Serre weight conjecture for Hilbert modular forms}

\author{Siqi Yang}
\email[]{siqi.yang21@imperial.ac.uk}
\address{Department of Mathematics\\
  Imperial College London\\
  London SW7 2AZ\\ UK }

\date{\today}

\begin{document}

\begin{abstract}
  Let $p$ be a prime, $F$ be a totally real field in which $p$ is unramified and $\rho: \mathrm{Gal}(\overline{F}/F)\rightarrow \mathrm{GL}_2(\overline{\mathbb{F}}_p)$ be a totally odd, irreducible, continuous representation. The geometric Serre weight conjecture formulated by Diamond and Sasaki can be viewed as a geometric variant of the Buzzard-Diamond-Jarvis conjecture, where they have the notion of geometric modularity in the sense that $\rho$ arises from a mod $p$ Hilbert modular form and algebraic modularity in the sense of Buzzard-Diamond-Jarvis. Diamond and Sasaki conjecture that if $\rho$ is geometrically modular of weight $(k,l)\in \mathbb{Z}^\Sigma_{\geq 2}\times\mathbb{Z}^\Sigma$ and $k$ lies in the minimal cone, then $\rho$ is algebraically modular of the same weight, where $\Sigma$ is the set of embeddings from $F$ into $\overline{\mathbb{Q}}$. We prove the conjecture without parity hypotheses for real quadratic fields $F$ in which $p \geq 5$ is inert, and for totally real fields $F$ in which $p \geq \min\{5, [F:\mathbb{Q}]\}$ totally splits.
\end{abstract}

\maketitle

\tableofcontents

\section{Introduction}

Serre's Conjecture can be regarded as a mod $p$ Langlands correspondence, and the weight part of Serre's Conjecture can be viewed as a type of local-global compatibility at $p$ in a hypothetical mod $p$ Langlands Programme.
Let $p$ be a rational prime and $\rho: \Gal(\overline{\Q}/\Q) \rightarrow \GL_2(\Fpbar)$ be a continuous, odd, irreducible representation.
Serre conjectured in \cite{Serre87} that such mod $p$ representation $\rho$ is \emph{modular} in the sense that it is isomorphic to the mod $p$ representation attached to a modular eigenform.
Building on the work of many others, Serre's conjecture was proved by Khare and Wintenberger \cite{Khare-Wintenberger-I, Khare-Wintenberger-II}, and Kisin \cite{Kisin09}. 
Moreover, under the assumption that $\rho$ is modular, Serre predicted the minimal weight $k(\rho) \geq 2$ and level $N(\rho)$ prime to $p$ such that $\rho$ arises from an eigenform of weight
$k(\rho)$ and level $N(\rho)$.
In particular, the description of the minimal weight $k(\rho)$ 
is given by an explicit recipe in terms of the restriction of $\rho$ to the inertia subgroup at $p$, which we
refer to as \emph{the weight part of Serre's Conjecture}.
Edixhoven reformulated the weight part of Serre's conjecture in \cite{Edixhoven}, 
where he included the weight one modular forms
by considering mod $p$
modular forms as sections of certain line bundles on the modular curves in characteristic $p$,
and hence involving the weight one mod $p$ modular forms that do not lift to characteristic zero.
The weight part of Serre's conjecture and its refinement were proved by Edixhoven \cite{Edixhoven} for $p > 2$,
building on the work of
Ribet \cite{Ribet90}, Gross \cite{Gross90}, and Coleman-Voloch \cite{Coleman-Voloch}; 
the remaining case for $p = 2$ was completed by Khare, Wintenberger and Kisin \cite{Khare-Wintenberger-I,Khare-Wintenberger-II, Kisin09}.

There has been a significant amount of work on the generalisations of the
weight part of Serre's Conjecture (based on an algebraic reinterpretation of the weight).
The most general formulation to date was given by Gee, Herzig
and Savitt in \cite{Gee-Herzig-Savitt} on connected reductive groups unramified at $p$.
One of the directions to generalise the weight part of Serre's conjecture is replacing $\Q$ with a totally real field $F$ of degree $d > 1$ and replacing modular forms with Hilbert modular forms.
A conjecture in this setting was formulated in \cite{Buzzard-Diamond-Jarvis} under the assumption that $p$ is unramified in $F$,
and was generalised to the case when $p$ is ramified in $F$ in \cite{Schein08} and \cite{Gee11}.
Let $\rho: \Gal(\overline{F}/F) \rightarrow \GL_2(\Fpbar)$ be a continuous, totally odd, irreducible representation,
and $V$ be an $\Fpbar$-representation of $\GL_2(\CO_F/p\CO_F)$, 
where $\CO_F$ is the ring of integers of $F$.  
There is a notion of $\rho$ being \emph{modular of weight $V$} in \cite{Buzzard-Diamond-Jarvis} in the sense that 
$\rho$ arises in the \'etale cohomology of a certain quaternionic Shimura curve
over $F$ with coefficients in a lisse sheaf associated with $V$.
It was
proved in \cite{Buzzard-Diamond-Jarvis} that $\rho$ is
modular of weight $V$ if and only if $\rho$ is modular of  some Jordan-Holder factor of $V$.
Then to determine the weights $V$ for which $\rho$ is
modular, 
it suffices to consider the weights $V$ that are irreducible as representations of $\GL_2(\CO_F/p\CO_F)$ which we call \emph{Serre weights}. 
In the work of Buzzard, Diamond and Jarvis, under the assumption that $\rho$ is modular of some weight, they defined a set $W^\BDJ(\rho)$ of Serre weights in terms of the restriction of $\rho$ to the inertia groups at primes of $F$ above $p$, and conjectured that the weights contained in $W^\BDJ(\rho)$ are the Serre weights for which $\rho$ is modular.
Under mild Taylor-Wiles
hypotheses, the Buzzard-Diamond-Jarvis conjecture was proved for $p > 2$ 
in \cite{BLGG13} and \cite{Gee-Liu-Savitt-Unitary} for a version of the conjecture for unitary groups, 
in \cite{Gee-Kisin} and \cite{Newton14} for the case of indefinite quaternion algebras, where \cite{Gee-Kisin} also treated the case of definite
quaternion algebras,
and in \cite{Gee-Liu-Savitt-GL2} which included the case when $p$ is ramified in $F$.
While the description for the predicted set $W^\BDJ(\rho)$ in \cite{Buzzard-Diamond-Jarvis} is explicit, it is not practical for computations when $\rho$ is a non-split extension of characters.
An alternative recipe was given in \cite{Dembele-Diamond-Roberts},
where they defined a more computable set $W^\mathrm{AH}(\rho)$ and conjectured that it agrees with $W^\BDJ(\rho)$, which was later proved by \cite{CEGM17}.

It is also natural to consider a generalisation of Edixhoven's refinement conjecture for a totally real field $F$ of degree $d > 1$.
However, the geometric variant on the weight part of Serre's Conjecture generalizing Edixhoven's refinement is much less studied than the algebraic versions.
Viewing mod $p$ Hilbert modular forms as sections of certain line bundles on Hilbert modular varieties in characteristic $p$, 
Diamond and Sasaki formulated a geometric Serre weight conjecture for $p$ unramified in $F$ in \cite{Diamond-Sasaki}, 
and generalised it to the ramified case in their forthcoming paper \cite{Diamond-Sasaki_inprep}.
There is the notion of $\rho$ being \emph{geometrically modular of weight $(k,l)$} in the sense that $\rho$ arises from a mod $p$ Hilbert eigenform of such weight,
where $(k, l) \in \Z^\Sigma \times \Z^\Sigma$
and $\Sigma$ denotes the set of embeddings of $F \hookrightarrow \Qbar$.
Diamond and Sasaki conjectured and now proved (with mild technical hypotheses) in their forthcoming
paper that, if $\rho$ is geometrically modular of some weight, then there is a minimal weight $\kmin$ (depending on $l$) in the minimal cone $\Ximinplus$  such that 
$\rho$ is geometrically modular of weight $(k,l)$ if and only if $k\geq_{\Ha} \kmin$,
where 
$$
\Ximinplus := \left\{k \in \Z^\Sigma_{\geq 1} \mid pk_\tau \geq k_{\Fr^{-1}\circ \tau}, \text{ for all } \tau\in \Sigma \right\},
$$
and $\geq_{\Ha}$ is a partial order given by the weights of partial Hasse invariants.

In this paper we discuss the relation between the two notions of modularity in \cite{Buzzard-Diamond-Jarvis} and \cite{Diamond-Sasaki}, 
under the assumption that $p$ is unramified in $F$.
We say $\rho$ is \emph{algebraically modular of weight} $(k,l) \in \Z^\Sigma_{\geq 2} \times \Z^\Sigma$ if it is modular of weight 
$V_{k, -k-l}$ in the sense of \cite{Buzzard-Diamond-Jarvis},
where $V_{k, m}$ denotes the $\Fpbar$-representation of $\GL_2(\CO_F/p \CO_F)$ of the following form
$$
\bigotimes_{\tau\in \Sigma}\left( \det^{m_\tau}_{[\tau]} \otimes \Sym_{[\tau]}^{k_\tau-2}\Fpbar^2\right),
$$
where $\GL_2(\CO_F/p \CO_F)$ acts on the factor indexed by $\tau$ via the homomorphism to
$\GL_2(\Fpbar)$ induced by $\tau$.
Our definition for algebraic modularity is different from \cite[Definition 7.5.1]{Diamond-Sasaki} by a twist by determinant because
our convention for Hecke action does not include the Tate twist as in \cite[\S4.3]{Diamond-Sasaki}.
Let $(k,l) \in \Z^\Sigma_{\geq 2} \times \Z^\Sigma$.
It is conjectured in \cite[Conjecture 7.5.2]{Diamond-Sasaki} that
if $\rho$ is algebraically modular of weight $(k,l)$, then it is geometrically modular of weight $(k,l)$; the converse is conjectured to hold if $k\in \Ximinplus \cap \Z_{\geq 2}^\Sigma$.
When $k_\tau \geq 2$ has the same parity for all $\tau$, 
it has been proved in \cite{Diamond-Sasaki} that 
$\rho$ being algebraically modular of weight $(k,l)$
implies geometrically modular of the same weight,
and the result was generalised to non-paritious weight in \cite{Diamond-Sasaki_inprep}.
The other direction of the conjecture is much harder and very little result has been established. 
In this paper, we focus on the other direction and prove the following result when $p$ totally splits in $F$ (see \Cref{thm: goem-alg-split}), which (under a mild hypothesis) proves \cite[Conjecture 7.5.2]{Diamond-Sasaki} in the case $p \geq \min\{5, [F:\Q]\}$ totally splits in $F$.
\begin{thmx}\label{thmx: goem-alg}
  Let $p \geq d$, $p > 3$ and $F$ be a totally real field in which $p$ totally splits.
  Let $\rho: G_F \rightarrow \GL_2(\Fpbar)$ be an irreducible, continuous, totally odd representation.
  If $p = 5$, assume that the projective image of $\rho|_{G_{F(\zeta_p)}}$ is
  not isomorphic to $A_5$.
  If $\rho$ is geometrically modular of weight $(k,l)$,
  where $k_\tau \geq 2$ for all $\tau$, 
  then $\rho$ is algebraically modular of weight $(k,l)$.
\end{thmx}
When $p$ is non-split, we treat the quadratic inert case (see \Cref{thm: geom-alg-allwt}), 
which (under a mild hypothesis) proves \cite[Conjecture 7.5.2]{Diamond-Sasaki} in the case $p > 3$ and $F$ is a real quadratic field in which $p$ is inert.
\begin{thmx}
  Let $p > 3$ and
  $F$ be a real quadratic field in which $p$ is inert.
  Let $(k_0,k_1) \in \Ximinplus \cap \Z_{\geq 2}^2$ and $(l_0, l_1) \in \Z^2$.
  Let $\rho : G_F \rightarrow \GL_2(\Fpbar)$ be an irreducible, continuous, totally odd representation.
  If $p = 5$, assume that the projective image of $\rho|_{G_{F(\zeta_p)}}$ is not isomorphic to $A_5$.
  If $\rho$ is geometrically modular of weight $((k_0,k_1), (l_0, l_1))$, then $\rho$ is algebraically modular of weight $((k_0,k_1), (l_0, l_1))$.
\end{thmx}

We now give an outline of our approach.
Let $\rho$ be geometrically modular of weight $(k,l)$, where $(k,l) \in \Z^\Sigma_{\geq 2} \times \Z^\Sigma$ and $k\in \Ximinplus$.
If $(k,l)$ is a \emph{paritious} weight in the sense that $k_\tau + 2l_\tau$ is independent of $\tau$ for all $\tau\in\Sigma$,
to prove $\rho$ being algebraically modular of weight $(k,l)$, 
it suffices to show
the liftability of mod $p$ Hilbert cusp forms of weight $(k,l)$ to characteristic zero.
However, if $(k,l)$ is non-paritious, 
there does not exist a Hilbert modular form of this weight in characteristic zero, 
so we cannot talk about liftability.
The way around it was specified in \cite{Diamond-Sasaki_inprep}, 
where they consider twisting by a non-algebraic character $\xi$ which exists in character $p$ but not necessarily in characteristic $0$.
For any mod $p$ Hilbert modular form of weight $(k,l)$, 
there is a twisted form of weight $k$ whose Hecke action is twisted by $\xi\circ\det$ (as in \eqref{eq: twist_xi_k/2}) and there are forms of such weight in characteristic zero.

The cohomology vanishing has been discussed in \cite{Lan-Suh_parallel} for parallel weights over the special fibre of PEL type Shimura varieties, and they improved the result to the weights that are not far from being parallel (different from parallel weights by a ``$p$-small'' part) in \cite{Lan-Suh_genrealPEL}.
We adapt the method in \cite{Lan-Suh_parallel} and then prove a cohomology vanishing result 
(\Cref{thm: my-vanish}) for not necessarily parallel weights over the special fibres of Hilbert modular varieties.
The key ingredients are the
vanishing theorem of Esnault-Viehweg and the ampleness result of the line bundles over the special fibres of Hilbert modular varieties proved by Deding Yang in \cite{Yang_ampleness}.
Then the liftability for weight $k$ forms
is directly given by the cohomology vanishing, if $k$ satisfies
\begin{align}\label{eq: liftcone-intro}
 p(k_\tau - 2) > k_{\Fr^{-1}\circ \tau} - 2, \text{ for all } \tau\in \Sigma.
\end{align}
Suppose $k$ satisfies
the above numerical condition. The Hecke eigensystem of the twisted mod $p$ representation $\xi \otimes \rho$ can be lifted to characteristic zero,
which corresponds to the Hecke eigensystem of some definite quaternionic automorphic forms or the cohomology of some Shimura curves by Jacquet-Langlands correspondence,
then we take the mod $p$ reduction and twist by $\xi^{-1}$,
hence we obtain $\rho$ algebraically modular of weight $(k,l)$.

We now treat $k \in \Ximinplus \cap \Z^\Sigma_{\geq 2}$ that does not satisfy \eqref{eq: liftcone-intro} with a weight shifting method.
Let $\Ha_\tau$ denote the partial Hasse invariants and $\Theta_\tau$ denote the partial Theta operators, where $\tau\in \Sigma$.
Suppose $\rho$ arises from a mod $p$ Hilbert modular eigenform $f$ of weight $(k,l)$.
We consider another eigenform $f' = \prod_{\tau\in \Sigma} \Ha_{\tau}^{s_\tau} \Theta_\tau^{t_\tau} f$ with same Hecke eigensystem,
where $s_\tau, t_\tau \geq 0$, 
and in particular, 
we choose $s_\tau$ and $t_\tau$ such that 
$f'$ has weight $(k',l')$, where $k'$ satisfies \eqref{eq: liftcone-intro}.
Note that we can choose $f$ so that $f'$ is non-zero by
\cite[Lemma 10.4.1]{Diamond-Sasaki}.
Then $\rho$ is algebraically modular of 
weight $(k', l')$ by the liftability result,
and hence is modular of a Jordan-Holder factor of 
$V_{k', -k'-l'}$.
Similarly, we can take another 
eigenform $f'' = \prod_{\tau\in \Sigma} \Ha_{\tau}^{u_\tau} \Theta_\tau^{v_\tau} f$ by wisely choosing $u_\tau, v_\tau \geq 0$, 
so that weight $(k'',l'')$ of $f''$ satisfies \eqref{eq: liftcone-intro} and 
$V_{k'', -k''-l''}$ has some Jordan-Holder factors that are not given by those of $V_{k', -k'-l'}$.
Using the Grothendieck group relation in \cite{Reduzzi15}, we can compute the set of Jordan-Holder factors 
$\JH(V_{k', -k'-l'})$ and $\JH(V_{k'', -k''-l''})$ explicitly.
By Buzzard-Diamond-Jarvis conjecture 
$W^\BDJ(\rho) \cap \JH(V_{k', -k' - l'}) \neq \varnothing$ and $W^\BDJ(\rho) \cap \JH(V_{k'', - k'' - l''}) \neq \varnothing$,
which allows us to compute all the possibilities of $W^\BDJ(\rho)$ by the explicit recipe in \cite{Dembele-Diamond-Roberts}.
We observe that, for the case when $p$ totally splits, 
we always have 
$\JH(V_{k,-k-l}) \cap W^\BDJ(\rho) \neq \varnothing$ for any of these possibilities of $W^\BDJ(\rho)$.
When $p$ is inert in the real quadratic field $F$, we also have the observation $\JH(V_{k,-k-l}) \cap W^\BDJ(\rho) \neq \varnothing$ if $k_\tau \leq p$ for all $\tau\in\Sigma$. 
Then $\rho$ is algebraically modular of weight $(k,l)$ by Buzzard-Jarvis-Diamond conjecture if it satisfies a mild Taylor-Wiles hypotheses; 
otherwise, $\rho$ is an induction of a character of a quadratic extension $F'/F$ that is ramified at $p$, and a further analysis gives algebraic modularity of weight $(k,l)$ if $p \geq 5$.   
We remark that the weight shifting method encounters a difficulty for weight $k = (2, p+1)$ in the inert quadratic case,
where the obstruction comes from a common Jordan-Holder factor contained in both $\JH(V_{k',-k'-l'})$ and $\JH(V_{k'',-k''-l''})$ which is not 
in $\JH(V_{k,-k-l})$
and we don't know how to rule it out.

We treat the remaining weights in the inert quadratic case by restricting the Hilbert eigenform $f$ that gives rise to $\rho$ to a stratum of the mod $p$ Hilbert modular variety.
We recall  
the Goren-Oort stratification on mod $p$ Hilbert modular varieties
described by Tian and Xiao in
\cite{TX16}, 
which identifies a stratum with a $\P^1$-bundle over the quaternionic Shimura variety associated to a quaternion algebra $B$ over $F$.
In particular, $B$ is a definite quaternionic algebra in the case when $p$ is inert in the real quadratic field $F$.
Assume $k = (k_0, k_1) \in \Ximinplus$, where $k_1 \geq k_0 \geq 2$.
By restricting $f$ to the zero locus of the partial Hasse invariant $\Ha_{\tau_0}$, we obtain a definite quaternionic form $f^B$ whose Hecke eigensystem also gives rises to $\rho$. 
It shows that $\rho$ is algebraically modular of another weight $(k^B, l^B)$, where the weight $(k^B, l^B)$ can be computed explicitly by the description of Goren-Oort stratification.
If $k$ is relatively big in the sense that $pk_0 \geq k_1 \geq p+1$ (note that $(2,p+1)$ satisfies it, but the weights dealt by weight shifting method don't),
the Grothendieck group relation shows that $V_{k^B, - k^B - l^B}$ is a subquotient of $V_{k, - k - l}$,
which 
proves that $\rho$ is 
algebraically modular of weight $(k, l)$.

For an arbitrary totally real field $F$ in which $p$ does not split, we expect the above two methods: weight shifting and restricting strata would also work (at least to give partial results), but some amount of work might be needed.
For the weight shifting method, we require explicit computations of the Jordan-Holder factors of $V_{k',l'}$ and of the set $W^\BDJ(\rho)$, 
while the way they are done now is kind of ad hoc, so we want a more systematic way to do it. 
For the method using the description of stratifications,
we expect an induction process on the number of infinite places where the quaternion algebra is ramified, 
which can lead to a generalisation of the geometric modularity with respect to an arbitrary quaternion algebra over the totally real field $F$.

\subsection*{Structure of the paper}
In \Cref{sec: Geom HMF}, we recall the definitions and some properties of Hilbert modular varieties and geometric Hilbert modular forms. 
We also consider definite quaternionic forms and the cohomology of Shimura curves, and recall the  Goren-Oort stratification on mod $p$ Hilbert modular varieties established in \cite{TX16}.
In \Cref{sec: wt_part_Serre}, we recall the weight part of Serre's conjecture for totally real fields formulated in \cite{Buzzard-Diamond-Jarvis} and a geometric variant in \cite{Diamond-Sasaki},
by which we define the algebraic and geometric modularity.
In \Cref{sec: alg_geo_mod}, we discuss the relation between algebraic and geometric modularity and prove the main theorems using the weight shifting method and the description of Goren-Oort stratification.

\subsection*{Acknowledgments}

I would especially like to thank Fred Diamond for suggesting the problem, sharing his insightful ideas and generously spending his time for discussions. I am grateful to Toby Gee for helpful discussions and his interest in my work. I appreciate the support of Lassina Demb\'el\'e. I thank George Boxer, Payman Kassaei, Shu Sasaki, Liang Xiao and Deding Yang for helpful conversations. I thank Kai-Wen Lan and Mladen Dimitrov for answering my questions about compactifications. I also thank Fred Diamond, Toby Gee, Martin Ortiz, and Yitong Wang for comments on an earlier draft. 
This work was supported by the Engineering and Physical Sciences Research Council [EP/S021590/1]. The EPSRC Centre for Doctoral Training in Geometry and Number Theory (The London School of Geometry and Number Theory), University College London, Kings College London and Imperial College London.

\subsection*{Notation}
Let $p$ be a rational prime.
We fix an algebraic closure $\overline{\Q}$ of $\Q$ and an algebraic closure $\overline{\Q}_p$ of $\Q_p$. 
Fix embeddings $\overline{\Q} \hookrightarrow \overline{\Q}_p$, and $\overline{\Q}\hookrightarrow \C$.

Let $F$ be a totally real field in which $p$ is unramified.
Let $G_F$ denote the absolute Galois group of $F$.
Let $\CO_F$ be the ring of integers of $F$, $\CO_{F, p}$ be the $p$-adic completion of $\CO_F$,
and $\CO_{F,(p)}$ be the localization
of $F$ at $p$. 
We identify the set of embeddings $\Sigma = \{F \rightarrow \Qbar\}$ with $\Sigma_\infty = \{F \rightarrow \R\}$, 
$\Sigma_p = \{F \rightarrow \Qpbar\}$ and
$\overline{\Sigma}_p = \{\CO_F/p \rightarrow \Fpbar\}$.
For a prime $v$ in $F$ above $p$,
let $\F_v = \CO_F/v$,
$f_v = [\F_v: \Fp]$, 
and $\Sigma_v$ be the subset of embeddings $\{\F_v \rightarrow \Fpbar\}$, which can be identified with  
the set of embeddings in $\overline{\Sigma}_p$ that factor through $\CO_F/v$.
Then $\Sigma = \overline{\Sigma}_p = \coprod_{v|p}\Sigma_v$.

Let $L$ be a finite extension
of $\Q_p$ in $\overline{\Q}_p$ that is sufficiently large to  contain the image of $F$ under every embedding in $\Sigma$. 
Assume in addition $L$ contains 
$\tau(\mu)^{1/2}$ for every $\mu\in \CO_{F, (p), +}^\times$ and $\tau\in \Sigma$,
where $\CO_{F, (p), +}^\times \subset \CO_{F, (p)}^\times$ consists of the totally positive elements.
Let $\CO$ be its ring of integers, and $E$ its residue field.

We use bold font for a tuple in $\Z^\Sigma$ with same numbers, for example, $\mathbf{0} := (0,0, \cdots, 0)$.
We denote $\prod_{\tau\in \Sigma} \tau(\mu)^{m_\tau}$ by $\mu^{m}$, for $m\in \frac{1}{2}\Z^\Sigma$ and $\mu\in \COFppx$.

Suppose $K$ is a finite extension of $\Q_p$. 
Let $G_K$ denote the absolute Galois group of $K$, $I_K$ be the inertia subgroup of $G_K$ and $W_K$ be the Weil group of $K$.
Let $\Art_K : K^\times \rightarrow W_K^{\ab}$ be the isomorphism given by local class field theory, and we normalize it such that uniformisers map to geometric Frobenius elements.
For each finite place $v$ of $F$, we fix a uniformiser $\varpi_v$ of $F_v$ and a lift $\Frob_v\in G_{F_v}$ of the geometric Frobenius element at $v$.
For $\tau\in \Sigma_v$, we denote by $\varepsilon_\tau$ the fundamental character of $I_{F_v}$, which is the composite:
\begin{align*}
  I_{F_v} \xrightarrow{\Art^{-1}_{F_v}} \CO_{F_v}^\times \rightarrow \F_v^\times \xrightarrow{\tau} \Fpbar^\times.
\end{align*} 
We have $\prod_{\tau\in\Sigma_v}\varepsilon_\tau = \varepsilon$, where $\varepsilon$ denotes the mod $p$ cyclotomic character.

Let $V$ be a de Rham representation of $G_K$ over $\Qpbar$.
Denote $\Sigma_K$ the set of embeddings from $K$ into $\Qpbar$.
For any $ \varsigma \in \Sigma_K$,
define $\HT_\varsigma(V)$ to be the set consisting of integer $i$ with multiplicity
$$\dim_{\Qpbar} (V \otimes_{K, \varsigma} \widehat{\overline{K}}(i))^{G_K},$$
where $\widehat{\overline{K}}(i)$ is the Tate twist of the completed algebraic closure of $K$.
For example, $\HT_\varsigma(\chi_{\mathrm{cyc}}) = \{-1\}$ for $\varsigma \in \Sigma_K$, where $\chi_{\mathrm{cyc}}$ is the $p$-adic cyclotomic character.
If $V$ is a crystalline representation of rank $2$, define the Hodge-Tate type of $V$ to be $(\lambda_1, \lambda_2) \in (\Z^{\Sigma_K})^2$, 
where $\lambda_{1,\varsigma} \geq \lambda_{2,\varsigma}$ are the integers contained in $\HT_\varsigma(V)$ for $\varsigma \in \Sigma_K$.

\section{Geometric Hilbert modular forms}\label{sec: Geom HMF}

\subsection{Hilbert modular varieties of level prime to \texorpdfstring{$p$}{p}}\label{subsec: HMV}
We let $\AFf$ denote the finite adeles of $F$, $\AFf^{(p)}$ denote the prime-to-$p$ finite adeles and $\COFhat^{(p)}$ denote the
prime-to-$p$ completion of $\COF$.
Let $U = U^p U_p$ be a sufficiently small\footnote{sufficiently small in the sense that \cite[Lemma 2.4.1]{Diamond-Sasaki} holds.} 
open compact subgroup of $\GL_2(\AFf)$ such that $U^p \subset \GL_2(\AFf^{(p)})$ and $U_p = \GL_2(\CO_{F,p})$.
Let $\tilY_U/\CO$ be the fine moduli scheme of level $U$ defined as in \cite[\S 2.1]{Diamond_cmpt}.
More specifically, for a locally Noetherian $\CO$-scheme $S$, 
$\tilY_U(S)$ is the set of isomorphism classes of tuples $(A, \iota, \lambda, \eta)$ such that
\begin{enumerate}
  \item $s: A \rightarrow S$ is an abelian scheme over $S$ of relative dimension $d$;
  \item $\iota : \COF \rightarrow \End_S(A)$ is an embedding such that $s_*\Omega^1_{A/S}$ is a locally free $\COF\otimes_{\Z_p}\CO_S$-module of rank one;
  \item $\lambda$ is an $\COF$-linear quasi-polarisation of $A$ such that for each connected component $S_i$ of $S$, $\lambda$ induces an isomorphism $\frakc_i \mathfrak{d} \otimes_{\COF} A_{S_i} \rightarrow A^\vee_{S_i}$ for some prime-to-$p$ fractional ideal $\frakc_i$ of $F$,
  where $\frakd$ is the different of $F$ over $\Q$;
  \item $\eta$ is a level $U^p$ structure on $A$. More precisely, we choose a geometric point $s_i$ on each connected component $S_i$, 
  $\eta$ is a collection of $\pi(S_i, s_i)$-invariant $U^p$-orbits of $\COFhat^{(p)}$-linear isomorphism
  $\eta_i: (\COFhat^{(p)})^2 \xrightarrow{\sim} T^{(p)}(A_{s_i})$,
  where $T^{(p)}$ denotes the product over $\ell \neq p$ of the $\ell$-adic Tate modules, and $g\in U^p$ acts on $\eta_i$ by right multiplication by $g^{-1}$.
 \end{enumerate}
The scheme $\tilY_U$ admits an action of $\CO_{F,(p),+}^\times$, such that $\psi_\mu \cdot (A, \iota, \lambda, \eta) = (A, \iota, \mu\lambda, \eta)$, for  $\mu\in \CO_{F,(p),+}^\times$. 
The action factors through the quotient $\COFppx/{(U\cap \COF^\times)}^2$ which acts freely if $U$ is sufficiently small. 

We define the \emph{Hilbert modular variety of level $U$} to be
$$
Y_U := 
\big(\COFppx \big/{(U\cap \COF^\times)}^2\big)
\big\backslash \tilY_U.
$$

The scheme $\tilY_U$ is an infinite disjoint union of PEL Shimura varieties $\widetilde{Y}_U^\epsilon$ 
for $\epsilon\in \big(\AFf^{(p)}\big)^\times/\big(\det U^p\big) \big(\widehat{\Z}^{(p)}\big)^\times$,
where $\widetilde{Y}_U^\epsilon$ denotes the open and closed subscheme of $\tilY_U$ which consists of isomorphism classes of $(A, \iota, \lambda, \eta)$ satisfying the following commutative diagram in \cite[\S 2.1.2]{DKS}:
    $$\begin{tikzcd}
    (\COFhat^{(p)})^2 \times (\COFhat^{(p)})^2  \arrow[r, "\wedge^2"] \arrow[d, "\eta_i\times \eta_i"'] & \AFf^{(p)} \arrow[d, "\epsilon \otimes \zeta"] \\ 
    T^{(p)}(A_{s_i})\times T^{(p)}(A_{s_i}) \arrow[d, "1\times \lambda"'] & \AFf^{(p)}(1) \arrow[d, "\Tr_{F/\Q}"] \\
    T^{(p)}(A_{s_i})\times(\Q\times T^{(p)}(A^\vee_{s_i}))\arrow[r, "\text{Weil}"] & \A_{\bff}^{(p)}(1)\, ,
    \end{tikzcd}$$
    where $\zeta: \hat{\Z}(p) \xrightarrow{\sim} \hat{\Z}(p)(1)$ is a fixed isomorphism.

Considering the action of $\CO_{F,(p),+}^\times$, 
$\tilY_U$ is a union of the orbits of finitely many $\tilY_U^\epsilon$ 
and the stabiliser of each $\tilY_U^\epsilon$ is the finite group $G_U := (\CO_{F,+}^\times \cap \det(U))/(U\cap \COF^\times)^2$.
Let $\CE$ be a finite set such that $\tilY_U = \coprod_{\epsilon\in\CE} (\CO_{F,(p),+}^\times)\cdot(\tilY_U^\epsilon)$.
Then 
$$
Y_U = \coprod_{\epsilon\in \CE} G_U\backslash\tilY_U^\epsilon.
$$
The definition for $Y_U$ coincides with the definition in \cite{Diamond-Sasaki}, where they take $Y_U$ to be the disjoint union of the finite quotient of $\frakc$-polarised Hilbert-Blumenthal abelian varieties, for prime-to-$p$ ideals $\frakc$ run through the representatives of narrow ideal class group.

Although Hilbert modular varieties $Y_U$ are of abelian type, not PEL type, 
each of its connected components admits a finite \'etale projection from a PEL Shimura variety.
Therefore, $Y_U$ is smooth and quasi-projective over $\CO$.

\begin{remark}
    (Relation with the model in Lan's thesis). 
    The scheme $\tilY_U^\epsilon$ is isomorphic to $\sfM_{\CH, \CO}$ in \cite{Lan-thesis} 
    with $\square = \{p\}$, $L$ and $\CH$ as in \cite[\S 2.4]{Diamond_cmpt}.
\end{remark}

\subsubsection{Toroidal and minimal compactifications}\label{subsub: tor&min}

The toroidal and minimal compactifications of the Hilbert-Blumenthal modular variety 
(which can be identified with a finite disjoint union of $\tilY_U^\epsilon$) 
have been discussed in \cite{Rapoport78} and \cite{Chai90}.
The arithmetic compactifications of Hilbert modular varieties of level $\Gamma_1(\mathfrak{c}, \mathfrak{n})$ are studied in \cite{Dimitrov2004}.
A general theory for compactifications of Shimura varieties of PEL type is studied in \cite{Faltings-Chai} and more recently in \cite{Lan-thesis}.

We consider the toroidal compactifications for each connected component $\tilY_U^\epsilon$ of $\tilY_U$. 
By \cite[\S 6.4.1, \S 7.3.3]{Lan-thesis}, we have the following results for $\tilY_U^\epsilon$:
\begin{enumerate}
    \item Let $\sfSigma = \{\sfSigma^\CC\}_{\CC}$ be an admissible rational polyhedral cone decomposition\footnote{admissible in the sense that, for every cusp $\CC$ of $\tilYUe$, $\sfSigma^\CC$ is invariant under the natural action of $(\CO_F^\times\cap U)^2$ and the orbit is finite.}, where $\CC$ runs through all the cusps of $\tilYUe$.
    We have the toroidal compactification $\tilYUetor := \tilYUesig$ depending on the choice of $\sfSigma$ and an open immersion from $\tilde{\jmath}:\tilYUe \hookrightarrow \tilYUetor$.
    \item If $\sfSigma$ is smooth (resp.\,projective) in the sense of \cite[Definition 6.1.1.14]{Lan-thesis} (resp.\,\cite[Definition 7.3.1.3]{Lan-thesis}), then $\tilYUesig$ is smooth (resp.\,projective).
    \item There exists an admissible rational polyhedral cone decomposition $\sfSigma$ that is both smooth and projective.
    \item Let $\sfSigma$ be smooth and projective. The complement $\widetilde{D}_\infty := \tilYUetor - \tilYUe$ is a simple normal crossing divisor.    
    \item The universal abelian scheme $A^\epsilon \rightarrow \tilYUe$ extends to a semi-abelian scheme $A^{\epsilon,\tor} \rightarrow \tilYUetor$.
    \item Let $\tilYUemin$ be the minimal compactification obtain by Stein factorization as in \cite[\S 7.2.2]{Lan-thesis}\footnote{The minimal compactification is independent of the choice of $\sfSigma$.
    We can identify the normal projective scheme $\tilY_U^{\epsilon,\min}$ with $\Proj(\oplus_{k\geq 0} \Gamma(\tilY_U^{\epsilon,\tor}, \omega^{\tor, \otimes k}))$, 
    where $\omega^{\tor}= \wedge^d e^*\Omega_{A^{\epsilon,\tor}/\tilYUetor}$.}. 
    We have a canonical morphism $\widetilde{\pi} : \tilYUetor \rightarrow \tilYUemin$ and an open immersion $\tilde{\imath}: \tilYUe \hookrightarrow \tilYUemin$ such that $\tilde{\imath} = \widetilde{\pi}\circ \tilde{\jmath}$.
    $$\begin{tikzcd}[row sep=tiny]
        & \tilYUetor \arrow[dd, "\widetilde{\pi}"] \\
          \tilYUe \arrow[ur, hook, "\tilde{\jmath}"] \arrow[dr, hook, "\tilde{\imath}"] &              \\
          & \tilYUemin
        \end{tikzcd}$$
        
    \item There is a $\widetilde{\pi}$-relatively ample line bundle $\CO_{\tilYUetor}(-\widetilde{D}')$ on $\tilYUetor$, 
    for some divisor $\widetilde{D}'$ supported on $\widetilde{D}_{\infty}$.
\end{enumerate}

The action of $G_U$ on $\tilYUe$ induces an action on the cone decomposition $\sfSigma$.
Now we obtain finitely many cone decompositions $g\cdot\sfSigma$, where $g\in G_U$.
As in \cite[Corollaire 7.5]{Dimitrov2004}, 
we can take a common refinement $\sfSigma'$ of
$g\cdot\sfSigma$ for all $g\in G_U$ that is smooth 
and compatible with the action of $G_U$, where compatible means that $g\sfSigma' = \sfSigma'$ for all $g\in G_U$. 
We can further subdivide $\sfSigma'$ into $\sfSigma''$ to make it projective while remaining smooth and compatible with $G_U$-action.

Then $\tilYUetor = \tilY_{U,\sfSigma''}^\epsilon$ is a toroidal compactification (depending on the choice of $\sfSigma''$) that admits an action of $G_U$.
Let 
$$
\tilYUtor := \coprod_{\epsilon} \tilYUetor
$$
be the toroidal compactification of $\tilY_U$, 
where $\epsilon \in \big(\AFf^{(p)}\big)^\times/\big(\det U^p\big) \big(\widehat{\Z}^{(p)}\big)^\times$. Then $\tilYUtor$ admits universal semi-abelian scheme $A^{\tor}$.
Let 
$$
\YUtor := \coprod_{\epsilon\in \CE} G_U\backslash\tilYUetor
$$
be the toroidal compactification of $Y_U$.
Then $\YUtor$ is a projective and smooth scheme over $\CO$.
We have an open immersion $\jmath: Y_U \hookrightarrow \YUtor$
 and 
the complement ${D}_\infty := \YUtor - Y_U$ is a simple normal crossing divisor that can be identified with the disjoint union the quotient of $\widetilde{D}_\infty$ by $G_U$.

The action of $G_U$ on $\tilYUe$
extends uniquely to an action on $\tilYUemin$ and the morphism $\widetilde{\pi} : \tilYUetor \rightarrow \tilYUemin$ is $G_U$-equivariant.
Define the minimal compactification of $\tilY_U$ to be 
$$
\tilYUmin = \coprod_{\epsilon} \tilYUemin,
$$
where $\epsilon\in\big(\AFf^{(p)}\big)^\times/\big(\det U^p\big) \big(\widehat{\Z}^{(p)}\big)^\times$.
Let the minimal compactification of $Y_U$ be $$
\YUmin = \coprod_{\epsilon\in \CE}G_U\backslash\tilYUemin.
$$
We have an induced morphism
$\pi: \YUtor \rightarrow \YUmin$ and an open immersion 
$\imath: Y_U \hookrightarrow \YUmin$ such that $\imath = \jmath \circ \pi$.
Let $D'$ be the disjoint union of quotient of $\widetilde{D}'$ by $G_U$.
Then $D'$ is a divisor whose support is the simple normal crossing divisor $ {D}_{\infty}$ and $\CO_{\YUtor}(-D')$ is $\pi$-relatively ample. 

\begin{definition}
  We define $Y_U^\infty$ to be the set of cusps of $Y_U$ as in \cite[\S 7.1]{Diamond_wtshift}:
  $$
  Y_U^\infty:= P(F)_+\backslash \GL_2(\AFf)/U = 
  P(\CO_{F,(p)})_+\backslash \GL_2(\AFfp)/U^p,
  $$
  where $P(F)_{+} \subset \GL_2(F)$ and $P(\CO_{F,(p)})_+\subset \GL_2(\AFfp)$ denote the subgroups of upper-triangular matrices with totally positive determinant.
\end{definition}
There are only finitely many cusps. 
We have a bijection between $Y_U^\infty$ and the set of connected components of $\YUmin - Y_U$, and the set of connected components of $D_\infty = \YUtor - Y_U$ by \cite[\S 2]{Diamond_cmpt},
and we will refer to them also as the \emph{set of cusps}.

\subsection{Automorphic line bundles}

For $\tau \in \Sigma$, we define the line bundles $\widetilde{\omega}_\tau$ and $\widetilde{\delta}_\tau$ on $\tilY_U$ as in \cite[\S 3.2]{Diamond_wtshift}, 
and define the line bundles $\widetilde{\omega}^{\tor}_\tau$ and $\widetilde{\delta}^{\tor}_\tau$ on $\tilYUtor$ as in \cite[\S 7.1]{Diamond_wtshift}.\footnote{
  We use $\dotw_\tau$ (resp.\,$\dotdelta_\tau$) to denote $\CL_\tau$ (resp.\,$\CN_\tau$) in \cite[\S 3.2]{Diamond_wtshift},
  and $\dotw^{\tor}_\tau$ (resp.\,$\dotdelta^{\tor} _\tau$) to denote $\widetilde{\CL}^{\tor}_\tau$ (resp.\,$\widetilde{\CN}^{\tor}_\tau$) in \cite[\S 7.1]{Diamond_wtshift}. }

Let $A$ be the universal abelian scheme over $S = \tilY_U$ and $s: A \rightarrow S$ the structure morphism.
We have an exact sequence of locally free $\CO_F\otimes \CO_S$-modules 
$$
0\rightarrow s_*\Omega^1_{A/S} \rightarrow
 \CH_\dR^1(A/S) \rightarrow R^1s_*\CO_A \rightarrow 0
$$
of rank one, two and one, respectively.
We have decompositions 
$\CH_\dR^1(A/S) = \bigoplus_{\tau\in \Sigma}\CH_\dR^1(A/S)_\tau$, 
$s_*\Omega^1_{A/S} = \bigoplus_{\tau\in \Sigma}\widetilde{\omega}_\tau$,
and $\wedge^2 \CH_\dR^1(A/S) = \bigoplus_{\tau\in \Sigma} \widetilde{\delta}_\tau$.
Let $A^{\tor}$ be the semi-abelian scheme over $S^\tor = \tilYUtor$ that extends the universal abelian scheme $A$.
The sheaf $\CL ie(A^\tor/S^\tor)$ is locally free over $\CO_F\otimes \CO_{S^\tor}$, hence it has a similar decomposition.
We denote 
$\widetilde{\omega}_\tau^\tor := \CH om_{\CO_{S^\tor}}(\CL ie(A^\tor/S^\tor)_\tau, \CO_{S^\tor})$ 
the locally free $\CO_{S^\tor}$-module of rank one which extends $\widetilde{\omega}_\tau \cong \CH om_{\CO_{S}}(\CL ie(A/S)_\tau, \CO_{S})$.
Similarly, we also have a line bundle $\widetilde{\delta}^\tor_\tau$ over $S^\tor$ that extends $\widetilde{\delta}_\tau$.
Note that the line bundle $\dotdelta_\tau$ (resp.\,$\dotdelta^{\tor}_\tau$) is a free $\CO_{\tilY_U}$-module (resp.\,$\CO_{\tilYUtor}$-module) of rank one, not just locally.

For tuples $k = (k_\tau), l = (l_\tau) \in \Z^\Sigma$,
let 
$\widetilde{\CA}^{k,l}$
denote the line bundle $\bigotimes_\tau \widetilde{\omega}_\tau^{\otimes k_\tau} \otimes \widetilde{\delta}_\tau^{\otimes l_\tau}$
and 
$\widetilde{\CA}^{\tor, k,l}$
denote $\bigotimes_\tau \widetilde{\omega}_\tau^{\tor,\otimes k_\tau} \otimes \widetilde{\delta}_\tau^{\tor, \otimes l_\tau}$.
The line bundles $\widetilde{\CA}^{k,l}$ and $\widetilde{\CA}^{\tor, k, l}$ admit an action of $\COFppx$. 
Recall $\psi_\mu$ is the underlying action of $\mu\in \COFppx$ on $\tilY_U$,
and we choose the cone decomposition so that the action extends to $\tilYUtor$, which we denote by $\psi_\mu^t$.
We have isomorphisms
$\alpha_\mu: \widetilde{\CA}^{k,l} \xrightarrow{\sim} \psi_\mu^*\widetilde{\CA}^{k,l}$
and 
$\alpha_\mu^t: \widetilde{\CA}^{\tor, k,l} \xrightarrow{\sim} \left(\psi_\mu^t\right)^*\widetilde{\CA}^{\tor, k,l}$.
If $\mu^{k+2l}$ is trivial in an $\CO$-algebra $R$ for all $\mu\in \COF^\times \cap U$, 
then the line bundle $\widetilde{\CA}^{k,l}_R$ (resp.\,$\widetilde{\CA}^{\tor, k,l}_R$) descends to a line bundle ${\CA}^{k,l}_R$ (resp.\,$\widetilde{\CA}^{\tor, k,l}_R$) on $Y_{U,R}$ (resp.\,$Y^\tor_{U, R}$),
where the subscript $R$ denotes the base change.

\begin{remark}\label{rmk: paritious}
  Assume $U$ is sufficiently small.
  If $k_\tau +2 l_\tau = w$ is independent of $\tau$, we call such weight $(k,l)$ paritious, in which case we can always define $\CA^{k,l}$ over $\CO$. 
  For non-paritious weights, we only consider $\CA^{k,l}_R$ where $R$ has characteristic $p$ in this paper.
\end{remark}

\begin{proposition} 
  We have the Kodaira-Spencer isomorphisms on $\tilY_U$ and $\tilYUtor$ \cite[\S 5.1]{Diamond_cmpt}, hence also on the quotient schemes 
  $Y_U$ and $\YUtor$: 
  $$
  \Omega_{Y_U/\CO}^d = 
  {\bigwedge}^d_{\CO_{Y_U}}\Omega^1_{Y_U/\CO} 
  \cong \CA^{\mathbf{2}, \mathbf{-1}},
  $$
  $$
  \Omega_{\YUtor/\CO}^d({D}_\infty) = 
  {\bigwedge}^d_{\CO_{\YUtor}}\left(\Omega^1_{\YUtor/\CO}(\log({D}_\infty))\right) 
  \cong \CA^{\tor, \mathbf{2}, \mathbf{-1}}.
  $$
\end{proposition}

Now we consider the sheaves on the minimal compactifications.
Recall the open immersion $\tilde{\imath}: \tilY_U \hookrightarrow \tilY_U^{\min}$ and the morphism $\widetilde{\pi}:\tilYUtor \rightarrow \tilY_U^{\min}$.
The Koecher Principle implies that
$\CO_{\tilYUmin} = \widetilde{\pi}_*\CO_{\tilYUtor} = \tilde{\imath}_*\CO_{\tilY_U}$, and by \cite[Theorem 2.5]{Lan-HigherKoecher},
we have $\widetilde{\pi}_*\widetilde{\CA}^{\tor, k,l} = \tilde{\imath}_*\widetilde{\CA}^{k,l}$.

\begin{remark}\label{rmk: linebundle/Ymin}
  The sheaf $\tilde{\imath}_*\widetilde{\CA}^{k,l}$ on $\tilYUmin$ is coherent, but not necessarily invertible.
  By \cite[(22)]{Diamond_wtshift},
  if we take level $U$ sufficiently small such that
  $\mu^l := \prod_{\tau\in\Sigma}\tau(\mu)^{l_\tau}$ is trivial in $R$ for all $\mu\in U\cap \COF^{\times}$, then 
  $i_*\widetilde{\CA}^{k,l}$ is a line bundle on $\tilY_{U,R}^{\min}$.
  In particular,
  when $R = E$, 
  we can choose $U$ sufficiently small
  that $\mu \equiv 1$ mod $p$ for all $\mu\in U\cap \COF^{\times}$. 
\end{remark}

We choose $U$ sufficiently small so that $\mu^l$
and $\mu^{k+2l}$ is trivial in $E$ for all $\mu\in U \cap \COF^\times $.
Denote $\dotw_{\tau, E}^{\min} = \tilde{\imath}_*\dotw_E$, $\dotdelta_{\tau, E}^{\min} = \tilde{\imath}_*\dotdelta_{\tau, E}$ and $\widetilde{\CA}^{\min, k, l}_E = \tilde{\imath}_*\widetilde{\CA}^{k,l}_E$.
The line bundle $\widetilde{\CA}^{\min, k, l}_E$ descends to a line bundle ${\CA}^{\min, k, l}_E$ on $Y_{U, E}^{\min}$.
By Koecher Principle,
we have a natural isomorphism 
$\pi_* {\CA}^{\tor, k, l}_E \xrightarrow{\sim} {\CA}^{\min, k, l}_E$.

\begin{lemma}\label{lem: Ltor and Lmin}
  Let $U$ be sufficiently small such that 
  $\mu^l = 1$ in $E$ for $\mu\in \COFppx$.
  Then 
  the natural map of sheaves on $Y^\tor_{U,E}$
  $$
   \pi^*\pi_*\CA^{\tor,k,l}_E \longrightarrow \CA^{\tor,k,l}_E
  $$
  is an isomorphism.
\end{lemma}

\begin{proof}
  The natural map $\pi^*\pi_*\CA^{\tor,k,l}_E \rightarrow \CA^{\tor,k,l}_E$ is an isomorphism on $Y_{U, E}$. 
  By \cite[\href{https://stacks.math.columbia.edu/tag/05ET}{Proposition 05ET}]{stacks-project},
  to prove the isomorphism, it suffices to show the isomorphism on the completion along each component $C$ of the cuspidal divisor $D_\infty$.
  Let $\widehat{\pi}: (Y^{\tor}_{U,E})^\wedge_{C} \rightarrow (Y^{\min}_{U,E})^\wedge_{C}$ be the completion along $C$.
  Note that $\pi^*$ and $\pi_*$ are compatible with completion (the latter by the theorem on formal functions),
  so it suffices to prove that the natural map
  $$
   (\widehat{\pi})^*(\widehat{\pi})_*(\CA^{\tor,k,l}_E)_C^\wedge \rightarrow (\CA^{\tor,k,l}_E)_C^\wedge
  $$
  is an isomorphism.
  The key observation is that there's a trivialization of the completion of $({\CA}^{\tor, k, l})_{C}^\wedge$.\footnote{By the construction of the toroidal compactification, $\tilYUtor \cong \widehat{S}/V_N^2$, for a formal scheme $\widehat{S}$,
  where $V_N = \ker(\CO_F^\times \rightarrow (\CO_F/ N\CO_F)^\times)$.
  Let $\xi: \widehat{S} \rightarrow \tilYUtor$, then there's a canonical
  trivialization $\xi^*\widetilde{\CA}^{\tor, k, l} \cong D_{k,l}\otimes \CO_{\widehat{S}}$ that descends to $(\tilYUtor)_C^\wedge$,
  where $D_{k, l}$ is free of rank one over $\CO$
  (\emph{cf}.\,\cite[\S 7.1]{Diamond_wtshift}).} 
  Therefore, we only need to prove the result for the structure sheaf
  $$
  (\widehat{\pi})^*(\widehat{\pi})_* (\CO_{\YUtor, E})_C^\wedge \rightarrow (\CO_{\YUtor, E})_C^\wedge,
  $$
  where the isomorphism is given by the Koecher principle
  $\pi_*\CO_{\YUtor, E} \xrightarrow{\sim} \CO_{\YUmin, E}$.
\end{proof}

\subsection{Hilbert modular forms}

\begin{definition}
  Let $(k, l) \in \Z^\Sigma \times \Z^\Sigma$.
  Assume $\mu^{k+2l}$ is trivial in the $\CO$-algebra $R$ for all $\mu\in \COF^\times \cap U$.
  We define the space of Hilbert modular forms over $R$ of weight $(k,l)$ and level $U$ to be  
  $$
  M_{k, l}(U;R) := H^0(Y_{U,R}^{\tor}, \CA_R^{\tor, k, l})$$ 
  and the subspace of cusp forms to be 
  $$
  S_{k, l}(U;R) := H^0(Y_{U,R}^{\tor}, \CA_R^{\tor, k, l}(-D_\infty)).
  $$
  In particular, if $R=E$, we call it the space of mod $p$ Hilbert modular (reps. cusp) forms.
\end{definition}

By \Cref{rmk: paritious},
we can talk about Hilbert modular forms of paritious weights over $\CO$, but for non-paritious weights, we  only consider forms over $R$, where $R$ has characteristic $p$.

\begin{remark}
  By Koecher Principle,
  $
  H^0(Y_{U,R}^{\tor}, \CA_R^{\tor, k, l})
  = 
  H^0(Y_{U,R}, \CA_R^{k, l}).
  $
\end{remark}

For $f\in M_{k,l}(U;R)$, define the $q$-expansion of $f$ at each cusp $C$ as in \cite[\S 9.1]{Diamond-Sasaki}.
When $U = U_1(\frakn)$ for some ideal $\frakn$ prime to $p$, the set of connected components of $Y_U$ is in bijection with the strict class group $\CC l(F)^+$ of $F$,
where
\begin{align*}
  U_1(\frakn) = \left\{
  \left(\begin{matrix}
    a & b \\ c & d
  \end{matrix}\right) \in \GL_2(\widehat{\CO}_F) \mid c\in \frakn \widehat{\CO}_F, d-1 \in \frakn\widehat{\CO}_F
  \right\}.
\end{align*}
The $q$-expansion principle \cite[Lemma 9.2.1]{Diamond-Sasaki} tells us the following $q$-expansion\footnote{This $q$-expansion depends on a choice of basis for a free $R$-module of rank one. A canonical description is given in \cite[\S9.1]{Diamond-Sasaki}} map is injective
\begin{align}\label{eq: q-expn}
M_{k,l}(U_1(\frakn);R) \rightarrow \bigoplus_{\frakc \in \CC l(F)^+} 
\left\{ 
  \sum_{m\in \frakc_+ \cup \{0\}} r_m q^m \ \middle\vert \ r_m\in R, r_{\nu m} = \nu^{-l}r_m, \forall \nu\in \CO_{F,+}^\times
\right\}.
\end{align}

\subsubsection{Hecke operators}
We follow the convention in \cite{Diamond-Sasaki_inprep} to define the Hecke operators, which is different from \cite{Diamond-Sasaki} by a twist.

For sufficiently small open compact subgroups $U$ and $U'$ such that $g^{-1}U'g \subset U$, 
we have morphism $\rho_g: Y_{U', R} \rightarrow Y_{U, R}$ and $\rho_g^* \CA^{k,l}_R \xrightarrow{\sim} \CA'^{k,l}_R$ as in \cite[\S 4.1, \S 4.2]{Diamond-Sasaki}, 
for $g\in \GL_2(\AFf)$ such that $g_p\in \GL_2(\CO_{F,p})$.
We define the $R$-linear map 
$[U'gU]: M_{k, l}(U;R) \rightarrow M_{k, l}(U';R)$ to be 
the composite
\begin{align}\label{eq: Hecke_pullback}
  H^0(Y_{U,R}, \CA^{k,l}_R) \rightarrow 
  H^0(Y_{U', R}, \rho_g^* \CA^{k,l}_R)
  \xrightarrow{\sim}
  H^0(Y_{U', R}, \CA'^{k,l}_R).
\end{align}
Let $M_{k, l}(R) := \varinjlim_U M_{k, l}(U;R)$ and $S_{k,l}(R) := \varinjlim_U S_{k, l}(U;R)$.
Then we have the action of $\GL_2(\AFfp)$ on $M_{k, l}(R)$ and $S_{k, l}(R)$.

For a place $v$ of $F$ such that $v\nmid p$ and $\GL_2(\CO_{F,v})\subset U$, we define the Hecke operators $T_v$ and $S_v$ on $M_{k,l}(U;R)$ by
\begin{align*}
  T_v = \left[ U \left( \begin{matrix}
    \varpi_v & 0 \\ 0 &1
  \end{matrix}\right) U \right], \ 
  S_v =\left[ U \left(  \begin{matrix}
    \varpi_v & 0 \\ 0 & \varpi_v
  \end{matrix}\right) U \right],
\end{align*}
as in \cite[\S 4.3]{Diamond-Sasaki} but without twisting by $\Nm_{F/\Q}(v)^{-1}$ and $\Nm_{F/\Q}(v)^{-2}$, respectively, where $\varpi_v$ is a uniformiser of $\CO_{F,v}$.

For $v \mid p$ and weight $(k,l)$ such that
$\sum_{\tau\in \Sigma_v} \min\{l_\tau+1, k_\tau+ l_\tau\} \geq 0$,
we define the Hecke operators $T_v$ as in \cite[\S 5.4]{Diamond_Kodaira-Spencer},
with $p^{-f_v}\psi^*$ replacing $\psi^*$.

\subsubsection{Weight shifting operators}
Let $R = E$. 
For $\tau\in\Sigma$, we define the partial Hasse invariants $\Ha_\tau \in M_{k_{\Ha_\tau}, 0}(E)$ as in \cite[\S 5.1]{Diamond-Sasaki},
where $k_{\Ha_\tau} \in \Z^\Sigma$ is as follows:
\begin{itemize}
  \item if $\Fr\circ \tau = \tau$, then $k_{{\Ha_\tau},\tau} = p-1$ and $k_{{\Ha_\tau},{\tau'}}=0$ for $\tau' \neq \tau$;
  \item if $\Fr\circ \tau \neq \tau$, then $k_{{\Ha_\tau},\tau} = -1$, $k_{{\Ha_\tau},\Fr^{-1}\circ\tau} = p$ and $k_{{\Ha_\tau},{\tau'}}=0$ for $\tau' \neq \tau,\ \Fr^{-1}\circ \tau$;
\end{itemize}
where $\Fr$ denotes the absolute Frobenius on $\Fpbar$, and we can compose it with $\tau$ by the identification $\Sigma \cong \overline{\Sigma}_p$.
For any $(k,l)\in \Z^\Sigma\times \Z^\Sigma$, we have an injective map by multiplying by $\Ha_\tau$ 
$$
M_{k, l}(U;E) \hookrightarrow M_{k+k_{\Ha_\tau}, l}(U;E),
$$
which is Hecke-equivariant in the sense that it commutes with $T_v$ and $S_v$ for $v\nmid p$ and $\GL_2(\CO_{F,v})\subset U$.

For any $\tau\in \Sigma$, let $\Theta_\tau$ be the partial Theta operator as in \cite[\S 8.2]{Diamond-Sasaki}.
Then $\Theta_\tau$ defines a Hecke-equivariant map
$$
M_{k, l}(U;E) \rightarrow M_{k', l'}(U;E),
$$
where $l'_\tau = l_\tau -1$, $l'_{\tau'} = l_\tau$ for $\tau'\neq \tau$, and $k'$:
\begin{itemize}
  \item if $\Fr\circ \tau = \tau$, then $k'_{\tau} = k_\tau + p+1$ and $k'_{{\tau'}}=k_{\tau'}$ for $\tau' \neq \tau$.
  \item if $\Fr\circ \tau \neq \tau$, then $k_{\tau} =  k_\tau + 1$, $k'_{\Fr^{-1}\circ\tau} = k_{\Fr^{-1}\circ\tau} + p$ and $k'_{{\tau'}}=k_{{\tau'}}$ for $\tau' \neq \tau,\ \Fr^{-1}\circ \tau$.
\end{itemize}

\begin{theorem}\cite[Theorem 8.2.2]{Diamond-Sasaki}\label{thm: Theta-Hasse-divisibility}
  For $(k,l)\in \Z^\Sigma \times \Z^\Sigma$ and $f\in M_{k,l}(U;E)$,
   $\Theta_\tau(f)$ is divisible\footnote{By $f$ being divisible by $\Ha_\tau$, we mean that $f = \Ha_\tau\cdot g$ for some $g\in M_{k-k_{\Ha_\tau},l}(U;E)$.} 
  by $\Ha_\tau$ if and only if $f$ is divisible by $\Ha_\tau$ or $p|k_\tau$. 
\end{theorem}

\subsubsection{Hilbert modular forms over \texorpdfstring{$\CO$}{O}}

For any $k\in\Z^\Sigma$, we follow \cite{Diamond-Sasaki_inprep} to define a line bundle $\CL^{k}$ on $Y_U$ such that $\CL^k_E$ is twist of $\CA^{k, \mathbf{0}}_E$ by a free line bundle. 
Let $L$ be sufficiently large so that it contains 
$\tau(\mu)^{1/2} \in \Qbar \subset \Qpbar$ for any $\mu\in \COFppx$ and $\tau\in \Sigma$.
Recall we defined the action of $\alpha_\mu$ on $\widetilde{\CA}^{k,l}$ and it extends to $\alpha_\mu^t$ on $\widetilde{\CA}^{\tor, k,l}$ 
for $\mu\in \COFppx$.
Consider a twisted action of $\COFppx$ on $\widetilde{\CA}^{k, \mathbf{0}}$ defined by $\alpha_\mu' :=  \mu^{-k/2} \alpha_\mu$ for $\mu \in \COFppx$.
The action $\alpha_\mu'$ extends to $\widetilde{\CA}^{\tor, k, \mathbf{0}}$ and coincides with 
$\mu^{-k/2} \alpha_\mu^t$.
Under such action, the line bundle $\widetilde{\CA}^{k, \mathbf{0}}$ (resp.\,$\widetilde{\CA}^{\tor, k, \mathbf{0}}$) descends to a line bundle
$\CL^{k}$ (resp.\,$\CL^{\tor, k}$) on $Y_U$ (resp.\,$\YUtor$).
Recall
$$\begin{tikzcd}[row sep=tiny]
    & \YUtor \arrow[dd, "\pi"] \\
      Y_U \arrow[ur, hook, "\jmath"] \arrow[dr, hook, "\imath"] &              \\
      & \YUmin.
\end{tikzcd}$$ 
Let $\CL^{\min, k}_R := \imath_*\CL^{k}_R = \pi_*\CL^{\tor,k}_R$. 
When $U$ is sufficiently small such that $\mu^{k/2} = 1$ in $E$ for $\mu\in U \cap \COF^\times$,  
then $\CL^{\min, k}_E$ is a line bundle by checking $q$-expansion (\emph{cf}.\,\cite{Diamond-Sasaki_inprep}). 
Moreover, the natural map 
$$
\pi^*\pi_*\CL^{\tor,k}_E \rightarrow \CL^{\tor,k}_E
$$
is an isomorphism
for the same reason as \Cref{lem: Ltor and Lmin},
and hence we identify
$\pi^*\CL^{\min, k}_E$ with $\CL^{\tor, k}_E$.

\begin{definition}
  Let $R$ be an $\CO$-algebra. Define
  $$
  M_{k}(U; R) := H^0(Y_{U,R}, \CL^k_R)
  $$
  $$
  S_{k}(U; R) := H^0(Y^\tor_{U,R}, \CL^{\tor, k}_R (-D_\infty))
  $$
\end{definition}
We have the $\GL_2(\AFfp)$-action on $M_{k}(R):= \varinjlim_U M_{k}(U;R)$ similarly as how we define its action on $M_{k, l}(R)$, 
as well as
the Hecke operators $T_v$ and $S_v$ on $M_{k}(U; R)$, for $v\nmid p$ and $\GL_2(\CO_{F,v})\subset U$

Let $l \in \Z^\Sigma$.
Suppose $U$ is sufficiently small such that $\mu^{l + k/2} = 1$ in $R$ for $\mu\in \COFppx$.
Consider the action of $\COFppx$ on $\CO_{\tilY_{U,R}}$ defined by $\mu^{l + k/2}\alpha_\mu$ for $\mu\in \COFppx$.
Under such action, $\CO_{\tilY_{U, E}}$ descends to a line bundle $\CN_R$ on $Y_U$ which extends to $\CN^{\tor}_R$ on $\YUtor$.
Let $\CN^{\min}_R := \imath_* \CN_R$.
The line bundles $\CN_R$, $\CN^{\min}_R$, and $\CN^{\tor}_R$ are free (not just locally free) line bundles over $Y_U$, $\YUmin$ and $\YUtor$, respectively.
We have the following identification
$$
\CA^{k,l}_R = \CL^{k}_R \otimes \CN_R 
$$ 
and the same for minimal and toroidal compactifications.

\begin{remark}
  We can define $\CL^k$ over $\CO$, but we only consider $\CN_R$ for $R$ with characteristic $p$, unless $k_\tau + 2l_\tau = w$ is even and independent of $\tau$ for all $\tau$.
\end{remark}

The $q$-expansion Principle holds in this setting as well.
Let $U = U_1(\frakn)$ for some $\frakn$ prime to $p$ 
and $\frakc \in \CC l(F)^+$ be a cusp.
Then we have an injective map
\begin{align*}
  M_{k}(U_1(\frakn);R) \rightarrow \bigoplus_{\frakc \in \CC l(F)^+} 
\left\{ 
  \sum_{m\in \frakc_+ \cup \{0\}} r_m q^m \ \middle\vert \ r_m \in R,\ r_{\nu m} = \nu^{k/2}r_m, \forall \nu\in \CO_{F,+}^\times
\right\}.
\end{align*}

\subsection{Definite and indefinite quaternionic settings}\label{subsec: def&indef}
We define the quaternionic forms for definite quaternion algebras and consider the first cohomology of Shimura curves associated with indefinite quaternion algebra that is unramified at only one infinite place (\emph{cf}.\,\cite{Gee-Kisin}, \cite{Diamond-Sasaki_inprep}). 

Let $B$ be a quaternion algebra over $F$ that is split at primes above $p$.
Assume $B$ is either ramified at all infinite places
(definite case), 
or $B$ is ramified at all but one infinite place $\tau_0$ (indefinite case).
Let $\Sigma^B$ be the set of places of $F$ where $B$ is ramified and
$\Sigma^B_\bff$ (resp.\,$\Sigma^B_\infty$) be the set of finite (resp.\,infinite) places where $B$ is ramified.
Fix a maximal order $\CO_B$ of $B$ and fix isomorphisms $(\CO_B)_v \simeq M_2(\CO_{F_v})$ for $v \not\in \Sigma^B_\bff$.
In particular, we fix $\CO_{B,p}^\times \cong \GL_2(\CO_{F,p})$.

Denote $B_\A := B\otimes\A_{F}$,
$B_{\A,\bff} := B\otimes \AFf$,
$B_\bff^{(p)} := B \otimes \AFfp$ and $B_\infty := B\otimes \R$.
Let $U= U^pU_p$ be an open compact subgroup of $B_\A^\times $ such that 
$U_p = \CO_{B,p}^\times$ and 
$U^p \subset (B_\bff^{(p)})^\times$.
Let $Q$ be an $\CO$-module with a smooth action of $U_p$.

\begin{definition}
Suppose $B$ is definite.
We define
\begin{align*}
  M^B(U, Q) := \{
    f: B_{\A}^\times \rightarrow Q \mid f(b x u) = u_p^{-1} f(x)\text{ for } b \in B^\times,  x \in B_\A^\times,  u\in B_\infty^\times U\} \\
    \cong 
    \{
      f: (B_{\bff}^{(p)})^\times \rightarrow Q \mid f(b_\bff^{(p)} x u) = b f(x) \text{ for } b\in B^\times \cap U_p,  x \in (B_{\bff}^{(p)})^\times, u\in U^p
    \}.
\end{align*}
\end{definition}

When $B$ is definite, we let 
$
  Y^B_U := B^\times \backslash B_\A^\times/ B_\infty^\times U
  \cong B^\times \backslash B_{\bff}^\times/U   
$
which is a finite set,
and let $\CF_{Q,U} := B^\times \backslash B_{\bff}^\times \times Q/U$ be a locally constant sheaf on $Y^B_U$
such that 
the action of $B^\times$ and $U$ is given by 
$b (x, q) u = (b_\bff x u, u_{p}^{-1}t )$.
We can identify $M^B(U,Q)$ with $H^0(Y^B_U, \CF_{Q,U})$.

For $B$ definite, we have a natural action of $(B_{\bff}^{(p)})^\times$ on 
$\varinjlim_{U} M^B(U, Q)$.
Let $S$ be a set containing $\Sigma^B_\bff$, the finite places above $p$ and primes $v$ such that $(\CO_{B,v})^\times \not\subset U_v$.
We define Hecke operators $T_v$ and $S_v$ on $M^B(U, Q)$, for $v\not\in S$, via double cosets
\begin{align}\label{eq: double-coset-TvSv_B}
  T_v = \left[ U \left( \begin{matrix}
    \varpi_v & 0 \\ 0 &1
  \end{matrix}\right) U \right], \ 
  S_v =\left[ U \left(  \begin{matrix}
    \varpi_v & 0 \\ 0 & \varpi_v
  \end{matrix}\right) U \right].
\end{align}

\begin{definition}
  Suppose $B$ is indefinite. 
  Let 
  \begin{align*}
    Y^B_U := B^\times_+ \backslash (\frakH \times B_{\bff}^\times)/U 
    \cong 
    (B^\times_+ \cap U_p) \backslash (\frakH \times (B_{\bff}^{(p)})^\times)/U
  \end{align*}
  be a compact Riemann surface,
  where $\frakH$ is the complex upper half-plane.
  Let
  \begin{align*}
    \CF_Q  = \CF_{Q, U}:= B^\times_+ \backslash (\frakH \times B_{\bff}^\times \times Q)/U
    \cong 
    (B^\times_+ \cap U_p)\backslash (\frakH \times (B_{\bff}^{(p)})^\times \times Q)/U
  \end{align*}
  be the locally constant sheaf on $Y^B_U$,
  where the actions of $B^\times_+$ and $U$ are given by 
  $b (z, x, q) u = (b_\infty(z), b_\bff x u, u_{p}^{-1}t )$, and  
  $b (z, x, q) u = (b_\infty(z), b_\bff^{(p)} x u, b_p t)$, respectively.
  Define 
  \begin{align*}
    M^B(U, Q) := H^1(Y^B_U, \CF_Q).
  \end{align*}
\end{definition}

Let $g\in (B_\bff^{(p)})^\times$.
For compact open subgroups $U, U' \subset B_{\A,\bff}^\times$ such that $g^{-1}U'g \subset U$,
we have a covering map $\rho_g^B : Y^B_{U'}\rightarrow Y^B_U$ induced by right multiplication by $g$, 
and isomorphism 
$(\rho_g^B)^*: \CF_{Q, U} \rightarrow \CF_{Q, U'}$.
Define 
$[U' g U]: M^B(U, Q) \rightarrow M^B(U', Q)$ to be the composite
\begin{align*}
  H^1(Y^B_U, \CF_{Q, U}) \rightarrow H^1(Y^B_{U'}, (\rho_g^B)^*\CF_{Q, U})
  \xrightarrow{\sim} H^1(Y^B_{U'}, \CF_{Q, U'}).
\end{align*}
Define the double coset operator $[UgU]$ on $M^B(U, Q)$
to be the composite
\begin{align*}
  H^1(Y^B_U, \CF_{Q, U}) \rightarrow H^1(Y^B_{U'}, \CF_{Q, U'}) \rightarrow H^1(Y^B_U, \CF_{Q, U})
\end{align*}
where $U'= U \cap gUg^{-1}$,
the first map is $[U'g U]$ as above,
and the second map is the trace map relative to $\rho_1: Y^B_{U'} \rightarrow Y^B_U$.
We define the Hecke operators $T_v$ and $S_v$ by the double cosets as in \eqref{eq: double-coset-TvSv_B} 
for $v\not\in S$.

Let $R$ be an $\CO$-algebra.
Let $\Sym^n R^2$ be the $n$-th symmetric power of the standard representation of
$\GL_2(R)$,
and $\det: \GL_2(R) \rightarrow R^\times$ be the determinant character.

Let $k\in \Z^\Sigma_{\geq 2}$, and $l \in \Z^\Sigma$.
For an $\CO$-algebra $R$ such that $p^N R = 0$ for some $N>0$, define
\begin{align*}
  V_{k,l,R} := \bigotimes_{\tau\in \Sigma} \det_{[\tau]}^{l_\tau} \otimes_R \Sym_{[\tau]}^{k_\tau - 2}R^2,
\end{align*}
where the subscript $[\tau]$ means it admits a $\GL_2(\CO_{F,v})$-action given by the homomorphism $\GL_2(\CO_{F,v}) \rightarrow \GL_2(\CO)$ induced by $\tau$, 
for $v|p$ and $\tau\in \Sigma_v$.
We then define the action of 
$\CO_{B,p}^\times \cong \GL_2(\CO_{F,p}) \cong \prod_{v|p} \GL_2(\CO_{F,v})$ on $V_{k,l,R}$ such that 
$\GL_2(\CO_{F,v})$ acts on the $[\tau]$-part for $\tau\in \Sigma_v$. 

\begin{definition}
  Let $B$ be a definite or an indefinite quaternion algebra, 
  and $U \subset B_\A^\times$ an open compact subgroup such that $U_p = \CO_{B,p}^\times$.
  Let $R$ be an $\CO$-algebra such that $p^N R = 0$ for some $N>0$.
  Define
  \begin{align*}
    M^B_{k,l}(U; R) := M^B(U, V_{k,l+\mathbf{1},R}).
  \end{align*}
\end{definition}

For an $\CO$-algebra $R$, define 
\begin{align*}
  V_{k, R}:= \bigotimes_{\tau\in \Sigma} \det_{[\tau]}^{1 - k_\tau / 2} \otimes_R \Sym_{[\tau]}^{k_\tau - 2}R^2,
\end{align*}
which also admits an $\CO_{B,p}^\times$-action defined by the action of $\GL_2(\CO_{F,v})$ acting on the $[\tau]$-part for every $v|p$ and $\tau\in \Sigma_v$.

\begin{definition}
  Let $R$ be an $\CO$-algebra.
  For $B$ and $U$ as above, define
  \begin{align*}
    M^B_{k}(U; R) := M^B(U, V_{k,R}).
  \end{align*}
\end{definition}

\subsection{Twisting by characters}\label{subsec: twisting}
Following \cite[\S 4.6]{Diamond-Sasaki} and \cite{Diamond-Sasaki_inprep}, we discuss how twisting a character affects the $l$ for the space of weight $(k,l) \in \Z^\Sigma \times \Z^\Sigma$.

Let $R$ be an $\CO$-algebra.
Let
$l' \in \Z^\Sigma$ and
$\xi: (\A_{F,\bff}^{(p)})^\times \rightarrow R^\times$ be a continuous character such
that
$\xi(\mu) = \mu^{l'}$ for $\mu\in \COFppx$.
In particular, when $R = E$, the characters $\xi$ exists.
For a sufficiently small $U = U^p\GL_2(\CO_{F,p}) \subset \GL_2(\AFf)$ such that $\det(U^p) \subset \ker(\xi)$, $\xi\circ \det$ defines an eigenspace of rank one over $R$ in 
$H^0(Y_{U,R}, \CA^{0,l'}_{R})$.
Multiplying by a basis $e_\xi$ gives an isomorphism
\begin{align}\label{eq: twist_xi}
  M_{k,l}(U;R) \xrightarrow{\cdot e_\xi} M_{k,l+l'}(U;R)
\end{align} such that 
$T_v(e_\xi f) =\xi(\varpi_v) e_\xi T_v(f)$ and 
$S_v(e_\xi f) =\xi(\varpi_v)^2 e_\xi S_v(f)$ for $v\nmid p$ such that $\GL_2(\CO_{F, v}) \subset U$.

Let $B$ be either definite or indefinite and $U$ a compact open subgroup as in \ref{subsec: def&indef}. 
Assume $U$ is sufficiently small such that $\det(U^p) \subset \ker(\xi)$.
We have an element $e_\xi^B \in H^0(Y_U^B, \otimes_\tau \det_{[\tau]}^{l'_\tau} R)$ decided by 
$\xi\circ \det$.
In particular, the operator $[U' g U]: H^0(Y^B_{U}, V_{\mathbf{2},l',R}) \rightarrow H^0(Y^B_{U'}, V_{\mathbf{2},l',R})$ sends $e_\xi^B $ to $\xi(\det(g)) e_\xi^B$ 
for $g$ and $U'$ such that $g^{-1}U'h \subset U$.
Hence, we have an isomorphism given by the multiplication (definite case), or cup product (indefinite case) with $e_\xi^B$:
\begin{align}\label{eq: twist_xi_B}
  e^B_\xi: M^B_{k,l}(U;R) \xrightarrow{\sim} M^B_{k,l+l'}(U;R)
\end{align} such that 
$T_v(e^B_\xi f) =\xi(\varpi_v) e^B_\xi T_v(f)$ and 
$S_v(e^B_\xi f) =\xi(\varpi_v)^2 e^B_\xi S_v(f)$ for $v\not\in S$, where $S$ contains $\Sigma^B_\bff$, primes $v$ above $p$, and primes $v$ such that $(\CO_{B,v})^\times \not\subset U$.

Now let $\xi: (\AFfp)^\times \rightarrow R^\times$ be a continuous character such that $\xi(\mu) = \mu^{l + k/2}$, for $\mu\in \COFppx$.
Recall we have $\CA^{k,l}_R = \CL^{k}_R \otimes \CN_R$, for a free line bundle $\CN_R$.
When $U$ is sufficiently small, we can find $e_\xi \in H^0(Y_{U,R}, \CN_R)$ defined by $\xi \circ \det$, and obtain
\begin{align}\label{eq: twist_xi_k/2}
  M_k(U; R) \xrightarrow{\sim} M_{k,l}(U; R)
\end{align}
by multiplying $e_\xi$ as in \eqref{eq: twist_xi}.

Let $B$ and $U$ as in \ref{subsec: def&indef} and assume $U$ sufficiently small.
We have an element $e_\xi^B \in H^0(Y_U^B, \otimes_\tau \det_{[\tau]}^{l_\tau + k_\tau/2} R)$ defined by $\xi \circ \det$.
Multiplying $e_\xi^B$ gives an isomorphism 
\begin{align}\label{eq: twist_xi_B_k/2}
  e^B_\xi: M^B_{k}(U;R) \xrightarrow{\sim} M^B_{k,l}(U;R)
\end{align}
similar as \eqref{eq: twist_xi_B}.

\subsection{The Jacquet-Langlands Correspondence}
Recall the Jacquet-Langlands correspondence over $\C$ (\emph{cf.}\,\cite[\S 2.3]{Diamond-Sasaki_inprep}).

Let $B$ be a non-split quaternion algebra over
$F$ and
$\Sigma^B := \Sigma^B_\bff \cup \Sigma^B_\infty$ denote the set of places of $F$ where $B$ is ramified, where $\Sigma^B_\bff$ (resp.\,$\Sigma^B_\infty$) is the subset of finite (resp.\,infinite) places where $B$ is ramified.
Denote $B_v := B\otimes F_v$ if $v$ is finite, and $B_\tau := B \otimes_{F, \tau}\R$ if $\tau\in \Sigma$.
For $\tau\in \Sigma^B_\infty$, fix an isomorphism $B_\tau \otimes_\R \C \cong M_2(\C)$.

Let $\sfD_n$ be the irreducible unitary {(limit
of)} discrete series representation of $\GL_2(\R)$ of weight $n \in \Z_{\geq 1}$.
If $k\in \Z^\Sigma_{\geq 1}$,
let $\CC_k$ be the set of cuspidal automorphic representations $\pi = \otimes'_v \pi_v$ of $\GL_2(\AF)$ 
such that $\pi_\tau \simeq \sfD_{k_\tau}$ for any infinite place $\tau$.
If moreover $k_\tau \geq 2$ for $\tau\in \Sigma^B_\infty$,
let $\CC_k^B$ be the set of all cuspidal automorphic representations $\pi = \otimes'_v \pi_v$ of $(B\otimes \AF)^\times$
such that
$\pi_\tau \cong \det^{(2-k_\tau)/2}\otimes\Sym^{k_\tau-2}\C^2$ for $\tau \in \Sigma^B_\infty$
and 
$\pi_\tau \cong \sfD_{k_\tau}$
for $\tau\in \Sigma - \Sigma^B_\infty$.

\begin{theorem}(Local Jacquet-Langlands correspondence)\label{thm: local J-L}
For any $v \in \Sigma^B$, there is a bijection $\JL_v$ between the following two sets of isomorphism classes:
\begin{align*}
  \JL_v: \left\{ \begin{matrix}
    \text{smooth irreducible repre-} \\\text{sentations of }B_v^\times
  \end{matrix}\right\} 
  \longrightarrow
  \left\{\begin{matrix}
    \text{discrete series represen-}\\ \text{tations of } \GL_2(F_v)
  \end{matrix}\right\}
\end{align*} 
In particular, if $\tau\in \Sigma^B_\infty$ and $n \geq 2$, 
then $\JL_\tau(\det^{(2-n)/2}\otimes\Sym^{n-2}\C^2) = \sfD_n$.
\end{theorem}

\begin{theorem}(Global Jacquet-Langlands correspondence)\label{thm: global J-L}
  Let $\Pi= \otimes'_v \Pi_v \in \CC^B_k$.
  Then there is an injection $\JL_k^B: \CC^B_k \hookrightarrow \CC_k$ such that 
  $\JL_k^B(\Pi) = \pi = \otimes'_v \pi_v$,
  where
  $\pi_v \cong \Pi_v$ for $v\not\in \Sigma^B$, 
  and $\pi_v \cong \JL_v(\Pi_v)$ for $\tau\in \Sigma^B$.
  The image is $\{\pi = \otimes'_v \pi_v \mid \pi_v \text{ is a discrete series 
  representation of } \GL_2(F_v), \forall v\in \Sigma^B\}$.
\end{theorem}

\subsection{Goren-Oort stratification}
We recall the description of the Goren-Oort stratification on mod $p$ Hilbert modular varieties established in \cite{TX16}.

\subsubsection{Unitary and quaternionic Shimura varieties}

We first recall the relations between Hil\-bert modular varieties, unitary Shimura varieties and quaternionic Shimura varieties (\emph{cf.}\,\cite{DKS}).
Let $F'/F$ be a quadratic CM extension  such that prime $v$ splits in $F'$, for all $v|p$ in $F$, and $\Sigma' = \Sigma_{F'} := \{F' \rightarrow \overline{\Q}\}$.
Let $\ttS \subset \Sigma$ be a subset of places of $F$ with even cardinality and $B_\ttS$ be the quaternion algebra over $F$ ramified at exactly the places in $\ttS$, and $D_\ttS := B_\ttS  \otimes F'$.
If $\ttS = \varnothing$, then $B_\ttS^\times = \GL_2(F)$.
Choose $\widetilde{\ttS} \subset \Sigma'$ that contains exactly one element extending each element of $\ttS$.
Let $G_\ttS$ be the algebraic group over $\Q$ defined by 
$G_\ttS(R) = (B_\ttS \otimes R)^\times$
and 
$G'_\ttS$ be the algebraic group defined by
$G'_\ttS(R) = \{ g \in D_\ttS \otimes R \mid g\bar{g} \in (F \otimes R)^\times \}$,
where $g \mapsto \bar{g}$ is the
anti-involution on $D_\ttS$.
We denote $T_{F} := \Res_{F/Q}\G_m$ and $T_{F'} := \Res_{F'/Q}\G_m$. 
We can identify $G_\ttS'$ with the quotient $(G_\ttS \times T_{F'})/T_F$
and denote the projection map 
$\kappa: G_\ttS \times T_{F'} \rightarrow G'_\ttS$.

Let $U' = (U')^p U'_p$ be a sufficiently small open compact subgroup of $G_\ttS'(\A_\bff)$ that is maximal at $p$ in the sense that  
$U'_p = \kappa(\CO_{B_\ttS,p}^\times \times \CO_{F',p}^\times)$.
Let $\widetilde{Y}_{U'}(G'_\ttS)$ be the representing scheme of the moduli problem in \cite[\S 2.2.2]{DKS} and $Y_{U'}(G'_\ttS)$ the quotient of $\widetilde{Y}_{U'}(G'_\ttS)$ by $\CO_{F, (p),+}^\times$.
The scheme $Y_{U'}(G'_\ttS)$ is smooth and quasi-projective over $\CO$ and it is the integral model for the unitary Shimura variety associated to $G'_\ttS$.
We have a $G'_\ttS(\Afp)$-action on $\varprojlim_{U'}Y_U'(G'_\ttS)$ such that for $g\in G_\ttS(\Afp)$ and sufficiently small open compact groups $g^{-1}U_1' g \subset U_2'$,
we have $\rho'_{\ttS, g}: Y_{U'_1}(G'_\ttS) \rightarrow Y_{U'_2}(G'_\ttS)$.

Now consider the automorphic line bundles on ${Y}_{U'}(G'_\ttS)_E$ as in \cite[\S 3.1.2]{DKS}. 
Let $S' = \widetilde{Y}_{U'}(G'_\ttS)_E$ and $s: A' \rightarrow S'$ the pull-back of universal abelian variety. 
We have an exact sequence of right $\CO_{D_\ttS} \otimes \CO_{S'}$-modules
$$
0 \rightarrow s_*\Omega_{A'/S'} \rightarrow \CH_\dR^1(A'/S')\rightarrow R^1s_*\CO_{A'} \rightarrow 0.
$$
We also have decompositions
$\CH_\dR^1(A'/S') = \bigoplus_{\tau'\in\Sigma'} \CH_\dR^1(A'/S')_{\tau'}$,
$s_*\Omega_{A'/S'}  = \bigoplus_{\tau'\in\Sigma'} \tilde{\omega}_{\tau'}$,
and 
$R^1s_*\CO_{A'}  = \bigoplus_{\tau'\in\Sigma'} \tilde{\upsilon}_{\tau'}$,
where $\CH_\dR^1(A'/S')_\tau$ is a locally free $\CO_{S'}$-module of rank four 
and $\tilde{\omega}_{\tau'}$ is locally free of rank two if $\tau'|_{F} \notin\ttS$ and rank zero (resp.\,four) if $\tau'\in \widetilde{\ttS}$(resp.\,$\tau'\notin \widetilde{\ttS}$).
Denote $e_0 :=  \left(\begin{smallmatrix}
  1 & 0 \\ 0 & 0
\end{smallmatrix}\right)$.
Let $\CH_\dR^1(A'/S')_{\tau'}^\circ := \CH_\dR^1(A'/S')_{\tau'} e_0$, then it is of rank two.
If $\tau \not\in {\ttS}$,
then $\tilde{\omega}_{\tau'}$ is locally free of rank two;
if $\tau' \in \tilde{\ttS}$,
then $\tilde{\omega}_{\tau'}$ is of rank zero;
otherwise, $(\tau')^c \in \tilde{\ttS}$ and then $\tilde{\omega}_{\tau'}$ is of rank four.
Let $\tilde{\omega}_{\tau'}^\circ := \tilde{\omega}_{\tau'} e_0$ and $\tilde{\upsilon}_{\tau'}^\circ := \tilde{\upsilon}_{\tau'} e_0$.
The vector bundle $\CH_\dR^1(A'/S')_{\tau'}$ 
(resp.\,$\CH_\dR^1(A'/S')_{\tau'}^\circ$, $\tilde{\omega}_{\tau'}^\circ$, and $\tilde{\upsilon}_{\tau'}^\circ$) 
descends to $\CV_{\tau'}$ 
(resp.\,$\CV_{\tau'}^\circ$, $\omega_{\tau'}^\circ$, and ${\upsilon}_{\tau'}^\circ$) on ${Y}_{U'}(G'_\ttS)_E$ if $U'$ is sufficiently small.
Denote $\delta_{\tau'} := \wedge^2 \CV_{\tau'}^\circ$.

Let $U \subset G_\ttS(\A_\bff)$ be an open compact subgroup containing $\CO_{B,p}^\times$.
We construct the quaternionic Shimura variety $Y_U(G_\ttS)$ associated to $G_\ttS$ as in \cite[\S 2.3.3]{DKS}, which is also smooth and quasi-projective over $\CO$.
In particular, we have an inclusion
$$
i: Y_U(G_\ttS)_{\Fpbar} \hookrightarrow Y_U'(G'_\ttS)_{\Fpbar}.
$$
that is compatible with the Hecke action. 

Choose a prime $\tilde{v}$ of $F'$ dividing $v$ for each prime $v$ of $F$ that divides $p$.
Let $\Sigma_{F',\tilde{v}} \subset \{\CO_{F'} \rightarrow \Fpbar\}$ be the set of the embeddings that factor through $\CO_{F'}/\tilde{v}$, 
and let $\tilde{\Sigma} :=  \coprod_{v|p} \Sigma_{F',\tilde{v}} \subset \Sigma'$.
Define $\CV^B_\tau := i^* \CV_{\tilde{\tau}}^\circ$ for $\tau \in \Sigma$ and its extension $\tilde{\tau}\in \tilde{\Sigma}$.
Then $\CV^B_\tau$ is locally free of rank two on ${Y}_{U}(G_\ttS)_{\Fpbar}$.
Similarly we define $\delta^B_\tau$, $\omega^B_\tau$ and $\upsilon^B_\tau$ for $\tau \in \Sigma$,
where $\delta^B_\tau$ is a line bundle, and if $\tau\not\in \ttS$, then $\omega^B_\tau$ and $\upsilon^B_\tau$ are line bundles.

When $\ttS = \varnothing$, we have the isomorphism between canonical models of Hilbert modular variety $Y_U$ and the quaternionic Shimura variety $Y_U(G_\varnothing)$ associated to $G_\varnothing = \Res_{F/\Q}\GL_2$ (\cite[\S 2.3.4]{DKS}) 
$$ \iota : Y_{U,\Fpbar} \xrightarrow{\sim} Y_U(G_\varnothing)_{\Fpbar},$$
that is compatible with Hecke action in the sense that, for $g^{-1}U_1 g \subset U_2$ the following diagram commutes
\[
  \begin{tikzcd}
  Y_{U_1,\Fpbar} & Y_{U_1}(G_\varnothing)_{\Fpbar}\\
  Y_{U_2,\Fpbar} & Y_{U_2}(G_\varnothing)_{\Fpbar}
  \arrow[from=1-1, to=1-2]
	\arrow["{\rho_g}"', from=1-1, to=2-1]
	\arrow["{\rho_{\varnothing, g}}", from=1-2, to=2-2]
	\arrow[from=2-1, to=2-2].
\end{tikzcd}
\]

\subsubsection{Stratification}
We now recall the results proved in \cite{TX16}, where they show the Goren-Oort stratum $Y_{U'}(G'_{\mathtt{S}})_{\Fpbar, \mathtt{T}}$ is a $(\P^1)^N$-bundle over some quaternionic Shimura variety and some integer $N$ depending on the set $\ttT \subseteq \Sigma - (\ttS\cap\Sigma)$.
For simplicity, we only describe the result when $\ttS = \varnothing$ and $\ttT = \{\tau\}$, which will be enough for our purpose.

Let $\tilde{\tau}\in \Sigma'$ extend $\tau$ and 
$h_{\tilde{\tau}} 
\in \Gamma(Y_{U'}(G'_\varnothing)_{E}, 
(\omega^\circ_{\tilde{\tau}})^{\otimes (- 1)}
\otimes 
(\omega^\circ_{\Fr^{-1}\circ\tilde{\tau}})^{\otimes p})$ 
be the $\tilde{\tau}$-partial Hasse
invariant 
which arises from the pull-back of Verschiebung.

\begin{theorem}\cite[Theorem 5.8]{TX16}
  Let $U'$ be a sufficiently small open compact subgroup of $G_\varnothing'(\A_\bff)$ maximal \footnote{The level $U = U^p U_p \subset G(\A_{\bff})$ maximal at $p$ means that the $U_p$ is the maximal compact subgroup of $G(F_p)$. For $G = G'_\varnothing$, $U'_p$ is conjugate to $(\GL_2(\CO_{F,p}) \times\CO_{F', p}^\times)/\CO_{F,p}^\times $; for quaternion algebra $D$ splitting at $p$, $U^D_p$ is conjugate to $\GL_2(\CO_{F,p})$} at $p$.
  Let $Z'_{\{\tilde{\tau}\}}\subset Y_{U'}(G'_\varnothing)_{E}$ be the vanishing locus of $h_{\tilde{\tau}}$.
  Assume $\Fr^{-1}\circ\tau \neq \tau$.
  Then we have an isomorphism 
  $\varphi' : Z'_{\{\tilde{\tau}\}} \rightarrow \P^1_{Y'}$
  which identifies 
  $Z'_{\{\tilde{\tau}\}}$ with 
  the projectivization of the locally free rank two vector bundle
  $\CV^\circ_{\Fr^{-1}\circ \tilde{\tau}}$ 
  over $Y'= Y_{U_{\tilde{\tau}}'}(G'_{\{\Fr^{-1}\circ \tilde{\tau}, \tilde{\tau}\}})_{E}$,
  for some level structure $U_{\tilde{\tau}}'$ maximal at $p$.
  The isomorphism is compatible with Hecke action in the sense that, for $g\in G'_\varnothing(\Afp)$ such that $g^{-1}U'_1 g \subset U'_2$, we have the following commutative diagram
  \[
  \begin{tikzcd}
  Z'_{1, \{\tilde{\tau}\}} & \P^1_{Y'_1}\\
  Z'_{2, \{\tilde{\tau}\}} & \P^1_{Y'_2},
  \arrow[from=1-1, to=1-2]
	\arrow["{\rho'_{\varnothing, g}}"', from=1-1, to=2-1]
	\arrow["{\rho'_{\{\Fr^{-1}\circ\tilde{\tau}, \tilde{\tau}\}, g}}", from=1-2, to=2-2]
	\arrow[from=2-1, to=2-2]
\end{tikzcd}
\]
  where $Z'_{i, \{ \tilde{\tau} \}} \subset Y_{U'_i}(G'_\varnothing)_E$ is the zero locus and $Y'_i$ is the corresponding unitary Shimura variety whose $\P^1$-bundle isomorphic to the zero locus for $i = 1,2$.
  Moreover, we have
  $\varphi'_* H_{\dR}^1(A'/S')|_{Z'_{\{\tilde{\tau}\}}} \cong (\pi')^*\CV^\circ_{\Fr^{-1}\circ \tilde{\tau}}$, where $A'$ is the universal abelian variety over $S' = Y_{U'}(G'_\varnothing)_E$ and $\pi': \P^1_{Y'} \rightarrow Y'$ is the natural projection. 
  We also have
  $\varphi'_* \omega^{\circ}_{\Fr^{-1}\circ\tilde{\tau}}|_{Z'_{\{\tilde{\tau}\}}} \cong \wedge^2 ((\pi')^*\CV^\circ_{\Fr^{-1}\circ \tilde{\tau}}) (-1)$, which is the tautological line bundle on $\P^1_{Y'}$, and
  both isomorphisms are Hecke-equivariant.
\end{theorem}  

Using the method explained in \cite[\S5.3, \S5.4]{DKS} (they work on Iwahori level, but the same method applies to level prime-to-$p$),
we can obtain the stratification result over the special fibre of Hilbert modular variety
and the comparison of vector bundles (stated and proved as in \cite[Theorem 5.2]{TX16} and \cite[\S6.6.2]{Gabriel-thesis}, respectively).

\begin{theorem}\label{thm: strata-HMV}\cite[Theorem 5.2]{TX16}\cite[\S6.6.2]{Gabriel-thesis}
  Let $U = U^p  \GL_2(\CO_{F,p})$ be a sufficiently small open compact subgroup of $\GL_2(\AFf)$.
  Let $Z_{\{{\tau}\}}\subset Y_{U, \Fpbar}$ be the vanishing locus of the partial Hasse invariant $\Ha_\tau$.
  Assume $\Fr^{-1}\circ\tau \neq \tau$.
  Then we have an isomorphism
  $\varphi : Z_{\{{\tau}\}} \rightarrow \P^1_{Y^B}$,
  which identifies $Z_{\{\tau\}}$ with
  the projectivization of the locally free rank two vector bundle
  $\CV^B_{\Fr^{-1}\circ\tau}$
  over $Y^B = Y_{U^B}(G_{\{\Fr^{-1}\circ\tau, \tau\}})_{\Fpbar}$,
  for some level structure $U^B\subset G_\varnothing(\A_\bff)$ maximal at $p$.
  The isomorphism is compatible with Hecke action in the sense that, for $g\in \GL_2(\AFfp)$ such that $g^{-1}U_1 g \subset U_2$, we have the following commutative diagram
  \[
  \begin{tikzcd}
  Z_{1, \{{\tau}\}} & \P^1_{Y^B_1}\\
  Z_{2, \{{\tau}\}} & \P^1_{Y^B_2},
  \arrow[from=1-1, to=1-2]
	\arrow["{\rho_g}"', from=1-1, to=2-1]
	\arrow["{\rho_{\{\Fr^{-1}\circ\tau, \tau\}, g}}", from=1-2, to=2-2]
	\arrow[from=2-1, to=2-2]
\end{tikzcd}
\]
  where $Z_{i, \{ \tau \}} \subset Y_{U_i, \Fpbar}$ is the zero locus and $Y_i$ is the corresponding unitary Shimura variety whose $\P^1$-bundle isomorphic to the zero locus for $i = 1,2$.
  Moreover, we have 
  $\varphi_* H_{\dR}^1(A/S)|_{Z_{\{\tau\}}} \cong \pi^*\CV^\circ_{\Fr^{-1}\circ \tilde{\tau}}$, where $A$ is the universal abelian variety over $S = Y_{U,\Fpbar}$ and $\pi: \P^1_{Y^B} \rightarrow Y^B$ is the natural projection.
  We also have
  $\varphi_*\omega_{\Fr^{-1}\circ\tau}|_{Z_{\{\tau\}}}\cong \wedge^2 (\pi^*\CV^B_{\Fr^{-1}\circ {\tau}}) (-1)$, which is the tautological line bundle on $\P^1_{Y^B}$, and
  both isomorphisms are Hecke-equivariant.
\end{theorem}

\section{Weight part of Serre's conjecture}\label{sec: wt_part_Serre}

We recall the weight part of Serre's conjecture for totally real fields $F$ formulated in \cite{Buzzard-Diamond-Jarvis} and a geometric variant in \cite{Diamond-Sasaki}.

\subsection{The Buzzard-Diamond-Jarvis Conjecture}\label{subsec: BDJ}
Let
$
  \rho: G_F \rightarrow \GL_2(\Fpbar)
$
be a continuous, irreducible representation and $V$ a finite dimensional $\Fpbar$-representation of $\GL_2(\COF/p)$.
For $v\mid p$,
let $\F_v = \CO_F/v$,
$f_v = [\F_v: \Fp]$, 
and $\Sigma_v$ be the subset of embeddings $\{\F_v \rightarrow \Fpbar\}$. 
Recall we have 
$\Sigma = \coprod_{v|p} \Sigma_v$.
Let $(k,l)\in \Z^\Sigma_{\geq 2}\times \Z^\Sigma$ and we call such weights \emph{algebraic weights}.
Let $V_{k,l}$ be the representation of 
$\GL_2(\CO_F/p) = \prod_{v|p} \GL_2(\F_v)$ of the form
\begin{align*}
  V_{k,l} := V_{k,l, E} \otimes \Fpbar 
  = \bigotimes_{\tau\in \Sigma} {\det}_{[\tau]}^{l_\tau} \otimes \Sym_{[\tau]}^{k_\tau -2} \Fpbar^2,
\end{align*}
where 
$\GL_2(\F_v)$ acts on $[\tau]$-part through the homomorphism 
$\GL_2(\F_v) \rightarrow \GL_2(\Fpbar)$ induced by $\tau$ for every $v|p$ and $\tau\in \Sigma_v$.

The irreducible representations of $\GL_2(\COF/p)$ over $\Fpbar$, which we call \emph{Serre weights}, are of the form $V_{k,l}$ with $2 \leq k_\tau \leq p+1$.
Two Serre weights $V_{k,l}$ and $V_{k',l'}$ are isomorphic if and only if 
$k = k'$ and 
$\sum_{i = 0}^{f_\tau}l_{\Fr^i\circ \tau}p^i = \sum_{i = 0}^{f_\tau}l'_{\Fr^i\circ \tau}p^i$ mod $p^{f_\tau} -1$ for every $\tau\in \Sigma$, 
where $f_\tau = f_v$ for the prime $v$ such that $\tau\in \Sigma_v$.

Let $B$ and $U$ be as in \Cref{subsec: def&indef}.
Let $M^B_{k,l}(U; \Fpbar) = H^0(Y^B_U, V_{k,l+1})$ (definite case), or $H^1(Y^B_U, \CF_{V_{k,l+1}})$ (indefinite case). We have a commuting
family of Hecke operators $T_v$ and $S_v$ for $v\not\in \Sigma_\bff^B$ such that $v\nmid p$ and $\CO_{B,v} \subset U$. 

Let $\rho: G_F\rightarrow \GL_2(\Fpbar)$ be a continous representation.
Let $T$ be a finite set of finite places of $F$ consisting of primes $v$ above $p$, 
the primes where $\rho$ or $B$ is ramified and 
primes $v$ such that $\GL_2(\CO_{F,v}) \not\subset U$.
Let ${\T}^{\univ}_T = \CO[T_v, S_v]_{v\not\in T}$.
Let $\frakm_\rho \subset \T^{\univ}_T$ be the maximal ideal
generated by the following elements
\begin{align*}
  T_v - \tr(\rho(\Frob_v)), \  \Nm_{F/\Q}(v)S_v - \det(\rho(\Frob_v))
\end{align*}
for all $v\not\in T$.
Its residue field is $E$.

\begin{definition}\label{def: mod of V}
  We say an irreducible representation $\rho : G_F \rightarrow \GL_2(\Fpbar)$ is modular for $B$ of weight $V_{k,l}$ if 
  $$
  M_{k,l}^B(U,\Fpbar)[\frakm_{\rho^\vee(-1)}] \neq 0,
  $$ 
  for some open compact subgroup $U$.

  We say $\rho$ is compatible with $B$ if $\rho|_{G_{F_v}}$ is either
  irreducible, or is a twist of an extension of the trivial character by the cyclotomic
  character, for any $v\in \Sigma^B_\bff$.

  We say $\rho$ is modular of weight $V_{k,l}$ if $\rho$ is compatible and modular for $B$ of weight $V_{k,l}$, for some definite or indefinite quaternion algebra $B$ as in \S\ref{subsec: def&indef}.
\end{definition}

\begin{remark}
  Our definition for being modular of weight $\sigma$ is different from \cite{Buzzard-Diamond-Jarvis}  because of the different conventions for the maximal ideals. 
  The maximal ideal $\frakm^{BDJ}_\rho$  
  in \cite[\S 4]{Buzzard-Diamond-Jarvis} 
  is generated by
  \begin{align*}
    T_v - S_v \tr(\rho(\Frob_v)), \ \Nm_{F/\Q}(v) - S_v\det(\rho(\Frob_v))
  \end{align*}
  for all $v\not\in T$.
  Hence, $\frakm^{BDJ}_\rho = \frakm_{\rho^\vee(-1)}$.
  Our convention is also different from the one in \cite{Gee-Kisin}, as
  we will see that the fundamental characters we take are the inverse of the ones in \cite{Gee-Kisin}.
\end{remark}

\begin{lemma}\label{lem: mod_of_JH}
  Let $\rho: G_F \rightarrow \GL_2(\Fpbar)$ be continuous, irreducible and totally odd.
  Then $\rho$ is modular of weight $V$ if and only if $\rho$ is modular of weight $W$ for some Jordan-Holder factor $W$ of $V$.
\end{lemma}

\begin{proof}
  By \cite[Proposition 2.5]{Buzzard-Diamond-Jarvis} for indefinite case
  and \cite[(4)]{Diamond-Reduzzi} for definite case (note that  the groups $\Gamma_i$ there are trivial if $U$ is sufficiently small), 
  we have $\rho$ is modular of weight $V$ if and only if $\rho$ is modular of weight $W$ for some Jordan-Holder factor $W$ of $V$. 
\end{proof}

In \cite{Buzzard-Diamond-Jarvis}, they define a set $W^\BDJ(\rho)$ of Serre weights and predict that $\rho$ is modular of a Serre weight $\sigma$ if and only if $\sigma \in W^\BDJ(\rho)$.
We now recall the definition of $W^\BDJ(\rho)$ in \cite{Buzzard-Diamond-Jarvis}.
Let $\varepsilon_\tau$ denote the fundamental character on $I_{F_v}$ corresponding to $\tau$, which is the 
composite of the maps induced by  $\tau$ and $\textrm{Art}^{-1}$ in local class field theory:
$$
\varepsilon_\tau: I_{F_v}\rightarrow\CO_{F,v}^\times\rightarrow \Fpbar.
$$
In our convention\footnote{In \cite{Gee-Kisin}, their convention is  $\prod_{\tau\in\Sigma_v}\varepsilon_\tau = \varepsilon^{-1}$. }, 
$\prod_{\tau\in\Sigma_v}\varepsilon_\tau = \varepsilon$, where $\varepsilon$ is the mod $p$ cyclotomic character.

Let Serre weight $V_{k,l} = \otimes_{v|p} V_{k,l}^{(v)}$, where $V^{(v)}_{k,l} := \bigotimes_{\tau\in \Sigma_v} {\det}^{l_\tau}_{[\tau]} \otimes \Sym^{k_\tau-2}_{[\tau]}$.

\begin{definition}\label{def: WBDJ}
  Let $W^\BDJ(\rho) := \prod_{v|p} W_v(\rho)$, where $W_v(\rho)$ are defined as follows for any $v|p$.
\begin{enumerate}
  \item  Suppose $\rho|_{G_{F_v}}$ is reducible. Then 
  \begin{align}\label{eq: reducible-case}
    \rho|_{G_{F_v}} \simeq \left(\begin{matrix}
      \chi_1 & c \\ 0 & \chi_2
    \end{matrix}\right),
  \end{align}
  where $c \in H^1(G_{F_v}, \Fpbar(\chi_1\chi_2^{-1}))$ is an extension class.
  Following the definitions in \cite[\S 3.2]{Buzzard-Diamond-Jarvis}, 
  we define $W'_v(\chi_1, \chi_2)$ be the set of pairs $(V_{b,d}^{(v)},J)$ such that 
  $$\chi_1|_{I_{F_v}} = \prod_{\tau\in J}\varepsilon_{\tau}^{b_\tau-1}\prod_{\tau\in \Sigma_v}\varepsilon_{\tau}^{d_\tau}, \text{ and } 
  \chi_2|_{I_{F_v}} = \prod_{\tau\not\in J}\varepsilon_{\tau}^{b_\tau-1}\prod_{\tau\in \Sigma_v}\varepsilon_{\tau}^{d_\tau}.$$ 
  Let $L_{V,J}$ be a certain
  distinguished subspace of $H^1(G_{F_v}, \Fpbar(\chi_1\chi_2^{-1}))$ defined as in \cite[\S3.2]{Buzzard-Diamond-Jarvis}.
  Then $V \in W_v(\rho)$ if and only if  $(V,J) \in W'_v(\chi_1, \chi_2)$ and $c\in L_{V,J}$ for some $J$.
  \item Suppose $\rho|_{G_{F_v}}$ is irreducible. 
  Let $K$ be a quadratic extension of $F$ in which $v$ is inert. 
  Let $\Sigma'_v$ be the set of embeddings of the residue field of $K_v$ into $\Fpbar$. 
  Then $V_{k,l}^{(v)} \in W_v(\rho)$ if and only if there is a subset $J'\subset \Sigma_v'$ containing exactly one element extending each element of $\Sigma_v$, such that
  \begin{align}\label{eq: irreducible-case}
    \rho|_{I_{F_v}} \simeq \prod_{\tau\in\Sigma_v} \varepsilon_\tau^{l_\tau}
    \left(\begin{matrix}
      \prod_{\tau'\in J'} \varepsilon_{\tau'}^{k_\tau-1} & 0 \\
      0 & \prod_{{\tau'} \not\in J'} \varepsilon_{\tau'}^{k_\tau-1}
    \end{matrix}\right).
  \end{align}
\end{enumerate}  

\end{definition}

By definition, $W^\BDJ(\rho)$ is non-empty and only depends on $\rho|_{I_{F_v}}$ for all $v|p$.

\begin{remark}\label{rmk: J=Sigmav}
  Let $p$ be an odd prime and $v\mid p$.
  Suppose $\rho|_{G_{F_v}} \simeq \left(\begin{smallmatrix}
    \chi_1 & * \\ 0 & \chi_2
  \end{smallmatrix}\right)$ and $(V, \Sigma_v) \in W'_v(\chi_1, \chi_2)$.
  Then $L_{V,\Sigma_v} = H^1(G_{F_v}, \Fpbar({\chi}_1{\chi}_2^{-1}))$ unless $V$ is trivial up to a twist, 
  in which case $L_{V,\Sigma_v}$ has codimension one.
  We have $V\in W_v(\rho)$ unless $V = \bigotimes_{\tau\in \Sigma_v}{\det}_{[\tau]}^{d_\tau}$ for some $d\in \Z^{\Sigma_v}$, in which case $\bigotimes_{\tau\in\Sigma_v} \det_{[\tau]}^{d_\tau} \Sym_{[\tau]}^{p-1} \in W_v(\rho)$.
\end{remark}

Under the assumption that $\rho$ is modular of some weight,
it is conjectured by Buzzard-Diamond-Jarvis that a Serre weight $\sigma$ belongs to $W^\BDJ(\rho)$ if and only if $\rho$ is modular of weight $\sigma$, which we refer to as \emph{BDJ conjecture}.
BDJ conjecture holds if $F = \Q$ (\emph{cf.}\,\cite[Theorem 3.17]{Buzzard-Diamond-Jarvis}).
For general totally real fields, BDJ conjecture was proved by a series of works, with a mild Taylor-Wiles assumption.

\begin{theorem}(BDJ conjecture, \cite{Gee-Liu-Savitt-Unitary, Gee-Liu-Savitt-GL2, Gee-Kisin, Newton14})\label{thm: BDJ-conj}
    Let $p$ be an odd prime and $F$ a totally real field in which $p$ is unramified. 
    Let $\rho: G_F \rightarrow \GL_2(\Fpbar)$ be a continuous representation that is modular of some weight. 
    If $\rho$ is modular of a Serre weight $\sigma$,
    then $\sigma\in W^\BDJ(\rho)$.

    Now assume that $\rho|_{G_{F(\zeta_p)}}$ is irreducible, and  
    when $p = 5$, further assume that the projective image of $\rho|_{G_{F(\zeta_5)}}$ in $\mathrm{PGL}(\overline{\F}_5)$ is
    not isomorphic to $A_5$.
    If $\sigma\in W^\BDJ(\rho)$, then
    $\rho$ is modular of weight $\sigma$.
\end{theorem}

\begin{remark}
    The technical condition on $\rho|_{G_{F(\zeta_p)}}$ in \Cref{thm: BDJ-conj} is slightly relaxed than the adequacy condition (see \cite[Proposition A.2.1]{BLGG13}).
    Also, note that we do not require this technical condition on $\rho|_{G_{F(\zeta_p)}}$ for one of the directions of the BDJ conjecture to hold.
\end{remark}

\begin{definition}
  Let $(k,l)\in \Z^\Sigma_{\geq 2}\times \Z^\Sigma$. 
  We say a continuous, irreducible, totally odd representation $\rho: G_F \rightarrow \GL_2(\Fpbar)$ is algebraically modular of weight $(k, l)$ if it is modular of weight
  $V_{k, -k-l}$ in the sense of \Cref{def: mod of V}.
  Equivalently, there is a definite or indefinite quaternion algebra $B$ and an open compact subgroup $U$ as in \S\ref{subsec: def&indef} such that $M^B_{k,l}(U, \Fpbar)[\frakm_\rho] \neq 0$.

\end{definition}

\begin{remark}
  To reconcile our convention with \cite{Diamond-Sasaki} (which has a normalised factor $||\det(g)||$ in the definition of Hecke actions), we use $V_{k, -k-l}$ instead of $V_{k, 1-k-l}$ as in \cite[Definition 7.5.1]{Diamond-Sasaki}.
\end{remark}

\subsection{Geometric variant}\label{subsec: geom variant}

It was proved in \cite{Diamond-Sasaki} that we can associate Galois representations to mod
$p$ Hilbert modular eigenforms of arbitrary weight.

\begin{theorem}[{\cite[Theorem 6.1.1]{Diamond-Sasaki},\,\cite{Diamond-Sasaki_inprep}}]
  Let $S$ be a finite set containing primes $v$ in $F$ such that $v|p$ or $\GL_2(\CO_{F,v})\not\subset U$.
  Let $f\in M_{k,l}(U; E)$ be an eigenform for $T_v$ and $S_v$, $v\not\in S$.
  Let $T_v f = a_v f$ and $S_v f = d_v f$, where $a_v,d_v\in E$.
  Then there is a Galois representation\footnote{Due to the difference in convention of Hecke operators, the Galois representation here is the one in \cite[Theorem 6.1.1]{Diamond-Sasaki} twisting by the inverse of the mod $p$ cyclotomic character.}
  $
  \rho_f: G_F \rightarrow \GL_2(E)
  $
  that is unramified at $v\not\in S$ and the characteristic polynomial of $\rho_f(\Frob_v)$ is 
  $$
  X^2 - a_vX + d_v\Nm_{F/\Q}(v).
  $$
\end{theorem}

\begin{definition}\label{def: geom-modular}
  Let $(k,l)\in\Z^\Sigma\times\Z^\Sigma$.
  We say a continuous, irreducible, totally odd representation $\rho: G_F \rightarrow \GL_2(\Fpbar)$ is geometrically modular of weight $(k, l)$ if there is an open compact $U\subset \GL_2(\widehat{\CO}_F)$ and an eigenform form $f\in M_{k,l}(U, E)$ such that $\rho$ is equivalent to the extension of scalars of $\rho_f$.
\end{definition}

We also consider the Galois representation associated to a twisted eigenform.
For any ideal $\frakM\subset \COF$, we denote $V_{\frakM}: = \ker(\widehat{\CO}_F^\times \rightarrow (\CO_F/\frakM)^\times)$. 
Let $k, l, l' \in \Z^\Sigma$ and $\frakm, \frakn$ be the ideals of $\COF$ prime to $p$ such that $\frakm | \frakn$.
We say a character
$$\xi: \{\AFf^{\times} \mid a_p\in \CO_{F,p}^\times\}/ V_\frakm \rightarrow E^\times,$$ 
is a character of weight $l'$ if
$\xi(a) = \bar{a}^{l'}$ for every $a\in F^\times \cap \CO_{F,p}^\times$.
Then we define $\xi': \A_F^\times /F^\times F_{\infty, + }^\times V_{\frakm p} \rightarrow E^\times$ by 
$\xi'(\alpha z a) = \xi(a)\bar{a}_p^{-l'}$, where $\alpha\in F^\times, z\in F_{\infty,+}^\times, a\in \AFf^\times$ with $a_p\in \CO_{F,p}^\times$.
Let $\rho_{\xi'}$ be the corresponding character by class field theory and that 
$\rho_{\xi'}(\Frob_v) = \xi(\varpi_v)$, for $v\nmid \frakm p$.
Suppose $f \in M_{k,l}(U_1(\frakn); E)$ is an eigenform for $T_v$ and $S_v$, $v\nmid p\frakn$.
We can define a twisted form 
$f_\xi\in  M_{k,l+l'}(U_1(\frakn\frakm^2); E)$ as in \cite[Lemma 10.3.2]{Diamond-Sasaki}.
Moreover, the associated Galois representation $\rho_{f_\xi}$ satisfies $\rho_{f_\xi} \simeq \rho_{\xi'}\otimes \rho_{f}$. In particular 
$$
\rho_{f_\xi} (\Frob_v) = \xi(\varpi_v)\rho(\Frob_v) \text{ for } v\nmid p\frakn\frakm
$$
$$  
\rho_{f_\xi}|_{I_{F_v}} \simeq
\Big(\prod_{\tau\in \Sigma_v} \varepsilon_\tau^{-l_\tau'} \Big)
\otimes \rho_{f}|_{I_{F_v}} \text{ for } v\mid p.
$$

\subsubsection{Normalised eigenform}

We recall the definition of (stabilised, strongly stabilised) normalised eigenform. 

\begin{definition}\cite[Definition 10.5.1]{Diamond-Sasaki}
  Let $(k,l)\in \Z^{\Sigma}_{\geq 2}\times \Z^{\Sigma}$. 
  We say $f \in M_{k,l}(U_1(\frakn); E)$ is a normalised eigenform if 
  \begin{enumerate}
    \item $r_1^{\CO_F}(f) = 1$, where we write $r_m^J(f)$
    for the coefficient of $q^m$ in the $J$-component of its $q$-expansion as in \eqref{eq: q-expn}.
    \item $f$ is an eigenform for $T_v$ for $v\nmid p$ and $S_v$ for $v\nmid p\frakn$.
    \item For any character $\xi$ of weight $-l-\mathbf{1}$ and conductor $\frakm$ prime ot $p$, $f_\xi$ is an eigenform for $T_v$ for all $v|p$.
  \end{enumerate}
\end{definition}
In particular, if $f$ is a normalised eigenform, then $\Theta_\tau(f) \neq 0$ for all $\tau\in\Sigma$.

\begin{proposition}\cite[Lemma 10.4.1, Proposition 10.5.2]{Diamond-Sasaki} \label{lem: Theta(f)nonzero}
  If $\rho$ is geometrically modular
  of weight $(k, l)$, 
  then $\rho$ arises from an eigenform $f\in M_{k,l}(U_1(\frakn); E)$
  for some $\frakn$ prime to $p$, 
  such that
  $\Theta_\tau(f) \neq 0$ for all $\tau \in \Sigma$.
  If in addition that $k_\tau \geq 2$ for all $\tau$, then we can take the eigenform $f$ to be a normalised eigenform.
\end{proposition}

\begin{definition}\cite[Definition 10.6.1]{Diamond-Sasaki}
  Let $(k, l) \in \Z^\Sigma_{\geq 2} \times \Z^\Sigma$.
  We say a normalised eigenform $f \in M_{k,l}(U_1(\frakn); E)$ is a stabilised eigenform if 
  $r^J_m(f) = 0$ for all $(m,J)$ such that $m \in J_+^{-1}$ and $mJ$ not prime to $\frakn$. Equivalently, $T_v f = 0$ for all $v|\frakn$.
\end{definition}

\begin{definition}\cite[Definition 10.6.4]{Diamond-Sasaki}
  Let $(k, l) \in \Z^\Sigma_{\geq 2} \times \Z^\Sigma$.
  We say a stabilised eigenform $f \in M_{k,l}(U_1(\frakn); E)$ is
  strongly stabilised if $r^J_m(f) = 0$ for all $(m,J)$ such that $m \in J_+^{-1} \cup \{0\}$ and $mJ$ not prime to $p$.
\end{definition}

\begin{lemma}\label{lem: stabilised_f}
  Let $(k, l) \in \Z^\Sigma_{\geq 2} \times \Z^\Sigma$.
  If $\rho$ is geometrically modular
  of weight $(k, l)$, 
  then $\rho$ arises from 
  a stabilised eigenform $f\in M_{k,l}(U_1(\frakn); E)$
  for some $\frakn$ prime to $p$.
\end{lemma}

\begin{proof}
  By \Cref{lem: Theta(f)nonzero}, we have
  $\rho$ arises from a normalised eigenform $f\in M_{k,l}(U_1(\frakn); E)$ 
  for some $\frakn$ prime to $p$.
  We can choose $\frakn' \subset \frakn$ so that $\frakn'$ and $\frakn$ satisfies the conditions in
  \cite[Lemma 10.6.2]{Diamond-Sasaki}.
  In particular, we can let $\frakn' = \prod_i \frakp_i^{e_i + 2}$ if $\frakn = \prod_i \frakp_i^{e_i}$.
  Then $\rho$ arises from a stabilised eigenform in $M_{k,l}(U_1(\frakn'); E)$ by \cite[Lemma 10.6.2]{Diamond-Sasaki} (enlarging $E$ if necessary).
\end{proof}

Suppose $k_\tau \geq 2$.
If $\rho$ arises from a mod $p$ eigenform $f$ of weight $(k, -\mathbf{1})$ with non-zero eigenvalue at $v$ for any $v|p$, then $\rho_f$ has a very explicit form, 
which suggests $\rho$ has a crystalline lift of Hodge type $(k - \mathbf{1}, \mathbf{0})$.

\begin{theorem}\cite[Corollary 10.7.2]{Diamond-Sasaki}\label{thm: ord-geom-rho}
  Let $(k,l)\in\Z^\Sigma_{\geq 2} \times \Z^\Sigma$.
  Suppose $U$ is an open compact subgroup of $\GL_2(\COFhat)$ containing $\GL_2(\CO_{F,p})$
  and $S$ is a finite set of primes containing $v|p$ and $v$ such that $\GL_2(\CO_{F,p}) \not\subset U$.
  Let $v_0$ be a prime in $F$ over $p$.
  Suppose
  $f \in M_{k,l}(U; E)$ is an eigenform for $T_v$ and $S_v$ for all $v\not\in S$.
  Suppose there is a character $\xi$ of weight $-l-\mathbf{1}$ such that $T_{v_0}f_\xi = a_{v_0}f_\xi$ with $a_{v_0} \neq 0$. 
  Then (after adapting to our convention) 
  \begin{align*}
    \rho_f|_{I_{F_{v_0}}} \simeq 
    \prod_{\tau\in \Sigma_{v_0}}\varepsilon_\tau^{-k_\tau-l_\tau}
    \left(\begin{matrix}
      \prod_{\tau\in \Sigma_{v_0}}\varepsilon_\tau^{k_\tau -1} & * \\
      0 & 1
    \end{matrix}\right).
  \end{align*}
\end{theorem}

\begin{definition}
  Let $(k,l) \in \Z^\Sigma_{\geq 2}\times\Z^\Sigma$.
  We say $\rho$ is ordinarily geometrically modular at $v_0$ of weight $(k,l)$, if $\rho \simeq \rho_f$ for some
  $f\in M_{k,l}(U;E)$ as in \Cref{thm: ord-geom-rho}.
\end{definition}

\begin{corollary}\label{cor: ord-geom-alg}
  Let $(k,l) \in \Z^\Sigma_{\geq 2}\times\Z^\Sigma$.
  Suppose $k_\tau \leq p+1$ for all $\tau\in \Sigma$ and $k \neq \mathbf{2}$.
  If $\rho$ is ordinarily geometrically modular at $v_0$ of weight $(k,l)$ for some $v_0|p$, then $V_{k, -k-l}^{(v_0)}\in W_{v_0}(\rho)$.
\end{corollary}

\begin{proof}
  This is directly by \Cref{thm: ord-geom-rho} and \Cref{rmk: J=Sigmav}.
\end{proof}

\subsubsection{Geometric Serre weight conjecture}

Let $\leq_\Ha$ a partial
ordering on $\Z^\Sigma$ defined by $k' \leq_\Ha k$ if and only if $k - k' = \sum_{\tau\in \Sigma} n_\tau k_{\Ha_\tau}$, for some $n_\tau \geq 0$.
It is shown in \cite{Diamond-Kassaei_17} that if $\rho$ is geometrically modular of weight $k$, then $\rho$ is geometrically modular of a weight $k' \leq_\Ha k$ and $k'$ lies in a certain cone, which is called the \emph{minimal cone}.
\begin{definition}
  Define the minimal cone in $\Z^\Sigma$ be 
  $$
  \Xi_{\min} := \{k\in \Z^\Sigma \mid pk_\tau \geq k_{\Fr^{-1} \circ \tau}, \text{ for all } \tau\in \Sigma\},
  $$
  and let $\Ximinplus := \Ximin \cap \Z^\Sigma_{\geq 1}$.
\end{definition}

\begin{theorem}\cite[Theorem 5.1]{Diamond-Kassaei_17}\label{thm: DK-minimalcone}
  Let $k\in \Z^\Sigma$ such that $pk_\tau < k_{\sigma^{-1}\circ\tau}$ for some $\tau\in \Sigma$.
  Then we have a Hecke-equivariant isomorphism
  $$M_{k-k_{\Ha_\tau},l}(U, E) \xrightarrow{\sim} M_{k,l}(U, E) $$ induced by multiplication by $\Ha_\tau$.
\end{theorem}

Let $f\in M_{k,l}(U, \Fpbar)$ be non-zero.
There is a unique maximal element in the following set by the partial ordering $\leq_{\Ha}$ (\emph{cf}.\,\cite{Andreatta-Goren05})
\begin{align*}
  \left\{
    n = (n_\tau) \in \Z_{\geq 0}^\Sigma \mid f = g \prod_{\tau\in \Sigma}\Ha_{\tau}^{n_\tau} \text{ for some } g\in M_{k - \Sigma n_\tau k_{\Ha_\tau}, l}(U, \Fpbar)
  \right\}.
\end{align*}

\begin{definition}\label{def: min_wt_f}
  Define $$\Phi(f) := k - \Sigma_\tau n_\tau k_{\Ha_\tau}$$
  for the maximal element $n = (n_\tau)$ as above.
\end{definition}

Then \Cref{thm: DK-minimalcone} immediately implies the following result.

\begin{theorem}\cite[Corollary 8.2]{Diamond-Kassaei_17}\label{thm: wt-in-min-cone}
  For $f\in M_{k,l}(U, \Fpbar)$ non-zero, we have
  $\Phi(f)\in \Xi_{\min}$.
\end{theorem}
In the forthcoming paper \cite{Diamond-Sasaki_inprep}, it is shown that if $f$ is a non-zero cuspform, then $\Phi(f) \in \Ximinplus$.

\begin{conjecture}[Diamond-Sasaki]\label{conj: DS}
  Let $\rho: G_F \rightarrow \GL_2(\Fpbar)$ be an irreducible, continuous, totally odd representation. 
  For any $l \in \Z^\Sigma$, there exists $k_{\min} = k_{\min}(\rho, l) \in \Ximinplus$ such that 
  $\rho$ is geometrically modular of weight $(k, l)$ if and only if $k \geq_\Ha k_{\min}$.
  Furthermore, if $k \in \Ximinplus$, then $k \geq_\Ha k_\min$ if and only if $\rho|_{G_{F_v}}$ has a crystalline lift of Hodge-Tate type $(k_\tau+l_\tau, l_\tau+1)_{\tau\in\Sigma_v}$ for all $v|p$.
\end{conjecture}
The above conjecture can be viewed as a geometric Serre weight conjecture. 
The first statement is suggested by \Cref{{thm: DK-minimalcone}} and the second is suggested by the Breuil-M\'ezard Conjecture.

\section{Algebraic and Geometric modularity}\label{sec: alg_geo_mod}

In this section, we discuss the relation between the algebraic and geometric modularity. 

\begin{conjecture}[Diamond-Sasaki]\label{conj: geom-alg}
  Let $\rho: \Gal(\Fbar/F) \rightarrow \GL_2(\Fpbar)$ be an irreducible, continuous, totally odd representation and $(k,l) \in \Z^\Sigma_{\geq 2} \times \Z^\Sigma$.
  If $\rho$ is algebraically modular of weight $(k,l)$, then it is geometrically modular of weight $(k,l)$; the converse holds if in addition $k\in \Ximinplus$.
 \end{conjecture}

\begin{proposition}[\cite{Diamond-Sasaki, Diamond-Sasaki_inprep}]\label{prop: alg-geom_DS}
  Let $\rho: \Gal(\Fbar/F) \rightarrow \GL_2(\Fpbar)$ be irreducible, continuous, totally odd.
  \begin{enumerate}
    \item If $\rho$ is geometrically modular of some weight, then it is algebraically modular of some (algebraic) weight.
    \item Let $(k,l)\in \Z^\Sigma_{\geq 2}\times \Z^\Sigma$. If $\rho$ is algebraically modular of weight $(k, l)$, then $\rho$ is geometrically modular of the same weight $(k, l)$.
  \end{enumerate}
\end{proposition}

We focus on the direction from geometric modularity to algebraic modularity in the \Cref{conj: geom-alg}.
First, we show that we can reduce to the case when $\rho$ is geometrically modular of weight $(k, \mathbf{0})$ by the following lemma.

\begin{lemma}\label{lem: reduce to l=0}
  Let $k\in \Z^\Sigma_{\geq 1}$ and $l,l'\in \Z^\Sigma$.
  Let $\rho : G_F \rightarrow \GL_2(\Fpbar)$ be an irreducible, continuous, totally odd representation.
  Let $\xi: \AFfp \rightarrow \Fpbar^\times$ be a continous character such that $\xi(a) = \bar{a}^{l'}$ for $a\in \COFppx$.
  Let
  $\xi': \A_F^\times /F^\times F_{\infty, + }^\times \rightarrow \Fpbar^\times$ be as in \S\ref{subsec: geom variant} such that 
  $\xi'(\alpha z a) = \xi(a)\bar{a}_p^{-l'}$, where $\alpha\in F^\times, z\in F_{\infty,+}^\times, a\in \AFf^\times$ with $a_p\in \CO_{F,p}^\times$,
  and $\rho_{\xi'}: G_F \rightarrow \Fpbar^\times$ be the corresponding character by class field theory.
  Then
  \begin{enumerate}
    \item $\rho$ is geometrically modular of weight $(k,l)$ if and only if $\rho \otimes \rho_{\xi'}$ is geometrically modular of weight $(k, l+l')$.
    \item if $k_\tau \geq 2$ for all $\tau \in \Sigma$, then $\rho$ is algebraically modular of weight $(k,l)$ if and only if $\rho \otimes \rho_{\xi'}$ is algebraically modular of weight $(k, l+l')$.
  \end{enumerate}
\end{lemma}

\begin{proof}
  \begin{enumerate}
    \item 
    If $\rho \simeq \rho_f$ for some eigenform $f$ of weight $(k,l)$, then $\rho \otimes \rho_{\xi'} \simeq \rho_{ f_\xi}$,
    where $f_\xi$ (as in \S\ref{subsec: geom variant}) is of weight $(k, l+l')$.
    On the other hand,
    if $\rho \otimes \rho_{\xi'} \simeq \rho_{ g}$, for some eigenform $g$ of weight $(k, l+l')$,
    then 
    $\rho \simeq \rho_g \otimes \rho_{\xi'}^{-1} \simeq \rho_{g_{(\xi')^{-1}}}$, where the eigenform $g_{(\xi')^{-1}}$ has weight $(k,l)$.
    
    \item Let $U$ be sufficiently small.
    Recall we have an element $e_{\xi}^B$ (see \Cref{subsec: twisting}) determined by ${\xi}\circ \det$ and an isomorphism \eqref{eq: twist_xi_B} by multiplying $e_{\xi}^B$:
    \begin{align*}
      e^B_{\xi}: M^B_{k,l}(U;\Fpbar) \xrightarrow{\sim} M^B_{k,l+l'}(U;\Fpbar)
    \end{align*} such that 
    $T_v(e^B_\xi f) =\xi(\varpi_v) e^B_\xi T_v(f)$ and 
    $S_v(e^B_\xi f) =\xi(\varpi_v)^2 e^B_\xi S_v(f)$ for $v\not\in S$, where $S$ consists of $\Sigma^B_\bff$, primes above $p$, 
    and $v$ such that $(\CO_{B,v})^\times \not\subset U$.
    Note that
    $\rho_{\xi'}(\Frob_v) = \xi(\varpi)$ for $v\not\in S$.
    Therefore, $M^B_{k,l}(U;R)[\frakm_\rho] \neq 0$ if and only if $M^B_{k,l+l'}(U;R)[\frakm_{\rho\otimes \rho_{\xi'}}] \neq 0$.    
  \end{enumerate}
\end{proof}

Suppose $k\in \Ximinplus \cap \Z^\Sigma_{\geq 2}$ and $\rho$ is geometrically modular of weight $(k,l)$.
By the previous lemma, to show that $\rho$ is algebraically modular of weight $(k,l)$,
it suffices to show that $\rho \otimes \rho_{\xi'}$ is algebraically modular of weight $(k,\mathbf{0})$,
where $\xi': \AFfp \rightarrow E^\times$ is a character such that $\xi'(a) = a^{-l}$, for $a\in \COFppx$.

\subsection{Cohomology vanishing} 
\label{subsec: coh-vanish}

In this section, we adapt the method in \cite{Lan-Suh_parallel}
to prove a cohomology vanishing result 
(\Cref{{thm: my-vanish}}) for not necessarily parallel weights.
Recall we defined the moduli space $\tilY_U$ and Hilbert modular variety $Y_U$ over $\CO$ in \S\ref{subsec: HMV}, and described the toroidal and minimal compactifications $\tilYUtor$, $\YUtor$, $\tilYUmin$, and $\YUmin$ in \S\ref{subsub: tor&min}, respectively.   

The ampleness of line bundles on Hilbert modular surfaces has been discussed in \cite{Andreatta-Goren04} and later generalised to Hilbert modular varieties in \cite{Yang_ampleness}.

\begin{theorem}\cite[Thm. 8.1.1, Prop. 8.2.4]{Andreatta-Goren04}\cite[Thm. 20]{Yang_ampleness}\label{thm: DedingYang}
   The line bundle $\dotw_E^{\min,\otimes k}$ on $\tilY_{U, E}^{\min}$ is ample if and only if $p k_\tau > k_{\Fr^{-1}\circ\tau}$ for all $\tau\in \Sigma$.
   It is nef if and only if $p k_\tau \geq k_{\Fr^{-1}\circ\tau}$ for all $\tau\in \Sigma$.
\end{theorem}

Recall that $\CA^{\min, k,l}_E$ is a line bundle on $Y_{U, E}^{\min}$ if $U$ is sufficiently small.
We have the follow ampleness result for $\CA^{\min, k,l}_E$.

\begin{corollary}\label{cor: Akl-ampleness}
  Let $U$ be sufficiently small.
  The line bundle $\CA^{\min, k,l}_E$ on $Y_{U, E}^{\min}$ is ample if and only if $p k_\tau > k_{\Fr^{-1}\circ\tau}$ for all $\tau\in \Sigma$.
  It is nef if and only if $p k_\tau \geq k_{\Fr^{-1}\circ\tau}$ for all $\tau\in \Sigma$.
\end{corollary}

\begin{proof}
  Recall in \S\ref{subsec: HMV}, we have 
  $\tilY_U = \coprod_{\epsilon} \tilY_U^\epsilon$ for $\epsilon\in \big(\AFf^{(p)}\big)^\times/\big(\det U^p\big) \big(\widehat{\Z}^{(p)}\big)^\times$ and 
  $Y_U = \coprod_{\epsilon\in \CE} G_U\backslash\tilY_U^\epsilon$, where $\CE$ is a finite set of elements in $\big(\AFf^{(p)}\big)^\times/\big(\det U^p\big) \big(\widehat{\Z}^{(p)}\big)^\times$. 
  Also, recall in \S\ref{subsub: tor&min}, we have $\tilYUmin = \coprod_{\epsilon} \tilYUemin$ for
  $\epsilon\in\big(\AFf^{(p)}\big)^\times/\big(\det U^p\big) \big(\widehat{\Z}^{(p)}\big)^\times$
  and $Y^{\min}_{U, E} = \coprod_{\epsilon\in \CE}G_U\backslash\tilY^{\epsilon, \min}_{U, E}$.
  The projection
  $\tilY^{\epsilon}_U \rightarrow Y_U^{\epsilon}$ is finite 
  and there are only finitely many cusps on $\tilY^{\epsilon, \min}_{U, E}$, 
  so we obtain $\tilY^{\epsilon, \min}_{U, E} \rightarrow Y^{\epsilon, \min}_{U, E}$ is proper and quasifinite, hence finite.
  Therefore, $\CA^{\min, k,l}_E$ is ample (resp.\,nef) on $Y^{\min}_{U, E}$ if and only if $\widetilde{\CA}^{\min, k,l}_E$ is ample (resp.\,nef) on $\tilY^{\epsilon, \min}_{U, E}$, for all $\epsilon\in \CE$.
  We also note that $\widetilde{\CA}^{\min, \mathbf{0},l}$ is a free $\CO_{\tilYUmin}$-module of rank one. 
  Hence the numerical condition follows from \Cref{thm: DedingYang}.
\end{proof}

\begin{corollary}\label{cor: Lk_ampleness}
  Let $U$ be sufficiently small.
  The line bundle $\CL^{\min, k}_E$ on $Y_{U, E}^{\min}$ is ample if and only if $p k_\tau > k_{\Fr^{-1}\circ\tau}$ for all $\tau\in \Sigma$.
  It is nef if and only if $p k_\tau \geq k_{\Fr^{-1}\circ\tau}$ for all $\tau\in \Sigma$.
\end{corollary}

\begin{proof}
  Recall in \S\ref{subsec: twisting} we have a free line bundle $\CN^{\min}_E$ over $Y^\min_{U,E}$ such that $\CA^{\min,k,l}_E = \CL^{\min, k}_E \otimes \CN^{\min}_E$.
  Then we obtain the result directly by the previous corollary.
\end{proof}

We now recall the vanishing theorem for projective smooth variety in characteristic $p$ proved by Esnault and Viehweg in \cite[Proposition 11.5]{Esnault-Viehweg}.

\begin{theorem}[Esnault-Viehweg]\label{thm: EV-vanish}
  Let $X$ be a proper smooth variety over a perfect field $k$ of characteristic $p > 0$, $E$ a simple normal crossings divisor on $X$, and $L$ a line bundle on $X$.
  Assume $X$ has dimension $d \leq p$.
  Also assume the following conditions:
  \begin{enumerate}
    \item $(X, E, L)$ lifts to $(\tilde{X}, \tilde{E}, \tilde{L})$ over $W_2(k)$.
    \item There is an integer $\nu_0$ and an effective divisor $E'$ supported on $E$ such that $ L^{\otimes (\nu_0+\nu)}(-E')$ is ample for every integer $\nu \geq 0$.
  \end{enumerate}
  Then we have 
  \begin{align}\label{eq: EV-vanish}
    H^i(X, \Omega^j_X(\log E)\otimes L^{\otimes(-1)}) = 0  \text{ for } i+j<d.
  \end{align}
\end{theorem}

We use the vanishing theorem to prove the following result:

\begin{theorem}\label{thm: my-vanish}
  Suppose $p \geq [F:\Q] = d$
  and 
  \begin{align}\label{eq: liftcone}
    p(k_\tau - 2) > k_{\Fr^{-1}\circ \tau} - 2, \text{ for all } \tau\in \Sigma.
  \end{align}
  Then we have 
  \begin{align}\label{eq: my-vanish}
    H^i(Y_{U, E}^{\tor}, \CL_{E}^{\tor, k}(-D_\infty)) = 0, \text{ for all } i > 0.
  \end{align}

\end{theorem}

\begin{proof}
  Recall in \S\ref{subsub: tor&min} we have the morphism $\pi: \YUtor\rightarrow \YUmin$ and there is a divisor $D'$ supported on the simple normal crossing divisor $D_\infty$ and $-D'$ is $\pi$-relatively ample.
  By \Cref{cor: Lk_ampleness}, $\CL^{\min, k-\mathbf{2}}_E$ is ample over $Y^{\min}_{U, E}$,
  hence $(\CL^{\tor, k-\mathbf{2}}_{U, E})^{\otimes \nu}(-D') = (\pi^* \CL^{\min, k-\mathbf{2}}_E)^{\otimes \nu}(-D')$ is ample, for $\nu$ sufficiently large.
  Recall that $\YUtor$ and $\CL^{\tor, k-\mathbf{2}}$ are defined over $\CO$ (Note that $\CO$ can be taken to be unramified at $p$ so that $\CO/p^2 \cong W_2(E)$, since $p$ is unramified in $F$).
  Therefore, $(Y^\tor_{U, E}, D_\infty, \CL^{\tor, k-\mathbf{2}}_E)$ satisfies the conditions in \Cref{thm: EV-vanish}.
  Let $j = 0$ in \eqref{eq: EV-vanish}, together with Serre's duality and Kodaira-Spencer isomorphism, we have
  $$
    H^i(Y_{U, E}^{\tor}, \CL^{\tor, k-\mathbf{2}}_E \otimes \CA^{\tor, \mathbf{2}, \mathbf{-1}}_E(-D_\infty))=0, \text{ for } i > 0.
  $$
  By the construction of $\CL^{k}$, which is the descent of $\widetilde{\CA}^{\tor, k, \mathbf{0}}$ with $\COFppx$-action defined by $\mu^{-k/2}\alpha_\mu^t$ for $\mu\in \COFppx$, we have an isomorphism 
  $
  \CL^{\tor, k-2} \otimes \CA^{\tor, \mathbf{2}, \mathbf{-1}} \cong \CL^{\tor, k}.
  $
  Hence we obtain \eqref{eq: my-vanish}.
\end{proof}

\begin{corollary}\label{cor: my-liftability}
  Assume $p \geq d$.
  The natural map 
  $$
    S_{k}(U; \CO) \otimes_{\CO} E \rightarrow S_{k}(U; E)
  $$
  is surjective if $k$ satisfies \eqref{eq: liftcone}.
\end{corollary}

\begin{proof}
  Let $\CL = \CL^{\tor, k}(-D_\infty)$.
  Consider the short exact sequence
  $$
    0 \rightarrow \CL \xrightarrow{\cdot p}\CL \rightarrow \CL_{E} \rightarrow 0.
  $$
  It induces the long exact sequence
    \begin{align*}
      0 \rightarrow H^0(\YUtor, \CL) \rightarrow &H^0(\YUtor, \CL) \rightarrow H^0(Y_{U, E}^{\tor}, \CL_{E}) \rightarrow \\
      & H^1(\YUtor, \CL) \rightarrow H^1(\YUtor, \CL) \rightarrow H^1(Y_{U, E}^{\tor}, \CL_{E}) \rightarrow \cdots
    \end{align*} 
  Note that $H^1(\YUtor, \CL)$ is of finite rank over $\CO$.
  By Nakayama's lemma, to prove $H^1(\YUtor, \CL) = 0$,
  it suffices to show $H^1(Y^{\tor}_{U, E}, \CL_{E}) = 0$, which is given by \Cref{thm: my-vanish}.
\end{proof}

\begin{theorem}\label{thm: my-lift-goem-alg}
  Let $\rho : G_F \rightarrow \GL_2(\Fpbar)$ be an irreducible, continuous, totally odd representation
  and $(k,l)\in \Z^\Sigma \times \Z^\Sigma$.
  Assume $\rho$ is geometrically modular of weight $(k,l)$. 
  If $k$ satisfies $p(k_\tau - 2) > k_{\Fr^{-1}\circ \tau} - 2$ for all $\tau\in \Sigma$, 
  then 
  $\rho$ is algebraically modular of weight $(k,l)$.
\end{theorem}

\begin{proof}
  Let $\xi: \AFfp \rightarrow E^\times$ be a continous character such that $\xi(\mu) = \mu^{l+k/2}$ for $\mu\in \COFppx$.
  Let $\rho \simeq \rho_f$ for some eigenform $f\in S_{k,l}(U;E)$,
  where $U$ is sufficiently small.
  Recall we have a free line bundle $\CN_E$ over $Y_{U,E}$ such that 
  $\CA_E^{k,l} = \CL_E^{k} \otimes \CN_E$.
  Let
  $e_\xi$ be a basis of the one-dimensional eigenspace in $H^0(Y_{U, E}, \CN_E)$ such that $\GL_2(\AFfp)$ acts by $\xi\circ \det$.
  By \S\ref{subsec: twisting}, we have an isomorphism given by multiplying $e_\xi$ 
  \begin{align*}
    S_k(U;E) \xrightarrow{\sim}
    S_{k,l}(U;E),
  \end{align*}
  such that $T_v(e_\xi f) =\xi(\varpi_v) e_\xi T_v(f)$ and  
  $S_v(e_\xi f) =\xi(\varpi_v)^2 e_\xi S_v(f)$ for $v\nmid p$ that $\GL_2(\CO_{F, v}) \subset U$.
  Let the eigenform $f_{-} \in S_k(U;E)$ be the 
  preimage of $f$ in the above isomorphism.
  Then $f_{-}$ can lift to an eigenform $\tilde{f}_{-}\in S_k(U;\CO)$ by \Cref{{cor: my-liftability}} and Deligne-Serre lifting lemma (enlarging $\CO$ if necessary).
  By Jacquet-Langlands correspondence,
  the Hecke eigensystem of 
  $\tilde{f}_{-}$ 
  appears from the Hecke eigensystem of 
  $M^B_{k}(U_{-}; \CO)$ 
  for some 
  quaternion algebra $B$ and open compact subgroup $U_{-} \subset B_\A^\times$ as in \S\ref{subsec: def&indef}.
  We can take $U_{-}$ sufficiently small, so that
  we have an isomorphism  given by
  the multiplication by $e_\xi^B$ as in \eqref{eq: twist_xi_B_k/2}:
\begin{align*}
  M^B_{k}(U_{-}; \CO) \xrightarrow{\sim} M^B_{k,l}(U_{-}; \CO)
\end{align*} such that 
$T_v(e^B_\xi f) =\xi(\varpi_v) e^B_\xi T_v(f)$ and 
$S_v(e^B_\xi f) =\xi(\varpi_v)^2 e^B_\xi S_v(f)$ for $v\nmid p$ such that $\GL_2(\CO_{F, v}) \subset U$ and $v\not\in \Sigma^B_\bff$, where
$e_\xi^B \in H^0(Y_U^B, \otimes_\tau \det_{[\tau]}^{l_\tau + k_\tau/2} R)$ defined by $\xi \circ \det$.
Therefore, the Hecke eigensystem of $f$ appears from 
$M^B_{k,l}(U_{-}; \CO) \otimes \Fpbar$,
which shows $\rho \simeq \rho_f$ is algebraically modular of weight $(k,l)$.
\end{proof}

\subsection{Weight shifting and Grothendieck group relations}

Let $\rho$ be irreducible and geometrically modular of weight $(k,l)$,
then 
$\rho\simeq \rho_f$ for some eigenform $f$ of weight $(k,l)$.
If $k$ satisfies \eqref{eq: liftcone}:
\begin{align*}
  p(k_\tau - 2) > k_{\Fr^{-1}\circ \tau} - 2, \text{ for all } \tau\in \Sigma,
\end{align*}
we proved the direction from geometric to algebraic modularity in \Cref{thm: my-lift-goem-alg}.  
However, it doesn't include the weights $k\in \Ximinplus \cap \Z^\Sigma_{\geq 2}$ that are ``near the boundary".
We can deal with such weights by the weight shifting method.

Let $v|p$.
Let $G_0(\Fpbar[\GL_2(\F_{v})])$ be the Grothendieck group on finite dimensional representations of $\GL_2(\F_{v})$ over $\Fpbar$,
on which we have a natural partial ordering $\leq$ defined by
$R \leq R'$ if $R' - R$ is in the submonoid of $G_0(\Fpbar[\GL_2(\F_{v})])$ consisting of classes of $\Fpbar$-representations of $\GL_2(\F)$.
We denote the isomorphism class of $V$ in the Grothendieck group by $[V]$.
Suppose $\tau\in\Sigma_v$.
By convention we let $[\Sym_{[\tau]}^ {-1}\F_v^2] = 0$.
If $n < -1$, we define
$$
\left[\Sym_{[\tau]}^{n}\right] = - \left[\det_{[\tau]}^{n+1}\Sym_{[\tau]}^{-n-2}\right].
$$
Moreover, we have the following Grothendieck group relation (\cite[(1.0.3)]{Reduzzi15})
\begin{align*}
  \left[\Sym_{[\tau]}^n \Sym_{[\Fr^{-1} \circ \tau]}^m \right] - \left[\det_{[\tau]}^1 \Sym_{[\tau]}^{n-1} \Sym_{[\Fr^{-1} \circ \tau]}^{m-p} \right] \\
  = \left[\Sym_{[\tau]}^{n+1} \Sym_{[\Fr^{-1} \circ \tau]}^{m-p} \right] - \left[\det_{[\tau]}^1 \Sym_{[\tau]}^{n} \Sym_{[\Fr^{-1} \circ \tau]}^{m-2p} \right]
\end{align*}
for $m,n\in \Z$.

Let $\rho$ be geometrically modular of weight $(k,l)$, where $k\in \Ximinplus \cap \Z^\Sigma_{\geq 2}$ and doesn't satisfy \eqref{eq: liftcone}.
Consider $f' = \prod_{\tau} \Ha_{\tau}^{s_\tau} \Theta_\tau^{t_\tau} f $ of weight $(k',l')$,
for some $s_\tau, t_\tau \geq 0$.
Note that we can choose $f$ so that $f'$ is non-zero by \Cref{lem: Theta(f)nonzero}.
Choose $s_\tau$ and $t_\tau$ such that the weights $k'$ satisfies \eqref{eq: liftcone}, 
then $\rho$ is algebraically modular of 
weight $(k', l')$.
Then
we have
$W^\BDJ(\rho) \cap \JH(V_{k', -k'-l'}) \neq \varnothing$, 
hence
we get some information of $\rho|_{I_{F_v}}$ for $v|p$, and compute the possible $W^\BDJ(\rho)$.
Our goal is to show that $\JH(V_{k,-k-l}) \cap W^\BDJ(\rho) \neq \varnothing$, which 
implies $\rho$ algebraically modular of weight $(k,l)$ if 
we further assume $\rho|_{G_{F(\zeta_p)}}$ satisfies the condition in \Cref{thm: BDJ-conj}.

\subsection{Split case}\label{subsec: split}
Let $F$ be a totally real field of degree $d$ in which $p$ totally splits.
Now we write $\Sym$ for $\Sym_{[\tau]}$ and $\det$ for ${\det}_{[\tau]}$.
Note that ${\det}^{p-1}$ is the trivial character.

\begin{lemma}\label{lem: split}
  Suppose $p \geq d$.
  Let $\rho$ be geometrically modular of weight $(k,l)$.
  If $k_\tau \geq 3$ for all $\tau$,
  then $\rho$ is algebraically modular of weight $(k,l)$.
\end{lemma}
\begin{proof}
  It is directly given by \Cref{thm: my-lift-goem-alg}.
\end{proof}

\begin{theorem}\label{thm: split_geom-alg-2}
  Suppose $p \geq d$ is an odd prime.
  Let $\rho$ be geometrically modular of weight $(k,l)$.
  If $k_\tau \geq 2$ for all $\tau$,
  then 
  $\JH(V_{k,- k - l}) \cap W^\BDJ(\rho) \neq \varnothing$.
\end{theorem}

\begin{proof}
  If $k_\tau \geq 3$ for all $\tau$, it's proved in the previous lemma.
  Now assume the set $S = \{ \tau\in \Sigma \mid k_\tau = 2\}$ is nonempty.
  Suppose $\rho$ arises from a normalised eigenform $f\in M_{k,l}(U_1(\frakn);E)$ as in \Cref{lem: Theta(f)nonzero}.
  Let  $ f' := \prod_{\tau\in S} \Ha_{\tau} f$. 
  Then $f'$ is a non-zero normalised eigenform of weight $(k', l)$ such that $\rho_{f'} \simeq \rho_{f}$, 
  where  $k'_\tau = k_\tau$ for $\tau \not\in S$ and $k'_\tau = k_\tau + p-1 \geq 3$ for $\tau\in S$.
  Let $m := -k-l$ and $m' := -k'-l$.
  Note ${\det}^{p-1} = {\det}^0$, so we have ${\det}^{m'_\tau} = {\det}^{m_\tau}$.
  By \Cref{lem: split}, $\rho$ is algebraically modular of weight $(k',l)$.
  Then $\rho$ is modular of a Jordan-Holder factor of $V_{k',m}$ by \Cref{lem: mod_of_JH}, 
  i.e. $W^\BDJ(\rho) \cap \JH(V_{k',m}) \neq \varnothing$.
  Let $W^\BDJ(\rho) = \prod_{v|p}W_v(\rho)$.
  For any $\tau\in \Sigma$, let $v_\tau$ be the unique prime above $p$ determined by $\tau$.
  If $\tau\not\in S$, then we have 
  $\JH\big({\det}^{m_\tau}\Sym^{k_\tau-2}\big) \cap W_{v_\tau}(\rho) \neq \varnothing $ since $k'_\tau = k_\tau$.
  If $\tau\in S$, 
  we have
  $ {\det}^{m_\tau}\Sym^{ p - 1}\in W_{v_\tau}(\rho)$, which implies
  the following possibilities of $W_{v_\tau}(\rho)$ (\textit{cf.}\,\cite[Theorem 3.17]{Buzzard-Diamond-Jarvis}):
  \begin{enumerate}
    \item $\{ {\det}^{m_\tau} \Sym^{p-1} \}$, or
    \item $\{ {\det}^{m_\tau} \Sym^{p-1},
    {\det}^{m_\tau} \Sym^{0}, {\det}^{m_\tau + 1} \Sym^{p-3} \}$ $ (p > 3)$, or
    \item $\{ {\det}^{m_\tau} \Sym^{2},
    {\det}^{m_\tau} \Sym^{0}, {\det}^{m_\tau + 1} \Sym^{2}, {\det}^{m_\tau + 1}\Sym^{0} \} $  $(p = 3)$, or
    \item $\{ {\det}^{m_\tau} \Sym^{p-1}, {\det}^{m_\tau} \Sym^{0} \}$.
  \end{enumerate}
  In the above possibilities, we observe that ${\det}^{m_\tau} \Sym^{0}\in W_{v_\tau}(\rho)$ unless $W_{v_\tau}(\rho) = \{ {\det}^{m_\tau} \Sym^{p-1} \}$.

  On the other hand, by \Cref{lem: theta_shift} we obtain that $\rho$ is geometrically modular of weight $(k'',l'')$, 
  where
  $k''_\tau = k_\tau$, $l''_\tau = l_\tau$ for $\tau\not\in S$, and 
  $k''_\tau = k_\tau + p + 1$, $l''_\tau = l_\tau - 1$ for $\tau\in S$.
  Then $\rho$ is algebraically modular of weight $(k'',l'')$ since $k''_{\tau} \geq 3$ for all $\tau$. 
  In particular, 
  we have $\JH({\det}^{m_\tau - 1}\Sym^{p+1}) \cap W_{v_\tau}(\rho) \neq\varnothing$.
  In the Grothendieck group $G_0(\Fpbar[\GL_2(\F_p)])$ on finite dimensional representations of $\GL_2(\F_p)$ over $\Fpbar$, we have
  $$
  [{\det}^{-1}\Sym^{p+1}] = [\Sym^{0}] +\  [{\det}^{-1}\Sym^{2}] + \ [{\det}^{2-p}\Sym^{p-3}],
  $$
  by \cite[(1.0.3)]{Reduzzi15}.
  If $W_{v_\tau}(\rho) = \{ {\det}^{m_\tau} \Sym^{p-1} \}$, then $\JH({\det}^{m_\tau - 1}\Sym^{p+1}) \cap W_{v_\tau}(\rho) = \varnothing$ gives a contradiction.
  Therefore, we must have ${\det}^{m_\tau} \Sym^{0}\in W_{v_\tau}(\rho)$ for any $\tau\in S$,
  and hence we
  conclude that $\JH(V_{k,m}) \cap W^\BDJ(\rho) \neq \varnothing$.
\end{proof}

\begin{lemma}\label{lem: theta_shift}
  Let $\rho$ be geometrically modular of weight $(k,l)\in\Z^\Sigma_{\geq 2} \times \Z^\Sigma$.
  Let $S = \{ \tau\in \Sigma \mid k_\tau = 2\}$.
  Then $\rho$ is geometrically modular of weight $(k'',l'')$, 
  where
  $k''_\tau = k_\tau$, $l''_\tau = l_\tau$ for $\tau\not\in S$, and 
  $k''_\tau = k_\tau + p + 1$, $l''_\tau = l_\tau - 1$ for $\tau\in S$.
\end{lemma}

\begin{proof}
  Let $S = \{\tau_1, \cdots, \tau_i\}$.
  There is a normalised eigenform $f\in M_{k,l}(U_1(\frakn);E)$ as in \Cref{lem: Theta(f)nonzero} such that $\rho$ arises from $f$.
  If f is a normalized eigenform, then so is $\Theta_{\tau}(f)$ for any $\tau\in \Sigma$.
  Therefore, $\Theta_{\tau_1}\cdots\Theta_{\tau_i}(f)$ is a normalized eigenform that gives rise to $\rho$ with the desired weight.
\end{proof}

\begin{theorem}\label{thm: goem-alg-split}
  Let $p \geq d$, $p > 3$ and $F$ be a totally real field in which $p$ totally splits.
  Let $\rho: G_F \rightarrow \GL_2(\Fpbar)$ be an irreducible, continuous, totally odd representation.
  If $p = 5$, assume that the projective image of $\rho|_{G_{F(\zeta_p)}}$ is
  not isomorphic to $A_5$.
  If $\rho$ be geometrically modular of weight $(k,l)$,
  where $k_\tau \geq 2$ for all $\tau$, 
  then $\rho$ is algebraically modular of weight $(k,l)$.
\end{theorem}

\begin{proof}
  By \Cref{thm: my-lift-goem-alg}, we only need to prove the case when $S = \{\tau\in\Sigma \mid k_\tau = 2\}$ is non-empty.
  Without loss of generality, we may assume $\rho$ is geometrically modular of weight $(k, - k)$ using \Cref{lem: reduce to l=0}.
  By \Cref{thm: split_geom-alg-2}, we have $V_{k,\mathbf{0}} \in W^\BDJ(\rho)$.
  If $\rho|_{G_{F(\zeta_p)}}$ satisfies the condition in \Cref{thm: BDJ-conj}, then we obtain $\rho$ is algebraically modular of weight $(k, -k)$.
  Therefore, by our assumption, we only need to consider the case when $\rho|_{G_{F(\zeta_p)}}$ is reducible.
  
Suppose $\rho|_{G_{F(\zeta_p)}}$ is reducible.
Then there is a quadratic extension $F'/F$ in $F(\zeta_p)$ and a continuous character $\xi: G_{F'} \rightarrow \Fpbar^\times$
such that 
$\rho \simeq \Ind^F_{F'} \xi$.
Since $p$ is split, $p = \prod_{\tau\in \Sigma}\mathfrak{p}_\tau$ in $F$ and $F_{\mathfrak{p}_\tau} \cong \Q_{p}$ for all $\tau\in \Sigma$.
Note that $F'/F$ is totally ramified at $\mathfrak{p}_\tau$ for every $\tau\in\Sigma$. 
Let $\tau_0 \in S$ (i.e. $k_{\tau_0} = 2$).
Let $\mathfrak{p}_{\tau_0} = v^2$ for a prime $v$ of $F'$.
Denote $K := F_{\mathfrak{p}_{\tau_0}} \cong \Q_p$ and 
$K' := F'_v \cong \Q_{p^2}$.
Let $\varepsilon$ be the mod $p$ cyclotomic character, 
$\varepsilon_2$ be the fundamental character of level $2$, which is
$I_K \cong I_{\Q_{p^2}} \xrightarrow{\Art^{-1}} \Z_{p^2}^\times \rightarrow \F_{p^2}^\times$.
Let 
$\varepsilon_{K'}$ be the fundamental character on 
$I_{K'}$ induced by $\Art^{-1}$, respectively.
Note that $\varepsilon_2^{p+1} = \varepsilon$ and we also have the commutative diagram
\[\begin{tikzcd}
	{\varepsilon_{K'}:\ I_{K'}} & {\CO_{K'}^\times} & {\F_p^\times} \\
	{\varepsilon:\ I_{\Q_p}} & {\Z_p^\times} & {\F_p^\times}.
	\arrow[from=1-1, to=1-2]
	\arrow[shift left=4, hook, from=1-1, to=2-1]
	\arrow[from=1-2, to=1-3]
	\arrow["\Nm", from=1-2, to=2-2]
	\arrow["{(\cdot)^2}", from=1-3, to=2-3]
	\arrow[from=2-1, to=2-2]
	\arrow[from=2-2, to=2-3]
\end{tikzcd}\]
We have $\xi|_{I_{K'}} = \varepsilon_{K'}^n$ for some $n\in \Z_{\geq 0}$.
Depending on the parity of $n$, we get the form of $\rho|_{I_K}$:
\begin{align}\label{eq: rhoI_K-v1}
  \rho|_{I_K} \simeq 
  \left\{ \begin{matrix}
    &\varepsilon^{n/2} \oplus \varepsilon^{n/2+ (p-1)/2}, \text{ if } n \text{ is even} \\
    &\varepsilon_2^{a} \oplus \varepsilon_2^{pa}, \text{ if } n \text{ is odd}
  \end{matrix}\right.
\end{align}
where $a = n(p+1)/2$.
On the other hand, $\Sym_{[\tau_0]}^{k_{\tau_0} - 2} \in W^{(\tau_0)}(\rho)$, so
we know that $\rho|_{I_K}$ as the following form as in \eqref{eq: reducible-case} and \eqref{eq: irreducible-case},
\begin{align}\label{eq: rhoI_K-v2}
  \begin{aligned}
    &\rho|_{I_K} \simeq \left(\begin{matrix}
      \varepsilon^{k_{\tau_0}-1} & * \\ 0 & 1
    \end{matrix}\right) \text{ reducible case, }
    \\
    &\rho|_{I_K} \simeq \left(\begin{matrix}
      \varepsilon_2^{k_{\tau_0}-1} & 0 \\ 0 &  \varepsilon_2^{p(k_{\tau_0}-1)}
    \end{matrix}\right) \text{ irreducible case.}
  \end{aligned}
\end{align}
  Comparing the form of $\rho|_{I_K}$ in \eqref{eq: rhoI_K-v1} and \eqref{eq: rhoI_K-v2}, we obtain:
  \begin{enumerate}
    \item If $n$ is even, then $\rho|_{G_K}$ is reducible and $k_\tau = (p+1)/2$.
    \item If $n$ is odd, then $\rho|_{G_K}$ is irreducible
    and $k_\tau = (p+3)/2$.
  \end{enumerate}
  When $p>3$, we have $k_{\tau_0} \geq 3$, contradicts to $k_{\tau_0} = 2$.
  Therefore, we have the algebraic modularity of such weight.
\end{proof}

We have the following result including the partial or parallel weight one cases.
\begin{corollary}\label{cor: my_split_&one}
  Suppose $p \geq d$ and $p > 3$.
  Let $\rho$ be geometrically modular of weight $(k,l)\in \Z_{\geq 1}^\Sigma \times\Z^{\Sigma}$.
  If $p = 5$, assume that the projective image of $\rho|_{G_{F(\zeta_p)}}$ is not isomorphic to $A_5$.
  Then $\rho$ is algebraically modular of weight $(k',l)$ where
  $k'_\tau = k_\tau$ if $k_\tau \geq 2$ and $k'_\tau = p$ if $k_\tau = 1$.
\end{corollary}

\begin{proof}
  Let $\rho$ arise from a normalised eigenform $f$ of weight $(k,l)$.
  Let $T := \{\tau\in \Sigma \mid k_\tau = 1\}$.
  We consider the eigenform $g := \prod_{\tau\in T} \Ha_{\tau} f$, which is of the weight $(k',l)$.
  Then $\rho$ is geometrically and hence algebraically modular of weight $(k',l)$ by \Cref{thm: goem-alg-split}.
\end{proof}

\subsection{Inert quadratic case}

Let $F$ be a real quadratic field in which $p$ is inert.
Let $\Sigma = \{\tau_0, \tau_1\}$, where $\tau_1 = \Fr \circ \tau_0$, and $\tau_0 = \Fr \circ \tau_1$.
Let $(k,l) = ((k_0,k_1), (l_0, l_1))\in \Z^\Sigma \times \Z^\Sigma$, where $k\in \Ximinplus$, i.e. $p k_0 \geq k_1 \geq k_0/p$.
Note that ${\det}^{p^2-1} = {\det}^0$ and ${\det}_{[\tau_1]} = {\det}_{[\tau_0]}^p$.
We denote $\Symo := \Sym_{[\tau_0]}$, $\Syml := \Sym_{[\tau_1]}$ and $\e := {\det}_{[\tau_0]}$.
We also denote $\Ha_{\tau_i}$ by $\Ha_i$ and $\Theta_{\tau_i}$ by $\Theta_i$ for $i = 0,1$.

Altering our choice of $\tau_0$ if necessary, we assume that $2 
\leq k_0 \leq k_1$.
By \Cref{thm: my-lift-goem-alg}, geometric modularity of weight $((k,l))$ implies algebraic modularity of weight $((k,l))$, if $p(k_0 - 2) > k_1 -2$. 
Now we consider the weights such that
$p(k_0 -2) \leq k_1 - 2 \leq pk_0 - 2$.

We will use the weight shifting technique to prove $V_{k,l} \in W^\BDJ(\rho)$ for relatively small $k$ (see \Cref{thm: geq2_VinW}). 
For bigger weights $k$, we use the description of stratification on the special fibre of Hilbert modular variety (given in \cite{TX16})
to get a function on the underlying quaternionic Shimura variety, which directly gives algebraic modularity (see \Cref{subsub: big_wt}).

\begin{lemma}\label{lem: not-ord-not-min}
  Let $(k,l), (k',l')\in \Z^\Sigma_{\geq 2} \times \Z^\Sigma$ where
  $k'_0 = k_0$, $k'_1 = k_1 -2$, and $l'_0 = l_0$, $l'_1 = l_1 + 1$.
  Suppose $\rho$ is geometrically modular of weight both $(k, l)$ and $(k',l')$.
  If $\rho$ is not ordinarily modular at $p$ of weight $(k,l)$, then either $\rho$ is also geometrically modular of weight $(k' - k_{\Ha_1}, l')$, 
  or $p|k'_1$.
\end{lemma}

\begin{proof}
  By \Cref{lem: stabilised_f}, $\rho$ arises from a stabilised eigenform $f$ of weight $(k, l)$ and a stabilised eigenform $f'$ of weight $(k', l')$.
  Let $\xi$ be a character of weight $-l$.
  Then $\rho_{\xi'} \otimes \rho$ arises from $f_\xi$ of weight $(k,(0,0))$ and $f'_\xi$ of weight $(k', (0, 1))$.
  Since $f$ is not ordinary, we have $f_\xi$ is a strongly stabilised eigenform, and hence $\Ha_1 f_\xi$ (as $k_1 \geq 4$ by our assumption).
  Also, $\Theta_1(f'_\xi)$ is a strongly stabilised eigenform by \cite[Proposition 9.4.1]{Diamond-Sasaki}.
  By $q$-expansion principle, we obtain
  $$\Ha_1 f_\xi = \Theta_1(f'_\xi). $$
  If $p\nmid k_1'$, then by \Cref{thm: Theta-Hasse-divisibility}, we have $\Ha_1\mid f'_\xi$, hence 
  $\rho_{\xi'} \otimes \rho$ is geometrically modular of $(k' - k_{\Ha_1}, (0,1))$, and then
  $\rho$ is geometrically modular of weight $(k' - k_{\Ha_1}, l')$. 
\end{proof}

\subsubsection{Small weights}

Let $\rho$ be geometrically modular of weight $(k,l)$, where $ k = (k_0, k_1)$ and $2 \leq k_0 \leq k_1 \leq p$.

\begin{lemma}\label{lem: 2 and p+1}
  If $V_{(p+1,p+1), (l_0, l_1)}\in W^\BDJ(\rho)$, then either $V_{(2,2), (l_0, l_1)}\in W^\BDJ(\rho)$, or $W^\BDJ(\rho) = \{V_{(p+1,p+1), (l_0, l_1)}\}$.
\end{lemma}

\begin{proof}
  It is explained in the last paragraph of the proof of \cite[Lemma A.4]{CEGS} that, 
  if $V_{(p+1,p+1), (l_0, l_1)}\in W^\BDJ(\rho)$ and $V_{(2,2), (l_0, l_1)} \not\in W^\BDJ(\rho)$, then 
  $\rho|_{G_{F_p}}$ is tr\`es ramifi\'ee,
  in which case $ W^\BDJ(\rho) = \{V_{(p+1,p+1), (l_0, l_1)}\}$ by the definition in \cite[\S3.2]{Buzzard-Diamond-Jarvis}.
\end{proof}

\begin{lemma}\label{lem: twotwo}
  If $\rho$ is geometrically modular of weight $ (k,l) =  ((2,2), (0,0))$, then $V_{(2,2), (-2,-2)} \in W^\BDJ(\rho)$.
\end{lemma}

\begin{proof}
  Let $\rho \simeq \rho_f$ for some normalised eigenform $f\in M_{(2,2), (0,0)}(U_1(\frakn); E)$ as in \Cref{lem: Theta(f)nonzero} so that $\Theta_1(f)\neq 0$.
  Consider the eigenforms $\Ha_0\Ha_1(f)$ and $\Ha_0\Theta_1(f)$.  
  Then
  $\rho$ is geometrically hence algebraically modular of weight $((p+1, p+1), (0,0))$ and $(p+1, p+3), (0,-1)$ by \Cref{thm: my-lift-goem-alg}.
  Let $M = -2-2p$ and $V' := \e^{M-p}\Symo^{p-1}\Syml^{p+1}$.
  Then $\e^{M}\Symo^{p-1}\Syml^{p-1}\in W^\BDJ(\rho)$ and $\rho \cap W(V')\neq \varnothing$.
  By \Cref{lem: 2 and p+1}, $V_{(2, 2), (-2, -2)} \in W^\BDJ(\rho)$ unless
  $W^\BDJ(\rho) = \{\e^{M}\Symo^{p-1}\Syml^{p-1}\}$.
  By Grothendieck group relation \cite[(1.0.3)]{Reduzzi15}, we have
  \begin{align*}
    \left[\e^{-p}\Symo^{p-1}\Syml^{p+1}\right] = 2 \left[\e^{1-p}\Symo^{p-2}\Syml^{1}\right] + \left[\Symo^0\Syml^{0}\right] +\\ \left[\e^{-p}\Symo^{0}\Syml^{2}\right] + \left[\e^{p}\Symo^{p-1}\Syml^{p-3}\right].
  \end{align*}
If $W^\BDJ(\rho) = \{\e^{M}\Symo^{p-1}\Syml^{p-1}\}$, then $W^\BDJ(\rho) \cap \JH(V') = \varnothing$, which is a contradiction.
Therefore, $V_{(2, 2), (-2, -2)} \in W^\BDJ(\rho)$. 
\end{proof}

\begin{lemma}\label{lem: twothree}
  Suppose $p > 3$.
  If $\rho$ is geometrically modular of weight $ (k,l) =  ((2,3), (0,0))$, then $V_{(2,3),(-2,-3)}\in W^\BDJ(\rho)$.
\end{lemma}

\begin{proof}
  Let $\rho \simeq \rho_f$ for some normalised eigenform $f\in M_{(2,3), (0,0)}(U_1(\frakn); E)$ as in \Cref{lem: Theta(f)nonzero} so that $\Theta_1(f)\neq 0$.
  Consider the eigenforms $\Ha_0 \Ha_1 f$ and $\Theta_1(f)$.
  The representation $\rho$ is
  geometrically hence algebraically modular of weight $((p+1, p+2), (0,0))$ and $(p+2, 4), (0,-1)$.
  Let $m_i = 1-k_i$ for $i = 0,1$ and $M =  m_0 + pm_1$. 
  We have the Grothendieck group relations
  \begin{align*}
    \left[\Symo^{p-1}\Syml^{p}\right] = 2 \left[\e^{1}\Symo^{p-2}\Syml^{0}\right] + \left[\Symo^{0}\Syml^{1}\right]  + \left[\e^{p}\Symo^{p-1}\Syml^{p-2}\right],
    \\
    \left[\e^{-p}\Symo^{p}\Syml^{2}\right] = \left[\Symo^{0}\Syml^{1}\right] + \left[\e^{-p}\Symo^{0}\Syml^{3}\right] + \left[\e^{1-p}\Symo^{p-2}\Syml^{2}\right].
  \end{align*}
  Since $\rho$ is modular of weight $e^M \Symo^{p-1}\Syml^{p}$, it is modular of at least one of $\e^{M+1}\Symo^{p-2}\Syml^{0}$, $e^M\Symo^{0}\Syml^{1}$ and $\e^{M + p}\Symo^{p-1}\Syml^{p-2}$.

  Suppose $\e^{M+1}\Symo^{p-2}\Syml^{0} \in W^\BDJ(\rho)$. 
  We can use the results in \cite[\S 8]{Dembele-Diamond-Roberts} to 
  work out all the possibilities for $W^\BDJ(\rho)$.
  (We list the computation result for $W^\BDJ(\rho)$ here once and omit it in the later proofs, since it is a straightforward but not enlightening enumeration using the results in \cite[\S 8]{Dembele-Diamond-Roberts}.)
  Since $\e^{M+1}\Symo^{p-2}\Syml^{0} \in W^\BDJ(\rho)$, then $\rho$ can be of Case I, III, IV or V in \cite{Dembele-Diamond-Roberts} and the
  possibilities of $W^\BDJ(\rho)$ are as follows: 
  \begin{enumerate}
    \item If
    $\rho$ is of Case I as in \cite[\S8.2]{Dembele-Diamond-Roberts}, then $\e^{-M}W^\BDJ(\rho) = $
      \begin{enumerate}
        \item $\{\e^{1}\Symo^{p-2}\Syml^{0}\}$, or
        \item $\{\e^{1}\Symo^{p-2}\Syml^{0}, \e^p \Symo^{p-1}\Syml^{p-2}\}$, or
        \item $\{\e^{1}\Symo^{p-2}\Syml^{0}, \Symo^{0}\Syml^{1}\}$, or
        \item $\{\e^{1}\Symo^{p-2}\Syml^{0}, \Symo^{0}\Syml^{1}, \e^p \Symo^{p-1}\Syml^{p-2}, \e^{2-2p}\Symo^{p-1}\Syml^{p-4}\}$, or 
      \end{enumerate}
      \begin{enumerate}[(a')]
        \item $\{\e^{p+1}\Symo^{p-3}\Syml^{p-2}, \e^{1}\Symo^{p-2}\Syml^{0}\}$, or
        \item $\{\e^{p+1}\Symo^{p-3}\Syml^{p-2}, \e^{1}\Symo^{p-2}\Syml^{0}, \e^{p-1}\Symo^{1}\Syml^{p-1}, \e^{p}\Symo^{p-1}\Syml^{p-2}\}$.
      \end{enumerate}

    \item If $\rho$ is of Case III as in \cite[\S8.3]{Dembele-Diamond-Roberts}, then $\e^{-M}W^\BDJ(\rho) = $
      \begin{enumerate}
        \item $\{\e^{p}\Symo^{p-1}\Syml^{p-2}, \e^{1}\Symo^{p-2}\Syml^{0}, \e^{p-1}\Symo^{1}\Syml^{p-1} \}$, or
        \item $\{\e^{p}\Symo^{p-1}\Syml^{p-2}, \e^{1}\Symo^{p-2}\Syml^{0}, \e^{p-1}\Symo^{1}\Syml^{p-1}, \e^{p+1}\Symo^{p-3}\Syml^{p-2} \}$, or
      \end{enumerate}
      \begin{enumerate}[(a')]
        \item $\{\e^{2p}\Symo^{p-1}\Syml^{p-4}, \e^{p}\Symo^{p-1}\Syml^{p-2}, \Symo^{0}\Syml^{1}, \e^{1}\Symo^{p-2}\Syml^{0} \}$.
      \end{enumerate}

    \item If $\rho$ is of Case IV as in \cite[\S8.5]{Dembele-Diamond-Roberts}, then $\e^{-M}W^\BDJ(\rho) = $
    
    $\{ \e^{1}\Symo^{p-2}\Syml^{0}, \e^{1}\Symo^{1}\Syml^{p-3}, \Symo^{0}\Syml^{1}, \e^{p+1}\Symo^{p-3}\Syml^{p-2} \}$.

    \item If $\rho$ is of Case V as in \cite[\S8.5]{Dembele-Diamond-Roberts}, then $\e^{-M}W^\BDJ(\rho) = $
    
    $\{ \e^{p-1}\Symo^{1}\Syml^{p-1}, \e^{p}\Symo^{p-1}\Syml^{p-2}, \e^{1}\Symo^{p-2}\Syml^{0} \}$.
  \end{enumerate}
  We observe that
  if $\rho$ is modular of weight $\e^{M+1}\Symo^{p-2}\Syml^{0}$, then 
  either $\e^M \Symo^{0}\Syml^{1} \in W^\BDJ(\rho)$ or 
  $W^\BDJ(\rho)\cap \JH(\e^{M-p} \Symo^{p}\Syml^{2}) = \varnothing$, the latter of which leads to a contradiction.
  
  Similarly, if $\rho$ is modular of weight $\e^{M+p}\Symo^{p-1}\Syml^{p-2}$, we can also compute the possibilities of $W^\BDJ(\rho)$ and observe that  
  either $\e^M \Symo^{0}\Syml^{1} \in W^\BDJ(\rho)$ or $W^\BDJ(\rho)\cap \JH(\e^{M-p}\Symo^{p}\Syml^{2}) = \varnothing$.
  Therefore, we always have $V_{(2,3),(-2,-3)}\in W^\BDJ(\rho)$.
\end{proof}

\begin{lemma}\label{lem: twofour}
  Suppose prime $p > 3$.
  If $\rho$ is geometrically modular of weight $ (k,l) =  ((2,4), (0,0))$, then $V_{(2,4),(-2,-4)}\in W^\BDJ(\rho)$.
\end{lemma}

\begin{proof}
  If $\rho$ is ordinarily geometrically modular of weight $((2,4), (0,0))$, then by \Cref{cor: ord-geom-alg},
  we have $V_{(2,4),(-2,-4)}\in W^\BDJ(\rho)$.
  Now we only need to prove the result when $\rho$ is non-ordinarily geometrically modular of weight $((2,4), (0,0))$.
  By \Cref{lem: not-ord-not-min}, $\rho$ cannot be geometrically modular of weight $((2,2), (0,1))$, otherwise $\rho$ is also geometrically modular of weight $((3, 2-p), (0,1))$, and \Cref{thm: wt-in-min-cone} gives a contradiction.
  Let $m_i = -k_i$ for $i=0,1$ and $M =  m_0 + pm_1$. 
  We know $\rho$ cannot by modular of weight $\e^{M+p} \Symo^0\Syml^0$.
  Let $\rho \simeq \rho_f$ for some normalised eigenform $f\in M_{(2,4), (0,0)}(U_1(\frakn); E)$ as in \Cref{lem: Theta(f)nonzero} so that $\Theta_1(f)\neq 0$.
  Denote $V:= \e^{M}\Symo^{p-1}\Syml^{p+1}$ and $V' := \e^{M-p}\Symo^{p}\Syml^{3}$.
  Consider the eigenforms $\Ha_0\Ha_1f$ and $\Theta_1(f)$, then we know $\rho$ is modular of weights $V$ and $V'$ by \Cref{thm: my-lift-goem-alg}.
  We have Grothendieck group relations
  \begin{align*}
    \left[\Symo^{p-1}\Syml^{p+1}\right] = & \left[\Symo^{0}\Syml^{2}\right] + \left[\e^p\Symo^{0}\Syml^{0}\right] + 2 \left[\e^1\Symo^{p-2}\Syml^{1}\right] + \\ 
    & \left[\e^{2p}\Symo^{p-1}\Syml^{p-3}\right],
\end{align*}
\begin{align*}
    \left[\e^{-p}\Symo^{p}\Syml^{3} \right] = \left[\Symo^{0}\Syml^{2}\right] + \left[\e^{-p}\Symo^{0}\Syml^{4}\right] + \left[\e^{1-p}\Symo^{p-2}\Syml^{3}\right].
\end{align*}
Since $\rho$ is modular of weight $V$ and not geometrically modular of weight $((2,2), (0,1))$, it is modular of at least one of the weights $\e^M\Symo^{0}\Syml^{2}$, $\e^{M+1}\Symo^{p-2}\Syml^{1}$ and $\e^{M+2p}\Symo^{p-1}\Syml^{p-3}$.
If $\e^{M+1}\Symo^{p-2}\Syml^{1}$ or $\e^{M+2p}\Symo^{p-1}\Syml^{p-3}$ $\in W^\BDJ(\rho)$, we use \cite[\S 8]{Dembele-Diamond-Roberts} to get all the possibilities of $W^\BDJ(\rho)$ and observe that
either $\e^{M}\Symo^{0}\Syml^{2} \in W^\BDJ(\rho)$ or $W^\BDJ(\rho) \cap \JH(V') = \varnothing$.
The latter case gives contradiction since $\rho$ is also modular of weight $V'$.
Therefore we must have $\e^{M}\Symo^{0}\Syml^{2} \in W^\BDJ(\rho)$.
\end{proof}

\begin{theorem}\label{thm: geq2_VinW}
  Let $k = (2, k_1)$, where $2 \leq k_1\leq p$.
  If $\rho$ is geometrically modular of weight $(k, (0,0))$, then $V_{k, -k} \in W^\BDJ(\rho)$.
\end{theorem}

\begin{proof}
  Let's assume $\rho$ is not algebraically modular of weight $((2,k_1), (0,0))$, and prove by contradiction.
  By 
  \Cref{lem: twotwo}, \Cref{lem: twothree}, \Cref{lem: twofour}, 
  and \Cref{cor: ord-geom-alg}, 
  it is sufficient to prove for $p \geq k_1 \geq 5$ and $\rho$ is non-ordinarily geometrically modular of weight $((2,k_1), (0,0))$ by \Cref{cor: ord-geom-alg}.
  Let $m_i = -k_i$ for $i = 0,1$ and $M =  m_0 + pm_1$.
  Suppose $\rho \simeq \rho_f$ for some normalised eigenform $f\in M_{(2,k_1), (0,0)}(U_1(\frakn); E)$ as in \Cref{lem: Theta(f)nonzero} so that $\Theta_1(f)\neq 0$.
  Consider the eigenforms $\Ha_1 f$ and $\Theta_1(f)$.
  Then $\rho$ is modular of weight $V = \e^{M}\Symo^{p}\Syml^{k_1-3}$ and $V' = \e^{M-p}\Symo^{p}\Syml^{k_1 - 1}$.
  We have Grothendieck relations
  \begin{align}
    \begin{aligned}\label{eq: (2,k1)_1}
      \left[\Symo^{p}\Syml^{k_1 - 3} \right]= \left[\e^{p} \Symo^{0}\Syml^{k_1 -4}\right] +  \left[\Symo^{0}\Syml^{k_1-2}\right] + \\  \left[\e^{1}\Symo^{p-2}\Syml^{k_1-3}\right], 
    \end{aligned}
  \end{align}
  \begin{align}
    \begin{aligned}\label{eq: (2,k1)_2}
      \left[\e^{-p}\Symo^{p}\Syml^{k_1 - 1}\right] = \left[\Symo^{0}\Syml^{k_1 - 2}\right] + \left[\e^{-p}\Symo^{0}\Syml^{k_1}\right] + \\ \left[\e^{1-p}\Symo^{p-2}\Syml^{k_1 - 1}\right].
    \end{aligned}
  \end{align}
By \eqref{eq: (2,k1)_1} and our assumption that $\rho$ is not algebraically modular of weight $((2,k_1), (0,0))$, we have
$\rho$ is modular of either weight $\e^{M+p} \Symo^{0}\Syml^{k_1 -4}$ or $\e^{M+1}\Symo^{p-2}\Syml^{k_1-3}$.
If $\rho$ is modular of $\e^{M+p} \Symo^{0}\Syml^{k_1 -4}$ and hence geometrically modular of weight $((2,k_1-2), (0,1))$ by \Cref{prop: alg-geom_DS}(ii),
then \Cref{lem: not-ord-not-min} shows that $\rho$ is geometrically modular of $ ((2-p, k_1-1),(0,1))$, which is a contradiction by \Cref{thm: wt-in-min-cone}. 
Therefore, $\rho$ is modular of weight $\e^{M+1}\Symo^{p-2}\Syml^{k_1-3}$, i.e. 
algebraically modular of weight $((p, k_1- 1), (1, 0))$, 
and hence it is geometrically modular of weight $((p, k_1- 1), (1, 0))$ by \Cref{prop: alg-geom_DS}(ii).
By \eqref{eq: (2,k1)_2} and our assumption that $\rho$ is not algebraically modular of weight $((2,k_1), (0,0))$, 
we have
$\rho$ is algebraically modular either of weight $((2,k_1+2),(0,-1))$, or of weight $((p,k_1+1), (1,-1))$.
\begin{enumerate}
  \item Suppose $\rho$ is algebraically modular of weight $((2,k_1+2),(0,-1))$.
  \begin{enumerate}
    \item Suppose $k_1 \leq p-1$. If $\rho$ is non-ordinarily geometrically modular of weight $((2,k_1+2),(0,-1))$,
    then together with $\rho$ being  geometrically modular of weight $((2, k_1), (0,0))$ and $p\nmid k_1$,  
    we obtain $\rho$ geometrically modular of weight $((2-p, k_1+1), (0,0))$ by \Cref{lem: not-ord-not-min}, 
    which gives a contradiction by \Cref{thm: wt-in-min-cone}. 
    On the other hand, if $\rho$ is ordinarily geometrically modular of weight $((2,k_1+2),(0,-1))$,
    then $\rho|_{I_{F_p}}$ has a particular form by \Cref{thm: ord-geom-rho}, 
    which enables us to compute the possible $W^\BDJ(\rho)$ by \Cref{def: WBDJ}(i).
    We obtain $\e^{M+1}\Symo^{p-2}\Syml^{k_1-3} \not\in W^\BDJ(\rho)$, which is a contradiction.
    \item Suppose $k_1 = p$. We have 
    $
      \left[\e^{-p}\Symo^{0}\Syml^{p}\right] = \left[\Symo^{0}\Syml^{p-2}\right]  + \left[\e^{-p}\Symo^{1}\Syml^{0}\right].
    $
    Since we assume $\rho$ is not algebraically modular of weight $((2,k_1), (0,0))$, we have $\rho$ is modular of weight $\e^{M-p}\Symo^{1}\Syml^{0}$. 
    We use \cite[\S 8]{Dembele-Diamond-Roberts} to compute all the possible $W^\BDJ(\rho)$ such that $\e^{M-p}\Symo^{1}\Syml^{0} \in W^\BDJ(\rho)$, and we obtain
    either $\e^M \Symo^{0}\Syml^{p-2} \in W^\BDJ(\rho)$
    or $\e^{M+1}\Symo^{p-2}\Syml^{k_1-3} \not\in W^\BDJ(\rho)$, both of which give contradictions.
  \end{enumerate}
  \item Suppose $\rho$ is algebraically of weight $((p, k_1+1), (1,-1))$.
  If $\rho$ is non-ordinarily geometrically modular of weight $((p, k_1+1), (1,-1))$,
  then together with $\rho$ being geometrically modular of weight $((p, k_1 -1), (1,0))$ and $p \nmid k_1 -1$, we obtain
  $\rho$ is geometrically modular of weight $((0, k_1),(1,0))$ by \Cref{lem: not-ord-not-min}, 
  which contradicts to \Cref{thm: wt-in-min-cone}. 
  If $\rho$ is ordinarily geometrically modular of weight $((p, k_1+1), (1,-1))$, 
  we compute possible $W^\BDJ(\rho)$ such that $\e^{M+1-p}\Symo^{p-2}\Syml^{k_1 -1}$ by \cite[\S8]{Dembele-Diamond-Roberts} and 
  get either $\e^M \Symo^{0}\Syml^{k_1-2} \in W^\BDJ(\rho)$ or $\e^{M+1}\Symo^{p-2}\Syml^{k_1-3}\not\in W^\BDJ(\rho)$, which gives contradictions.
\end{enumerate}
Therefore, we conclude that $\rho$ is algebraically modular of weight $((2,k_1), (0,0))$.
\end{proof}

\begin{corollary}\label{cor: one_VinW}
  Let $p > 3$ and $\rho$ be geometrically modular of weight $((1, k_1), (0,0))$, where $p \geq k_1 \geq 3$. 
  \begin{enumerate}
    \item If $p \geq k_1 \geq 4$, then $\rho$ is algebraically modular of weight $((p+1, k_1-1), (0,0))$.
    \item If $k_1 = 3$, then $V_{(p+1, k_1-1), (-1, -k_1)} \in W^\BDJ(\rho)$.
  \end{enumerate} 
\end{corollary}

\begin{proof}
  Let $\rho \simeq \rho_f$ for some normalised eigenform $f\in M_{(1,k_1), (0,0)}(U_1(\frakn); E)$ as in \Cref{lem: Theta(f)nonzero} so that $\Theta_1(f)\neq 0$.
  Consider the eigenform $\Ha_1 f$ and then we see that $\rho$ is geometrically modular of weight $((p+1, k_1-1), (0,0))$.
  By \Cref{thm: my-lift-goem-alg}, we have $\rho$ is algebraically modular of weight $((p+1, k_1-1), (0,0))$ if $p \geq k_1 \geq 4$. 

  Now let $k_1= 3$ and $M = -1 - 3p$.
  Consider $\Ha_0\Ha_1 f$ and $\Theta_1 f$,
  we have $\rho$ geometrically hence algebraically modular of weight $((p, p + 2), (0,0))$ and $((p+1, 4), (0, -1))$ by \Cref{thm: my-lift-goem-alg} and $p > 3$.
  Then $\rho$ is modular of weights $V := \e^M \Symo^{p-2}\Syml^{p}$ and $V' = \e^{M-p}\Symo^{p-1}\Syml^{2}$.
  We have Grothendieck group relation
  \begin{align*}
    \left[\Symo^{p-2}\Syml^{p}\right] = \left[\e^1\Symo^{p-3}\Syml^{0}\right] + \left[\Symo^{p-1}\Syml^{0}\right] + \left[\e^{p}\Symo^{p-2}\Syml^{p-2}\right].
  \end{align*}
  Since $V' \in W^\BDJ(\rho)$, we compute the possible $W^\BDJ(\rho)$ by \cite[\S8]{Dembele-Diamond-Roberts} and find that
  either $\Symo^{p-1}\Syml^{0} \in W^\BDJ(\rho)$ or $W^\BDJ(\rho)\cap \JH(V) = \varnothing$.
\end{proof}

\begin{theorem}\label{thm: geom-alg-small}
  Let $p > 3$ and
  $F$ be a real quadratic field in which $p$ is inert.
  Let $\rho: G_F \rightarrow \GL_2(\Fpbar)$ be an irreducible, continuous, totally odd representation.
  If $p = 5$, assume that the projective image of $\rho|_{G_{F(\zeta_p)}}$ is not isomorphic to $A_5$.
  Let $(k_0,k_1)\in \Ximinplus$ and $2\leq k_0, k_1\leq p$.
  If $\rho$ is geometrically modular of weight $((k_0,k_1), (0,0))$,
  then $\rho$ is algebraically modular of weight $((k_0,k_1), (0,0))$.
\end{theorem}

\begin{proof}
  Without loss of generality, we may assume $k_0 \leq k_1$.
  By \Cref{thm: geq2_VinW}, if $\rho$ is geometrically modular of weight $(k, (0,0))$, 
  then $V_{k, -k}\in W^\BDJ(\rho)$. 
  If $\rho|_{G_{F(\zeta_p)}}$ satisfies the condition in \Cref{thm: BDJ-conj}, then we obtain $\rho$ is algebraically modular of weight $(k, -k)$.
  Therefore, by our assumption, we only need to consider the case when $\rho|_{G_{F(\zeta_p)}}$ is reducible.

Suppose $\rho|_{G_{F(\zeta_p)}}$ is reducible.
Then there is a quadratic extension $F'/F$ in $F(\zeta_p)$ and a continuous character $\xi: G_{F'} \rightarrow \Fpbar^\times$
such that
$\rho \simeq \Ind^F_{F'} \xi$.
Since $F$ is a real quadratic field in which $p$ is inert, we have $F_p \cong \Q_{p^2}$.
Note that $F'/F$ is totally ramified at $p$ 
(in fact $F' = F(\sqrt{p})$ if $p \equiv 1$ mod $4$ and 
$F' = F(\sqrt{-p})$ if $p \equiv 3$ mod $4$).
Let $p = v^2$ in $F'$ for a prime $v$ in $F'$ and denote $K' := F'_v$.
Let $\varepsilon_{K'}$ denote the fundamental character on $I_{K'}$ induced by the $\mathrm{Art}^{-1}$.
We also let $\varepsilon_2: I_{F_p} \cong I_{\Q_{p^2}} \xrightarrow{\Art^{-1}} \Z_{p^2}^\times \rightarrow \F_{p^2}^\times$ be the fundamental character of level $2$ and $\varepsilon_4: I_{F_p} \cong I_{\Q_{p^4}} \xrightarrow{\Art^{-1}} \Z_{p^4}^\times \rightarrow \F_{p^4}^\times$ be the fundamental character of level $4$.
Note that $\varepsilon_4^{p^2+1} = \varepsilon_2$ and we also have the commutative diagram
\[\begin{tikzcd}
	{\varepsilon_{K'}:\ I_{K'}} & {\CO_{K'}^\times} & {\F_{p^2}^\times} \\
	{\varepsilon_2:\ I_{F_p}} & {\Z_{p^2}^\times} & {\F_{p^2}^\times}.
	\arrow[from=1-1, to=1-2]
	\arrow[shift left=4, hook, from=1-1, to=2-1]
	\arrow[from=1-2, to=1-3]
	\arrow["\Nm", from=1-2, to=2-2]
	\arrow["{(\cdot)^2}", from=1-3, to=2-3]
	\arrow[from=2-1, to=2-2]
	\arrow[from=2-2, to=2-3]
\end{tikzcd}\]
We have $\xi|_{I_{K'}} = \varepsilon_{K'}^n$ for some $n\in \Z_{\geq 0}$.
Depending on the parity of $n$, we get the form of $\rho|_{I_{F_p}}$:
\begin{align}\label{eq: rhoI_K-v3}
  \rho|_{I_{F_p}} \simeq 
  \left\{ \begin{matrix}
    &\varepsilon_2^{n/2} \oplus \varepsilon_2^{n/2+ (p^2-1)/2}, \text{ if } n \text{ is even} \\
    &\varepsilon_4^{a} \oplus \varepsilon_4^{p^2a}, \text{ if } n \text{ is odd}
  \end{matrix}\right.
\end{align}
where $a = n(p^2+1)/2$.
On the other hand, $V_{k, -k}\in W^\BDJ(\rho)$, 
  so $\rho|_{I_{F_p}}$ has an explicit form described in \eqref{eq: reducible-case} and \eqref{eq: irreducible-case}, 
  \begin{align}\label{eq: rhoI_K-v4}
    \begin{aligned}
      &\rho|_{I_{F_p}} \simeq \varepsilon_{2}^{-k_0 - pk_1}\left(\begin{matrix}
        \prod_{\tau_i\in J}\varepsilon_{2}^{p^{i}(k_i-1)} & * \\ 0 & \prod_{\tau_i\not\in J}\varepsilon_{2}^{p^{i}(k_i-1)}
      \end{matrix}\right) \text{ reducible case}
      \\
      &\rho|_{I_{F_p}} \simeq \varepsilon_{2}^{-k_0 - pk_1}\left(\begin{matrix}
        \prod_{\tau'_j\in J'} \varepsilon_4^{p^{\pi(j)}(k_{\pi(j)}-1)} & 0 \\ 0 &  \prod_{\tau'_j \not\in J'} \varepsilon_4^{p^{\pi(j)}(k_{\pi(j)}-1)}
      \end{matrix}\right) \text{ irreducible case}
    \end{aligned}
  \end{align}
  for some $J\subseteq \Sigma = \{\tau_0, \tau_1\}$ and $J'\subseteq \Sigma' = \{\tau_0', \tau_1', \tau_2', \tau_3'\}$,
  where $\Sigma'$ is the set of embeddings of a inert quadratic extension of $F_p$ into $\Qpbar$(which can be identified with the set of embedding of $\F_{p^4}$ into $\Fpbar$),
  $\tau'_{j+1} = \Fr \circ \tau'_{j}$, 
  $J'$ containing
  exactly one element extending each element of $\Sigma$,
  and $\pi(j) := i$ if $\tau'_j|_{F_p} \cong \tau_i$ for $i = 0,1$ and $j = 0, 1, 2, 3$.

  \begin{align}
    \rho|_{I_{F_v}} \simeq \prod_{\tau\in\Sigma'_v} \varepsilon_\tau^{l_\tau}
    \left(\begin{matrix}
      \prod_{\tau\in J} \varepsilon_\tau^{k_\tau-1} & 0 \\
      0 & \prod_{\tau \not\in J'} \varepsilon_\tau^{k_\tau-1}
    \end{matrix}\right).
  \end{align}
  Comparing the form of $\rho|_{I_K}$ in \eqref{eq: rhoI_K-v3} and \eqref{eq: rhoI_K-v4}, we obtain:
  \begin{enumerate}
    \item If $n$ is even, then $\rho|_{G_{F_p}}$ is reducible, and 
    \begin{enumerate}
      \item when $J = \Sigma$ or $\varnothing$, then $k_0 = k_1 = (p+1)/2$, and $n \equiv 0$ mod $p^2-1$;
      \item otherwise, $k_0 = k_1 = (p+3)/2$, and $n \equiv p+1$ mod $p^2-1$.
    \end{enumerate}

    \item If $n$ is odd, then $\rho|_{G_{F_p}}$ is irreducible,
    $k_0 = (p+1)/2$, $k_1 = (p+3)/2$, and $n \equiv p$ mod $p^2-1$.
  \end{enumerate}
  When $p>3$, the weights $(k_0,k_1)$ above satisfy \eqref{eq: liftcone}, hence by \Cref{thm: my-lift-goem-alg} we have the algebraic modularity of weight $((k_0,k_1), (0,0))$.
\end{proof}

\begin{corollary}\label{cor: geom-alg-partialone}
  Let $p > 3$ and
  $F$ be a real quadratic field in which $p$ is inert.
  If $p = 5$, assume that the projective image of $\rho|_{G_{F(\zeta_p)}}$ is not isomorphic to $A_5$.
  If $\rho$ is geometrically modular of weight $((1, k_1), (0,0))$, 
  where $3 \leq k_1 \leq p$,
  then $\rho$ is algebraically modular of weight $((p+1, k_1-1), (0,0))$.
\end{corollary}

\begin{proof}
  By \Cref{cor: one_VinW}, we only need to prove the case when $k_1 = 3$, in which case we have $V_{(p+1, k_1-1), (-1, -k_1)} \in W^\BDJ(\rho)$.
  If $\rho|_{G_{F(\zeta_p)}}$ satisfies the condition in \Cref{thm: BDJ-conj}, then we obtain $\rho$ is algebraically modular of weight $((p+1, 2), (0,0))$.
  Therefore, by our assumption, we only need to consider the case when $\rho|_{G_{F(\zeta_p)}}$ is reducible using the same method as in the proof of \Cref{thm: geom-alg-small}.
  In fact, we prove the following statement in \Cref{thm: geom-alg-small}: 
  if $\rho$ is the induction of a character of a ramified quadratic extension and $V_{k, l}\in W^\BDJ(\rho)$,
  then $k$ has to be one of the following weights, $((p+1)/2, (p+1)/2)$, $((p+3)/2, (p+3)/2)$, $((p+1)/2, (p+3)/2)$, or $((p+3)/2, (p+1)/2)$.
  None of the above weights can be equal to $(p+1, 2)$.
  Therefore, we conclude $\rho$ is algebraically modular of weight $((p+1, k_1-1), (0,0))$.
\end{proof}

\begin{remark}
  The weight shifting method doesn't work for weight $k = (2, p+1)$.
  Suppose $\rho$ geometrically modular of weight $((2, p+1), (0,0))$.
  We use $\Ha_1$ and $\Theta_1$ to obtain $\rho$ algebraically modular of weights 
  $((p+2, p),(0,0))$ and $((p+2, p+2), (0,-1))$.
  Let $M = -3-p$.
  We have
  $$
  \left[\Symo^{p}\Syml^{p-2}\right] = \left[\Symo^0\Syml^{p-1}\right]  + \left[\e^p \Symo^0\Syml^{p-3}\right] + \left[\e^1 \Symo^{p-2}\Syml^{p-2}\right],
  $$
\begin{align*}
    \left[\e^{-p}\Symo^p\Syml^{p}\right] = \left[\Symo^{0}\Syml^{p-1}\right] + \left[\e^{-p}\Symo^1\Syml^1 + {\e^p\Symo^0\Syml^{p-3}}\right] \\
    +   \left[\e^{1-p}\Symo^{p-1}\Syml^0\right] + \left[\e^{2-p}\Symo^{p-3}\Syml^{0}\right] + \left[\e^1\Symo^{p-2}\Syml^{p-2}\right].
\end{align*}
We can't rule out the case that 
$W^\BDJ(\rho) = \{\e^{M-p} \Symo^{p-2}\Syml^{p-2} \}$.
We will deal with this weight using another method in \S\ref{subsub: big_wt}.
\end{remark}

\subsubsection{Big weights}\label{subsub: big_wt}
Suppose $\rho \simeq \rho_f$ is geometrically modular of weight $(k,l)$,
for some cuspidal eigenform $f \in S_{k,\mathbf{0}}(U; \Fpbar)$ and sufficiently small $U$.
We restrict $f$ to a stratum of the Hilbert modular variety to obtain an eigenform on the quaternionic Shimura variety of certain weight whose associated Galois representation is isomorphic to $\rho$. 

Let $(k, l) = ((k_0, k_1), (l_0, l_1))\in \Z^\Sigma_{\geq 2} \times \Z^\Sigma$.
Without loss of generality, we may assume $k_1 \geq k_0 \geq 2$.
Let $h_0 = (-1, p)$ be the weight of $\Ha_0$.
Let $\overline{Y} := Y_{U, \Fpbar}$ and 
$\overline{A}$ be the universal abelian scheme over $\overline{Y}$.
Denote $\overline{\omega}^k = \CA^{k,\mathbf{0}}_{\Fpbar}$
and $\overline{\delta}^l = \CA^{\mathbf{0}, l}_{\Fpbar}$.
Let $Z_0$ be the vanishing locus of $\Ha_0$.
Consider the short exact sequence on $\overline{Y}$:
\begin{align*}
  0 \rightarrow \overline{\omega}^{k - h_0} \rightarrow \overline{\omega}^{k} \rightarrow \CF \rightarrow 0,
\end{align*}
where the first map is multiplication by partial Hasse invariant $\Ha_0$
and $\CF$ is supported on $Z_0$.
We have a long exact sequence
\begin{align*}
  0 \rightarrow H^0(\overline{Y}, \overline{\omega}^{k - h_0}) \rightarrow H^0(\overline{Y}, \overline{\omega}^{k}) \rightarrow H^0(Z_0, \omega^{k}|_{Z_0})
  \rightarrow H^1(\overline{Y}, \overline{\omega}^{k - h_0}) \rightarrow \cdots
\end{align*}
We also note that $Z_0$ is disjoint from the cusps, beacuse the $q$-expansion of $\Ha_0$ at any cusp is the constant $1$ (\emph{cf.}\,\cite[\S9.3]{Diamond-Sasaki}). 
If a cuspidal eigenform $f\in H^0(Y_U, \omega^{k})$ is not in the image of $H^0(\overline{Y}, \overline{\omega}^{k - h_0})$, 
then $ f|_{Z_0}\in H^0(Z_0, \overline{\omega}^{k}|_{Z_0})$ is non-trivial. 

On $Z_0$, we have a short exact sequence
\begin{align*}
  0 \rightarrow \overline{\omega}_{0} \rightarrow H^1_{\dR}(\overline{A}/\overline{Y})_{\tau_0} \rightarrow \overline{\omega}_1^{\otimes p} \rightarrow 0.  
\end{align*}
Hence we have
$$
\overline{\omega}_0|_{Z_0} \cong 
\wedge^2 H^1_{\dR}(\overline{A}/\overline{Y})_{\tau_0} \otimes \overline{\omega}_1^{\otimes -p}|_{Z_0}
= \overline{\delta}_0\otimes \overline{\omega}_1^{\otimes -p}|_{Z_0}.
$$
Since $\overline{\delta}_0 = \overline{\delta}_1^p$, we get
$H^0(Z_0, \overline{\omega}^{k}|_{Z_0})
= H^0(Z_0, \overline{\delta}_1^{\otimes  p k_0} \overline{\omega}_1^{\otimes -pk_0 + k_1}|_{Z_0})  $.

Recall the result on the stratification \Cref{thm: strata-HMV} that 
we have a Hecke-equivariant isomorphism 
$\varphi : Z_0 \rightarrow \P^1_{Y^B}$
which identifies $Z_0$ with the projectivization of the locally free rank two vector bundle
$\CV^B_{\tau_1}$ over $Y^B$,
where 
$B$ is a definite quaternion algebra over $F$, 
and $Y^B: = Y^B_{U^B, \Fpbar}$ is the associated Shimura variety over $\Fpbar$ of some level $U^B$ away from $p$.
Under such isomorphism, we identify 
$\varphi_* H^1_{\dR}(\overline{A}/\overline{Y})_{\tau_1}|_{Z_0}$ with $\pi^*\CV^B_{\tau_1}$,
where $\pi: \P^1_{Y^B} = \P_{Y^B}(\CV_{\tau_1})\rightarrow Y^B$ is the natural projection.
We also identify
$\varphi_*\overline{\omega}_{1}$ with the tautological line bundle $\CJ$ on $\P^1_{Y^B}$.
In particular, $\CJ\cong \wedge^2 \CV^B_{\tau_1} (-1)$ because of the following exact sequence
\begin{align*}
  0 \rightarrow \CJ \rightarrow \CV^B_{\tau_1} \rightarrow \CO(1) \rightarrow 0.
\end{align*}
Hence 
$\varphi_*(\overline{\omega}_1|_{Z_0}) 
\cong 
\wedge^2 \CV^B_{\tau_1} (-1)$. Considering the pushforward of $\varphi$, we have
\begin{align*}
  H^0(Z_0, \overline{\omega}^{k}) 
  &= H^0(\P^1_{Y^B}, (\wedge^2 \CV^B_{\tau_1})^{\otimes pk_0} \otimes (\wedge^2 \CV^B_{\tau_1} (-1))^{\otimes -pk_0+k_1} ) \\
  &= H^0(\mathbb{P}^1_{Y^B}, (\wedge^2 \CV^B_{\tau_1})^{\otimes k_1} \otimes \CO(pk_0 - k_1) ).
\end{align*}
The isomorphism
$\pi_*\CO(n) \cong \Sym^n_{\CO_{Y^B}}\CV_{\tau_1}$ is Hecke-equivariant in the sense that the diagram \eqref{eq: com-diag} below commutes.
More precisely, for sufficiently small open compact subgroups $U_1,U_2$ and $g\in \GL_2(\AF)$ such that $gU_1 g^{-1}\subset U_2$,
we let $\CV_{\tau_1, i}$ be the locally free sheaves over $Y^B_{U_i}$ and 
$\P^1_{Y^B_{U_i}} = \P_{Y^B_{U_i}}(\CV_{\tau_1, i})$.
Let $\rho_g: Y^B_{U_1} \rightarrow Y^B_{U_2}$, 
then $\rho_g^*(\CV_{\tau_1, 2}) \xrightarrow{\sim} \CV_{\tau_1, 1}$.
We denote $\pi_i : \P^1_{Y^B_{U_i}} \rightarrow Y^B_{U_i}$ the natural projection for $i = 1,2$ and 
$\tilde{\rho}_g: \P^1_{Y^B_{U_1}} \rightarrow \P^1_{Y^B_{U_2}}$ the induced morphism such that following diagram commutes:
\[\begin{tikzcd}
	{\mathbb{P}^1_{Y^B_{U_1}}} & {\mathbb{P}^1_{Y^B_{U_2}}} \\
	{Y^B_{U_1}} & {Y^B_{U_2}}
	\arrow["{\tilde{\rho}_g}", from=1-1, to=1-2]
	\arrow["{\pi_1}"', from=1-1, to=2-1]
	\arrow["{\pi_2}", from=1-2, to=2-2]
	\arrow["{\rho_g}"', from=2-1, to=2-2]
\end{tikzcd}\]
We use $\alpha_i$ to denote the isomorphism $\alpha_i: (\pi_i)_*\CO(n)\xrightarrow{\sim} \Sym^n\CV_{\tau_1, i}$ for
$i = 1,2$.
Then the commutativity of the following diagram can be checked locally:
\begin{equation}\label{eq: com-diag}
  \begin{tikzcd}
    {} & {\rho_g^*((\pi_2)_*\CO(n))} & {\rho_g^*(\Sym^n \CV_{\tau_1, 2})} \\
    & {(\pi_1)_*(\tilde{\rho}_g^*\CO(n))} & {\Sym^n (\rho_g^*\CV_{\tau_1, 2})} \\
    & {(\pi_1)_*\CO(n)} & {\Sym^n \CV_{\tau_1, 1}}
    \arrow["{\rho_g^*(\alpha_2)}", from=1-2, to=1-3]
    \arrow["\simeq"', from=1-2, to=2-2]
    \arrow["\simeq", from=1-3, to=2-3]
    \arrow["\simeq"', from=2-2, to=3-2]
    \arrow["\simeq", from=2-3, to=3-3]
    \arrow["{\alpha_1}"', from=3-2, to=3-3]
  \end{tikzcd}.
\end{equation}
Hence, we have
\begin{align*}
  H^0(Z_0, \overline{\omega}^{k}) =
  H^0(Y^B, (\wedge^2 \CV^B_{\tau_1})^{\otimes k_1} \otimes \Sym^{pk_0 - k_1} \CV^B_{\tau_1}),
\end{align*}
which is Hecke-equivariant.
Hence we have
$$H^0(Y^B, (\wedge^2 \CV_{\tau_1})^{\otimes k_1} \otimes \Sym^{pk_0 - k_1} \CV_{\tau_1})[\mathfrak{m}_\rho] \neq 0. $$

Let $Y^B_{U^B, \CO}$ be the Shimura variety associated to $B$ of level $U^B$ over $\CO$, then it is zero-dimensional.
Note we have
$Y^B_{U^B, \CO}(\C) = B^\times_+ \backslash (\frakH \times B_{\bff}^\times)/U$
and an one-to-one bijection between $Y^B_{U^B, \CO}(\C)$ and $Y^B_{U^B, \CO}(\Fpbar)$.
Recall the definition
$M^B(U, Q) = \{
f: B_{\A}^\times \rightarrow Q \mid f(b x u) = u_p^{-1} f(x)\text{ for } b \in B^\times,  x \in B_\A^\times,  u\in B_\infty^\times U\}$.
Also, recall that $\CV^B_{\tau_1}$ is defined by 
the pullback of the vector bundle $\CV_{\tilde{\tau}_1}^\circ$ over 
a unitary Shimura variety $Y'_{U^B}(G'_{\{\tau_0, \tau_1\}})_{\Fpbar}$,
where $\CV_{\tilde{\tau}_1}^\circ$ is the descent of $\CH_{\dR}(A'/S')_{\tau_1}^\circ$ for the universal abelian variety $A'$ over $S' = \widetilde{Y}'_{U^B}(G'_{\{\tau_0, \tau_1\}})_{\Fpbar}$.
We have the action of 
$B_\bff^{(p),\times}$ on $\CV^B_{\tau_1}$ induced by 
$G'_{\{\tau_0, \tau_1\}}(\A_{\bff})$ on 
$\CH_{\dR}(A'/S')_{\tau_1}^\circ$.
Therefore, 
we have a $B_\bff^{(p),\times}$-equivariant isomorphism 
$$
\varinjlim_{U} M^B(U, \det_{[1]}^{m}\otimes\Syml^{n}\Fpbar^2) = 
\varinjlim_{U} H^0(Y^B_{U, \Fpbar}, (\wedge^2\CV^B_{\tau_1})^{\otimes m} \otimes \Sym^{n}\CV^B_{\tau_1})
$$
for open compact subgroups $U \subset B_\bff^\times$ such that $U_p = \CO_{B,p}^\times$.
Hence we have
$H^0(Y^B, (\wedge^2 \CV_{\tau_1})^{\otimes k_1} \otimes \Sym^{pk_0 - k_1} \CV_{\tau_1})[\mathfrak{m}_\rho] \neq 0$
is equivalent to 
$$ M^B(U^B, V_{(2, pk_0-k_1 + 2), (0, k_1)})[\mathfrak{m_\rho}] \neq 0,$$
which is 
\begin{align}\label{eq: MB_neq0}
  M^B_{(2, pk_0-k_1 + 2),(-1, k_1 -1)}(U^B; \Fpbar)[\mathfrak{m_\rho}] \neq 0
\end{align}
by the definition
$M^B_{k,l}(U; \Fpbar) = M^B(U, V_{k,l+1})$.

\begin{lemma}\label{lem: Groth_relation}
  Suppose $pk_0 \geq k_1 \geq p+1$. Then
  \begin{align}\label{eq: Groth_relation}
    \left[{\e}^{pk_1 - p- 1} \Syml^{pk_0 - k_1}\right] \leq \left[\Symo^{k_0-2}\Syml^{k_1 - 2}\right].
  \end{align} 
\end{lemma}

\begin{proof}
  Let $pk_0 - k_1 = tp + r$, where
  $t, r\in \Z_{\geq 0}$ and $r \leq p-1$.
  Since $k_1 \geq p+1$, $k_0-t-2 \geq 0$, we have the following,
  \begin{align}
    \begin{aligned}\label{eq: Groth_ineq}
      &\left[\Symo^{k_0-2}\Syml^{k_1-2}\right] \geq 
      \left[{\e}^{k_0-t-2}\Symo^{t}\Syml^{2p - r -2}\right] \\
       \geq &\left[{\e}^{k_0-t-1}\Symo^{t-1}\Syml^{p - r -2}\right] + \left[{\e}^{pk_1-p-1}\Symo^{t}\Syml^{r}\right].
    \end{aligned}
  \end{align}
  The first inequality uses \cite[Lemma 2.3]{Diamond-Reduzzi}, 
  and the second inequality is by a periodic relation of the Grothendieck group \cite[(1.0.3)]{Reduzzi15}.
  Using an induction argument on $t$, we can show for any $t \in \Z_{\geq 0}$ and $r\in \Z$, we have
  \begin{equation}
    \begin{aligned}\label{eq: Grothendieck-eq}
      \left[\Symo^0\Syml^{pt+r}\right] 
      = \left[\Symo^{t}\Syml^{r}\right] + \left[{\e}^{pr + p}\Symo^{t-1}\Syml^{p - r -2}\right] .
    \end{aligned}
  \end{equation}
If $t = 0$, then \eqref{eq: Grothendieck-eq} is true for any $r$ by the convention $\left[\Sym_{[0]}^{-1}\right] = 0$.
Now let $n \geq 1$, and assume \eqref{eq: Grothendieck-eq} is true for any $t \leq n$ and $r\geq 0$. Then, by the induction assumption and the Grothendieck relation \cite[(1.0.3)]{Reduzzi15}, we have
\begin{align*}
  & \left[\Symo^0\Syml^{p(n+1)+r}\right] = 
  \left[\Symo^0\Syml^{pn + p +r}\right] \\
  = & \left[\Symo^{n}\Syml^{r + p}\right] + \left[{\e}^{pr + 1 + p}\Symo^{n-1}\Syml^{ - r -2}\right] \\
  = & \left[\Symo^{n+1}\Syml^{r}\right] + \left[{\e}^{pr + p}\Symo^{n}\Syml^{p - r -2}\right].
\end{align*}
Therefore, we prove \eqref{eq: Grothendieck-eq}. Together with \eqref{eq: Groth_ineq}, we obtain \eqref{eq: Groth_relation}.
\end{proof}

\begin{proposition}\label{prop: geom-alg-big}
  Let $(k_0,k_1)\in \Ximinplus$ and $pk_0 \geq k_1 \geq p+1$.
  If $\rho \simeq \rho_f$, for some $f\in M_{(k_0, k_1), (0,0)}(U_1(\frakn), \Fpbar)$, where $\frakn$ is prime to $p$ and 
  $f$ is not divisible by $\Ha_0$,
  then $\rho$ is algebraically modular of weight $((k_0, k_1), (0,0))$.
\end{proposition}

\begin{proof}
  By \eqref{eq: MB_neq0} and \Cref{lem: Groth_relation}, we have
  \begin{align*}
    M^B_{(k_0, k_1),(0, 0)}(U^B, \Fpbar)[\mathfrak{m_\rho}] \neq 0
  \end{align*}
  for some compact open subgroup $U^B$.
  By definition, $\rho$ is algebraically modular of weight $((k_0, k_1), (0,0))$.
\end{proof}

\begin{theorem} \label{thm: geom-alg-allwt}
  Let $p > 3$ and
  $F$ be a real quadratic field in which $p$ is inert.
  Let $(k_0,k_1) \in \Ximinplus \cap \Z_{\geq 2}^2$ and $(l_0, l_1) \in \Z^2$.
  Let $\rho : G_F \rightarrow \GL_2(\Fpbar)$ be an irreducible, continuous, totally odd representation.
  If $p = 5$, assume that the projective image of $\rho|_{G_{F(\zeta_p)}}$ is not isomorphic to $A_5$.
  If $\rho$ is geometrically modular of weight $((k_0,k_1), (l_0, l_1))$, then $\rho$ is algebraically modular of weight $((k_0,k_1), (l_0, l_1))$.
\end{theorem}

\begin{proof}
  By \Cref{lem: reduce to l=0}, we can assume $(l_0 , l_1) = (0,0)$.
  Let $\rho \simeq \rho_f$ for a normalised (cuspidal) eigenform $f\in M_{(k_0,k_1), (0,0)}(U_1(\frakn), \Fpbar)$.
  Let $(\kappa_0, \kappa_1) = \Phi(f)$ defined as in \Cref{def: min_wt_f}.
  Then $(\kappa_0, \kappa_1) \leq_{\Ha} (k_0, k_1)$ and $(\kappa_0, \kappa_1) \in \Ximinplus \subset \Z_{\geq 1}^2$.
  Also, by the definition of $\Phi(f)$, we have $\rho$ is geometrically modular of weight $((\kappa_0, \kappa_1), (0,0))$.
  Altering $\tau_0$ and $\tau_1$ if necessary, we can assume $\kappa_0 \leq \kappa_1$.
  \begin{enumerate}
    \item Suppose $(\kappa_0, \kappa_1) \in \Z^2_{\geq 2}$. Then by \Cref{thm: geom-alg-small} and \Cref{prop: geom-alg-big}, $\rho$ is algebraically modular of weight $(k', (0,0))$, where $k' = (\kappa_0, \kappa_1)$.
    \item Suppose $\kappa_0 = 1$ and $\kappa_1 \geq 3$. Then by \Cref{cor: geom-alg-partialone}, $\rho$ is algebraically modular of weight $(k', (0,0))$, where $k' = (\kappa_0 + p, \kappa_1 - 1)$. Also, because $(k_0, k_1)\in \Z^\Sigma_{\geq 2} \geq_{\Ha} (\kappa_0, \kappa_1)$, we know $(k_0, k_1) \geq_{\Ha} k'$. 
    \item Suppose $\kappa_0 = 1$ and $\kappa_1 \leq 2$. Then by \Cref{thm: my-lift-goem-alg}, $\rho$ is algebraically modular of weight $(k', (0,0))$, where $k' = (\kappa_0 + p - 1, \kappa_1 + p - 1)$. Also, because $(k_0, k_1)\in \Z^\Sigma_{\geq 2} \geq_{\Ha} (\kappa_0, \kappa_1)$, we know $(k_0, k_1) \geq_{\Ha} k'$. 
  \end{enumerate}
  In all the cases $k' \leq_{\Ha} (k_0, k_1)$.
  If $k'\neq (2,2)$, then \Cref{lem: Groth_Ha} shows that $\rho$ being algebraically modular of weight $(k',(0,0))$ implies that $\rho$ is algebraically modular of weight $(k,(0,0))$. 
  Suppose $k' = (2,2)$, then either we have $k = (2,2)$, in which case the result follows from \Cref{thm: geom-alg-small}, or we have $k \geq_{\Ha} (p+1, p+1)$, in which case $\rho$ is algebraically modular of weight $((p+1, p+1), (0,0))$ by applying weight shifting method on $(k', (0,0))$ and the desired result follows from \Cref{lem: Groth_Ha}.
\end{proof}

\begin{lemma}\label{lem: Groth_Ha}
  Let $k, k' \in \Z_{\geq 2}^\Sigma \cap \Ximinplus$ and $k' \leq_{\Ha} k$.
  Suppose $k' \neq (2,2)$.
  Then $[V_{k', \mathbf{0}}] \leq [V_{k,\mathbf{0}}]$.
\end{lemma}

\begin{proof}
  Let $C_0 = \{k\in \Ximinplus \mid p(k_0 - 2) > k_1 - 2 \geq 0 \}$ and $C_1 = \{k\in \Ximinplus \mid p(k_1 - 2) > k_0 - 2 \geq 0 \}$.
  Note that $C_0 \cup C_1  = (\Z^\Sigma_{\ge 2} \cap \Ximinplus) - \{(2,2)\}$.
  Let $h_i = k_{\Ha_i}$ be the weight of partial Hasse invariant $\Ha_i$ for $i = 1,0$.
  If $\kappa \in C_i$, 
  then $[V_{\kappa, \mathbf{0}}] \leq [V_{\kappa + h_i, \mathbf{0}}]$ by \cite[Lemma 2.5]{Diamond-Reduzzi}.

  Let $k = k' + t_0 h_0 + t_1 h_1$, where $t_0, t_1 \geq 0$.
  We prove the lemma by induction on $s = t_0 + t_1 \in \Z_{\geq 0}$.
  If $s = 0$, then the result follows directly.
  Assume for $s< N$ we have the result.
  Now let $s = N (>0)$.
  Without loss of generality, suppose $t_0 > 0$.
  If $k' \in C_0$, then $[V_{k', \mathbf{0}}] \leq [V_{k' + h_0, \mathbf{0}}]$ by \cite[Lemma 2.5]{Diamond-Reduzzi}, and $[V_{k' + h_0, \mathbf{0}}] \leq [V_{k, \mathbf{0}}]$ by the induction assumption, therefore $[V_{k', \mathbf{0}}] \leq [V_{k, \mathbf{0}}]$.
  If $k' \not\in C_0$, then $k'\in C_1$ and $t_1 >0$ (otherwise $k = k' + t_0 h_0 \not\in \Ximinplus$),
  hence we have 
  $[V_{k', \mathbf{0}}] \leq [V_{k' + h_1, \mathbf{0}}] \leq [V_{k, \mathbf{0}}]$ by \cite[Lemma 2.5]{Diamond-Reduzzi} and the induction assumption.
\end{proof}

\begin{corollary}
  Let $p > 3$.
  If $p = 5$, assume that the projective image of $\rho|_{G_{F(\zeta_p)}}$ is not isomorphic to $A_5$.
  If $\rho$ is geometrically modular of weight $((k_0,k_1), (0,0))$,
  where $(k_0, k_1) \in \Ximinplus$, 
  then $\rho$ is algebraically modular of weight $((k_0',k_1'), (0,0))$, where
  \begin{itemize}
    \item $(k_0',k_1') = (k_0, k_1)$, if $ k_0, k_1 \geq 2$;
    \item $(k_0',k_1') = (p+1, k_1 - 1)$, if $k_0 = 1, k_1 \geq 3$; 
    \item $(k_0',k_1') = (k_0 - 1 , p+1)$, if $k_1 = 1, k_0 \geq 3$;
    \item $(k_0',k_1') = (k_0 + p-1, k_1 +p-1)$, if $1 \leq k_0, k_1 \leq 2$. 
  \end{itemize}
\end{corollary}

\begin{proof}
  This is explained in the proof of \Cref{thm: geom-alg-allwt}.
\end{proof}

As an application, we can use our result on the relation between algebraic and geometric modularity to strenghten 
\cite[Theorem 11.4.1]{Diamond-Sasaki}
for $p > 3$ and
under mild hypotheses.

\begin{corollary}
  Let $p > 3$.
  Suppose that $2 \leq k_1 \leq p$ and that $\rho : G_F \rightarrow \GL_2(\Fpbar)$ is irreducible and modular. 
  Assume that $\rho|_{G_{F(\zeta_p)}}$ is irreducible, and  
  when $p = 5$, further assume that the projective image of $\rho|_{G_{F(\zeta_5)}}$ in $\mathrm{PGL}(\overline{\F}_5)$ is not isomorphic to $A_5$.
  Then $\rho$ is geometrically modular of weight $((1, k_1), (0, 0))$ if and only if
  $\rho|_{G_{F_p}}$ has a crystalline lift of weight $((1, k_1), (0, 0))$.  
\end{corollary}

\begin{proof}
  We follow the proof of \cite[Theorem 11.4.1]{Diamond-Sasaki} and replace \cite[Conjecture 7.5.2]{Diamond-Sasaki} in their proof with the relation between algebraic and geometric modularity proved in \cite{Diamond-Sasaki_inprep} and in this paper (\Cref{thm: geom-alg-allwt}).

  Suppose first that $\rho|_{G_{F_p}}$ has a crystalline lift of weight $((1, k_1), (0, 0))$.
  Then \cite[Lemma 11.2.6]{Diamond-Sasaki} implies that 
  $\rho|_{G_{F_p}}$ has crystalline lifts of weight 
  $((p + 1, k_1 - 1), (0, 0))$ (resp.\,$((p , p+1), (0, 0))$) 
  if $k_1 > 2$ (resp.\,$k_1 = 2$) and $((p+ 1, k_1 + 1), (0, -1))$, and that $\rho|_{G_{F_p}}$ has no subrepresentation on which $I_{F_p}$ acts as $\varepsilon_{\tau_1}$.
  By \Cref{thm: BDJ-conj}, we get
  $\rho$ is algebraically modular of the above weights,
  and hence geometrically modular of these weights by \cite{Diamond-Sasaki_inprep}.

  Now suppose that $\rho$ is geometrically modular of weight $((1, k_1), (0, 0))$.
  By considering $\Ha_1$ (resp.\,$\Ha_0\Ha_1$) if $k_0 > 2$ (resp.\,$k_1 = 2$) and $\Theta_1$, 
  we obtain $\rho$ is geometrically modular of weight 
  $((p + 1, k_1 - 1), (0, 0))$ (resp.\,$((p , p+1), (0, 0))$) 
  if $k_1 > 2$ (resp.\,$k_1 = 2$) and $((p+ 1, k_1 + 1), (0, -1))$.
  Then $\rho$ is algebraically modular of above weights by \Cref{thm: geom-alg-allwt},
  and hence has crystalline lifts of these weights by \Cref{thm: BDJ-conj}.
  Then \cite[Lemma 11.2.6, Remark 11.2.7]{Diamond-Sasaki} shows
  $\rho|_{G_{F_p}}$ has a crystalline lift of weight $((1, k_1), (0, 0))$.
\end{proof}

\bibliographystyle{alpha}
\bibliography{mybib}

\end{document}